\documentclass[a4paper]{article}
\usepackage[utf8]{inputenc}
\usepackage[left=2cm,right=2cm,top=2cm,bottom=2cm]{geometry}
\title{On the passage from nonlinear to linearized viscoelastodynamics}
\author{Barbora Bene\v{s}ov\'a, Malte Kampschulte, Martin Kru\v{z}\'ik}
\date{\today}
\usepackage{comment}
\usepackage{amsmath}
\usepackage{amsthm}
\usepackage{amssymb}
\usepackage{tikz-cd}
\usepackage{esint}
\usepackage{mathrsfs}
\usepackage{stackengine}
\usepackage{hyperref}

\allowdisplaybreaks

\theoremstyle{definition}
\newtheorem{definition}{Definition}[section]
\newtheoremstyle{remark2}{}{}{}{}{\bfseries}{.}{ }{}
\theoremstyle{remark2}

\newtheorem{remarkx}[definition]{Remark}
\newenvironment{remark}
  {\pushQED{\qed}\remarkx}
  {\popQED\endremarkx}

\theoremstyle{plain}
\newtheorem{lemma}[definition]{Lemma}
\newtheorem{theorem}[definition]{Theorem}
\newtheorem{corollary}[definition]{Corollary}
\newtheorem{proposition}[definition]{Proposition}

\newcommand{\R}{\mathbb{R}}
\newcommand{\Rsym}{\mathbb{R}_{\rm sym}}

\newcommand{\abs}[1]{\left|#1\right|}
\newcommand{\norm}[2][]{\left\|#2\right\|_{#1}}
\newcommand{\inner}[3][]{\left\langle #2, #3\right\rangle_{#1}}

\newcommand{\N}{\mathbb{N}}

\renewcommand{\R}{\mathbb{R}}
\newcommand{\C}{\mathbb{C}}

\newcommand{\stress}{{T^E}}
\newcommand{\vstress}{{S}}

\def\Id{\mathbf{Id}}
\def\id{\mathbf{id}}

\def\eps{\varepsilon}

\def\dist{\operatorname{dist}}

\def\XXint#1#2#3{{\setbox0=\hbox{$#1{#2#3}{\int}$}
     \vcenter{\hbox{$#2#3$}}\kern-.5\wd0}}

\newcommand{\E}{\mathcal{E}}
\newcommand{\Ry}{\mathcal{R}}
\begin{document}
 \maketitle
 
 \begin{abstract}

 The equations of linearized viscoelastodynamics in Kelvin-Voigt rheology are rigorously derived from a nonlinear model that satisfies the time-dependent frame indifference in the sense of Antman \cite{Antm98PUVS}. Besides showing the convergence of corresponding solutions of both systems, we also prove the convergence of time-discrete solutions on various scales and of continuous solutions of nonlinear problems to linearized ones.
 \end{abstract}

 \section{Introduction}

 Nonlinear viscoelastodynamics seeks to predict the deformation and internal stresses of a solid body under the action of applied forces  over a given time interval $[0,T]$, where $T>0$.   Let  $ d=2$ or $d=3$ and $\Omega\subset\R^d$ be a bounded Lipschitz domain  called a reference configuration of the solid body. Then, in the absence of non-mechanical effects, the state of the specimen is described by  a  \emph{deformation} $y: [0,T]\times \Omega \to \mathbb{R}^d$ such  that  $y(t,\cdot):\Omega\to\R^d$ is one-to-one and continuous. 
 Denoting by  $\varrho: \Omega \to \mathbb{R}$  the density of the body in the reference configuration, the deformation fulfills the following force balance. 
\begin{align}
 \varrho \partial_{tt} y - \mathrm{div}\, \mathcal{T} = \tilde{f}, 
 \label{basic-balLaw}
\end{align}
 in the time-space cylinder $ [0,T] \times \Omega$; here $\mathcal{T}$ is a stress tensor and $\tilde{f}$ is the volume density of applied body forces. In addition to the force balance \eqref{basic-balLaw}, we also prescribe initial and boundary conditions in the form
 \begin{align*}
y &= \id \,  \text{ in $ [0,T] \times  \partial \Omega$,}\\
 \partial_t y(0)&= v_0 \,  \text{ in $\Omega$,}\\
 y(0)&=y_0  \text{ in $\Omega$,}
 \end{align*}
 where $\partial \Omega$ is the boundary of $\Omega$ and $\id$ is the identity deformation. Therefore, in this work, we choose to load the sample only through body forces,  the initial deformation $y_0$,  and the initial velocity $v_0$.
  
The stress tensor $\mathcal{T}$ in \eqref{basic-balLaw} must be defined constitutively. We will  consider \emph{generalized standard materials} (see \cite{HalNgu75MSG}) for which it is assumed that the stress tensor can be determined by taking an appropriate functional derivative of two underlying functionals: The so-called elastic energy functional $y\mapsto \mathcal{E}(y)$ and the dissipation functional $y\mapsto \mathcal{R}(y, \partial_t y)$. Written abstractly, the stress tensor is assumed to be obtained via the additive rule
\begin{align}
-\mathrm{div}\mathcal{T} = D \mathcal{E} (y) + D_2 \mathcal{R} (y, \partial_t y),
\end{align}
which is also referred to as the \emph{Kelvin-Voigt rheology}.
We shall specify the precise form of $\mathcal{E}$ and $\mathcal{R}$ in Section \ref{sec:model}, but throughout this introduction, we aim to keep the presentation as simple as possible so as to highlight the advantages of the abstract formulation.

If the loading  is rather small and thus $y$ is expected to change slowly, the inertial term $\varrho \partial_{tt} y$ is often omitted in \eqref{basic-balLaw} and the following \emph{quasi-static} problem is considered instead:
\begin{equation}
 - \mathrm{div}\, \mathcal{T} = \tilde{f}.
 \label{basic-quasiStat}
\end{equation}
This setting has been widely studied, and many works on it have appeared in the literature; we refer, e.g.,  to \cite{kruvzik2019mathematical} for a review.
If only a static load is applied, then after some time — once the system reaches equilibrium — it is expected that $\partial_t y = 0$, and consequently $D_2 \mathcal{R}(y, 0) = 0$, since dissipative forces can arise only in the presence of motion.

 In this case, \eqref{basic-quasiStat} reduces even further to what is called the \emph{static} problem
\begin{equation}
 D \mathcal{E} (y) = \tilde{f},
 \label{basic-Stat}
\end{equation}
which can be understood as the (formal) Euler-Lagrange equation of the minimization problem
\begin{equation}
\mathrm{minimize } \ y\mapsto  \mathcal{E}(y)- \int_\Omega \tilde{f} y \mathrm{d}x 
\label{basic-Stable}
\end{equation}
The last minimization problem \eqref{basic-Stable} is probably the one that is studied most in nonlinear elasticity, starting with the pioneering work by Ball \cite{Ball:77} where a physically sound notion of polyconvexity was introduced, allowing for the application of the direct method of the calculus of variations (see e.g. \cite{Daco89DMCV}) to show the existence of a minimizer/solution to \eqref{basic-Stable}.

Problem \eqref{basic-Stable} has also been extensively studied in the past years in the context of model reduction, i.e. passing to a simpler model in solid mechanics. Prominent examples include linearization \cite{DalMasoNegriPercivale:02} when a model reduction is sought for very small strains; homogenization \cite{Tart09GTH} where we seek an effective  model for a material that contains fine microstructure; or dimension reduction \cite{FJM} where a model is found for a lower-dimensional specimen. The mathematical method used in these works is that of $\Gamma$-convergence (see Section \ref{Sec:Gamma_conv_energy} for a formal definition) which is a notion of convergence of functionals that is perfectly tailored to minimization problems. Thus, one might wonder whether these results have the potential to be extended to a situation beyond searching for stable states. This is indeed possible, we refer to \cite{MiRoSt08GLRR} for $\Gamma$-convergence methods in rate-independent systems or to  \cite{Mielke2016} for gradient systems.  In this work,  we extend this approach  to hyperbolic problems for viscoelastic solids.

In our approach, an important observation is that the problems \eqref{basic-balLaw}, \eqref{basic-quasiStat} and \eqref{basic-Stat} do not only correspond to different timescales, but can, in a sense, \emph{be seen as approximations of one another}. In order to explain this idea, we return to \eqref{basic-balLaw} and restate it as 
\begin{align}\label{eq:viscoel-nonsimple-recalled}
\varrho\partial_{tt} y + D_1 \mathcal{E}(y) + D_2 \mathcal{R}(y, \partial_t y) &=  \tilde{f} \, \text{ in $ [0,T] \times  \Omega$,} 
\end{align}
and we seek an \emph{approximation of \eqref{eq:viscoel-nonsimple-recalled} via time discretization}. To do so, we will follow the ideas of \cite{benesovaVariationalApproachHyperbolic2020} and introduce a   suitable two time-scale approximation. This will allow us to understand quasi-static problems as approximations of inertial ones and, in turn, static problems as approximations of quasi-static problems; the latter observation actually going back to De Giorgi \cite{Degi77GCGC}.

To explain the idea in more detail, we first introduce the so-called $h$-scale, or the so-called \emph{time-delayed} problem. In doing so, we replace the second partial derivative $\partial_{tt} y$ in \eqref{eq:viscoel-nonsimple-recalled} with a difference quotient as 
 \begin{align} \tag{NL-TD}\label{eq:introTimeDelayed}
\varrho \tfrac{\partial_t y(t) - \partial_t y(t-h)}{h}  + D_1 \mathcal{E}(y) + D_2 \mathcal{R}(y, \partial_t y) &=  \tilde{f} \quad \text{ on }  [0,T] \times \Omega
 \end{align}
 for some fixed $h$.

 We then solve this problem by splitting $[0,T]$ into individual subintervals $[nh,(n+1)h]$ of length $h$. It is important to realize that on each of these subintervals, the second half of the difference quotient, namely $-\frac{\partial_t y(t-h)}{h}$ depends completely on the solution on the previous subinterval. We can thus solve the problem iteratively on each interval and use the result as data for the next.
 
 Notice also that once this term is reinterpreted as given data, the equation takes the structure of a \emph{quasi-static} force balance as outlined in \eqref{basic-quasiStat}. In fact, understanding $\partial_t y(t-h)$ as being fixed and given, we can set $$ \mathcal{R}^{(h)}(y, \partial_t y) = \mathcal{R}(y, \partial_t y) + \frac{\rho h}{2}\int_\Omega \left| \frac{\partial_t y(t) - \partial_t y(t-h)}{h}\right|^2,
 $$
 which is now an \emph{altered} dissipation functional. Using this functional, the weak formulation of \eqref{eq:introTimeDelayed} can be abstractly written as 
 $$
 D \mathcal{E}(y) + D_2 \mathcal{R}^{(h)}(y,\partial_t y) = \tilde{f},
 $$
 the structure of which corresponds to \eqref{basic-quasiStat}. Note that $\mathcal{R}^{(h)}$ no longer satisfies the physical assumptions on a dissipation functional, in particular we have that $\mathcal{R}^{(h)}(y,0) \neq 0$ in general. This is not an issue for our purposes because it is more the \emph{structure} of the equation that is important. However, let us point out that there is an equivalent alternative interpretation by seeing $\partial_t y(t-h)$ as part of a force instead, which would not suffer from this problem.
 
 It has been shown in \cite{benesovaVariationalApproachHyperbolic2020} (and we also provide a proof in this paper) that as $h\to 0$ solutions of the time-delayed problem indeed converge (in an appropriate sense) to solutions of \eqref{eq:viscoel-nonsimple-recalled}. This shows that quasi-static problems indeed serve as an approximation of inertial problems.

 This viewpoint can be exploited even further once again. Introducing a \emph{second time-scale}, the so-called $\tau$-scale (with $\tau \ll h$), we may replace the partial derivatives $\partial_t y$ once again, now with a fully discrete difference quotient of the form $\frac{y_k-y_{k-1}}{\tau}$, where $y_k$ is thought to correspond to $y(k\tau)$. We thus approximate \eqref{eq:introTimeDelayed} with
  \begin{align} \tag{NL-disc}\label{eq:introdisceteEq}
\varrho \tfrac{\frac{y_k-y_{k-1}}{\tau} - w_k}{h} + D_1 \mathcal{E}(y_k) + D_2 \mathcal{R}\left(y_{k-1}, \tfrac{y_k-y_{k-1}}{\tau}\right) &=  \tilde{f}_k \quad \text{ in }  \Omega,
 \end{align}
 where $\tilde{f}_k$ is a discretization of $\tilde{f}$ and $w_k$ of $\partial_t y(t-h)$.

Notice that in introducing the second discretization, we used an implicit-explicit discretization in $\mathcal{R}$ whose first variable is explicitly discretized. This allows us to understand \eqref{eq:introdisceteEq}, as the Euler-Lagrange equation of a variational problem corresponding to an \emph{appropriately altered} energy functional, namely we find $y_k$ as minimizer of the following functional:
 \begin{align}
\mathcal{I}_\tau(y) := \mathcal{E}(y) + \tau \mathcal{R}\left( y_{k-1},\tfrac{ y- y_{k-1}}{\tau}\right) - \tau \int_\Omega \tilde{f}_{k}\cdot {\tfrac{y-y_{k-1}}{\tau}}+ \frac{\varrho}{2h} \left| \tfrac{y-y_{k-1}}{\tau} - w_k\right|^2 \mathrm{d} x.
\label{Eq:Altered-functional}
\end{align}

In other words, iteratively minimizing $\mathcal{I}_\tau$ (always with the knowledge of a previous minimizer) yields an approximation of the solution of \eqref{eq:introTimeDelayed} in discrete time points. This approximation has been first found by De Giorgi and later applied to a plethora of problems under the name \emph{minimizing movements} or \emph{time-incremental minimization}, see e.g., \cite{kruvzik2019mathematical,MieRou15RIST}.
 
Now, the minimizing movements approximation can provide us with a solution of \eqref{eq:introTimeDelayed} on $[0,h]$. Using the solution obtained and the same algorithm again, we can construct a solution of \eqref{eq:introTimeDelayed} on $[h,2h]$. Continuing iteratively and gluing the solutions on each subinterval, we ultimately obtain a solution on $[0,T]$. As already mentioned, it has been realized in \cite{benesovaVariationalApproachHyperbolic2020} that the functions obtained in this way are approximants of solutions to \eqref{eq:viscoel-nonsimple-recalled}; in fact, they converge to a solution of \eqref{eq:viscoel-nonsimple-recalled} in an appropriate sense. 

In summary, we may schematically depict the approximating procedure as follows:
 \begin{equation*}
  \begin{tikzcd}[row sep=huge,column sep=huge]
   (\text{NL-disc})  \arrow[r, ""]   & (\text{NL-TD})  \arrow[r, ""]  & (\text{NL})
  \end{tikzcd}
 \end{equation*}

It is this approximation scheme that suggests that results on model reduction obtained on the static level can be transferred to the quasi-static and dynamic settings as well. In fact, it seems that the following \emph{principle} is valid: The $\Gamma$-convergence of stored energy functionals is fundamental and underlying to convergence of the fully dynamic equations in the case of generalized standard materials. Thus, it seems plausible that many of the already obtained results of the convergence of $\Gamma$ can be transferred to dynamic situations.

The aim of this work is to demonstrate this principle for the case of the \emph{linearization} limit of \cite{DalMasoNegriPercivale:02}; i.e., we assume that the forces applied
to the solid body are vanishingly small and want to derive a system of effective equations. In our case, it will be the body force that is vanishingly small, with an asymptotic limit; i.e.:
\begin{equation}\label{smallness-f}
\tilde{f} \approx \delta f,
\end{equation}
for some $f \in L^2([0,T] \times \Omega)$ fixed. Thus, in this situation, it is expected that the deviation of the deformation of the body from a (relaxed) reference configuration, corresponding to the unloaded equilibrium, is small in a suitable sense, so that we write
$$
y(t,x) = x +  \delta u (t,x),
$$
with $u$ the displacement of the body. We describe this idea in more detail in Section \ref{sec:model}.

The goal of this paper is then to study the limit $\delta \to 0$ in \emph{all} the frequency regimes mentioned above, that is, in problems \eqref{basic-balLaw}, \eqref{basic-quasiStat}, \eqref{basic-Stat} both independently and also while establishing a connection between them via the approximation scheme described above. The outline of this endeavor and of this paper is summarized in the following diagram: 

 \begin{equation*}
  \begin{tikzcd}[row sep=huge,column sep=huge]
   (\text{NL-disc}) \arrow[rd, "\text{Prop. \ref{prop:TDDiagonal}}"  {xshift=-8pt, yshift=3pt}] \arrow[r, "\text{Prop. \ref{prop:NLTDexistence}}"] \arrow[d, "\text{Prop. \ref{prop:discDelta}}"] \arrow[d, blue, shift right = 1.5ex, "\text{Thm. \ref{maintheorem3}}"']  & (\text{NL-TD}) \arrow[rd, "\text{Thm. \ref{thm:Diagonal}}"{xshift=-8pt, yshift=3pt}] \arrow[r, "\text{Thm. \ref{thm:existence}}"] \arrow[d,"\text{Prop. \ref{prop:TDDelta}} "] \arrow[d, blue, shift right = 1.5ex, "\text{Cor. \ref{cor:tilt-EDP}}"']  & (\text{NL}) \arrow[d,"\text{Thm. \ref{thm:Delta}}"] \arrow[d, blue, shift right = 1.5ex, "\text{Cor. \ref{cor:dynamicEDP}}"']  \\ (\text{L-disc}) \arrow[r, "\text{Cor. \ref{cor:LTDexistence}}"] & (\text{L-TD}) \arrow[r, "\text{Thm. \ref{thm:existence}}"] & (\text{L})
  \end{tikzcd}
 \end{equation*}

In this diagram, the upper line shows the already described process of approximating inertial problems by quasi-static and, in turn, by static ones. However, at each level of approximation, we may pass to the linearized setting, too. We can do this in two ways: by passing to the limit with the corresponding equation (black line) or via variational convergence of the underlying functionals (blue lines). Our results show that, roughly speaking, the linearization limit commutes with the  approximation procedure described above. In other words, the variational convergence of the energy functional is underlying to the linearization at the other frequency regimes.

\section{Modeling and background on linearization}
\label{sec:model}

In this section, we review the concept of linearization in elasticity from both a modeling point of view as well as a more analytical point of view.

\subsection{Viscoelastic nonsimple materials}

The case of a simple elastic material is illustrative, but is not amenable for the analysis in this paper. In fact, to this date, it is an open problem  whether a \emph{weak} (not measure-valued) solution to \eqref{basic-balLaw} can be found. 

Therefore, within this work,  we will restrict our attention to a more regularized setting of so-called nonsimple \emph{generalized standard materials} (see \cite{HalNgu75MSG}) that may also dissipate energy. For generalized standard materials, it is assumed that there exists an underlying elastic energy functional together with a viscous dissipation functional so that the stress can be calculated by taking appropriate functional derivatives. The term {\it nonsimple} refers to the fact that the elastic properties of the material depend also on the second deformation gradient. This idea was first introduced by Toupin  \cite{Toupin:62,Toupin:64} and proved to be useful in mathematical elasticity, see e.g.~\cite{BCO,Batra,MielkeRoubicek:16,MielkeRoubicek,Podio}.   

Since the elastic energy functional and the viscous dissipation functional play a primary role in setting up the appropriate generalized standard materials model, we start by stating them first:
 \begin{align}
 \label{energies}
\E(y) &:=\int_\Omega W(\nabla y) + P(\nabla^2 y) \, \mathrm{d} x  \\
\Ry(\nabla y,\partial_t \nabla y) &:= \int_\Omega R(\nabla y,\partial_t\nabla y)\, \mathrm{d}x,
\end{align}  
where $W: \mathbb{R}^{d\times d} \to \mathbb{R}$ and $P: \mathbb{R}^{d \times d\times d} \to \mathbb{R}$ are two given functions that represent the elastic energy  densities corresponding to the contributions of simple and nonsimple materials, respectively. 
 The function $\Ry$ addresses the rate of dissipation due to viscosity.

Given \eqref{energies}, the elastic part or the first Piola-Kirchhoff stress reads for all $i,j\in\{1,\ldots, d\}$
\begin{align*}
\mathcal{T}^\text{el}_{ij}(F,G):=  \partial_{F_{ij}} W(F) + \big(\mathcal{L}_{P}(G)\big)_{ij} = \partial_{F_{ij}} W(F) -\sum_{k=1}^d \partial_k \big(\partial_{G_{ijk}}P(G)\big),
\end{align*}
 where $F\in\R^{d\times d}$, and $G\in\R^{d\times d\times d}$  are  the placeholders for the first and second gradient of $y$.  The term 
$\partial_{G}P(G)$ is usually called hyperstress and it is its divergence that contributes to the first Piola-Kirchhoff stress. Let us notice that this form reduces to classical hyperelasticity if the nonsimple material contribution is not present.

Let us consider  the Kelvin-Voigt rheology, i.e.\ that the total stress is the sum of its elastic and viscous part so that the total stress tensor reads
$$
\mathcal{T}(F, \dot{F},G) =  \partial_F W(F) + \mathcal{L}_{P}(G)  + \partial_{\dot{F}}R(F,\dot{F}),
$$
with $F\in\R^{d\times d}$, and $G\in\R^{d\times d\times d}$  as above placeholders for the first and second gradient of $y$ and $\dot{F}$ a placeholder for $\partial_t \nabla y$.  Additionally,  $\mathcal{L}_{P}(G)=-{\rm div}\, \partial_G P(G)$ represents the contribution of hyperstress.  In this case, equation \eqref{basic-balLaw} reads as 

\begin{align}\label{eq:viscoel-nonsimple-basic}
\varrho\partial_{tt}y -{\rm div}\Big( \partial_F W(\nabla y) + \tilde{\eps}\mathcal{L}_{P}(\nabla^2 y)  + \partial_{\dot{F}}R(\nabla y, \partial_t \nabla y)  \Big) = f.\end{align}

We will now discuss linearization in the framework of generalized standard materials. Therefore, we introduce a smallness parameter $\delta >0$ that measures the smallness of the deformation. We first rescale the underlying functionals as follows:
\begin{align}
\mathcal{E}_\delta(y_\delta) &:=\int_\Omega \frac{1}{\delta^2} W(\nabla y_\delta)) + \frac{1}{\delta^{\alpha p} } P(\nabla^2 y_\delta) \, \mathrm{d} x \label{scaled-energ} \\
\Ry_\delta(y_\delta,\partial_t y_\delta) &:= \int_\Omega \frac{1}{\delta^2} R(\nabla y_\delta,\partial_t\nabla y_\delta)\, \mathrm{d}x.
\label{scale:diss}  
\end{align}
Here,  $p>d, 0<\alpha < \frac{p-1}{p(p-2)}$ are selected in such a way that the contribution of the elastic energy corresponding to the second-order, nonsimple material terms is not present in the linearization. Indeed, in this work, we follow the idea that this term is mostly included for its regularizing effect, which is no longer needed once we linearize. Therefore, the scaling is chosen in such a way that it vanishes. However, for materials that form microstructure or feature localized phenomena (\cite{frost2021thermomechanical}) one may want to choose a different scaling for which this term does not vanish in the limit, which can be done by a straightforward adaptation of the argument.

For the viscous contribution, we chose its scaling in such a way that it has a viscous contribution to the stress tensor in the linearized system. However, similarly to the nonsimple material contribution, we could have chosen it in a different way also here so that this term would vanish upon linearization.

The rescaling in \eqref{scaled-energ}-\eqref{scale:diss} is introduced because we are interested in the case of small strains, i.e., when  $\nabla u:=\nabla y - \Id$ is of order $\delta$ for the $\delta$ used in \eqref{scaled-energ}-\eqref{scale:diss}. Here, $u:=y-\id$ is the displacement corresponding to $y$ with  $\id$ and $\Id$ standing for the identity map and identity matrix, respectively. Such a property is certainly meaningful if one considers initial values $y_0$ with $\Vert \nabla y_0 - \Id \Vert_{L^2(\Omega)} \le \delta$. Therefore, it is convenient to define the rescaled displacement $u = \delta^{-1}(y - \id)$. Introducing a proper scaling in the above equation we get
\begin{align}\label{eq:viscoel-nonsimple-scaled}
\varrho\partial_{tt}u -{\rm div}\Big( \delta^{-1}\partial_F W(\Id + \delta \nabla u) + \tilde{\eps}\mathcal{L}_{P}(\delta\nabla^2 u)  + \delta^{-1}\partial_{\dot{F}}R(\Id + \delta \nabla u, \delta \partial_t \nabla u)  \Big) = f\end{align}
for $\tilde{\eps}=\tilde{\eps}(\delta)$ appropriate.  Note that to obtain \eqref{eq:viscoel-nonsimple-scaled} from \eqref{eq:viscoel-nonsimple-basic} we write  $\delta\tilde f:=f$ and then divide the whole equation by $\delta$. 
 Formally, we can pass to the limit and obtain the equation (for $\tilde{\eps} \to 0$ as $\delta \to 0$, e.g.,  $\tilde\eps=\delta^{-\alpha p}$)
\begin{align}\label{eq:viscoel-small}\varrho\partial_{tt}u -{\rm div}\Big( \C_W e(u) + \C_D  e(\partial_t u)   \Big) = \tilde f,
 \end{align}
 where $\C_W:=\partial^2_{F^2}W(\Id)$ is the tensor of elastic constants, $\C_D:= \partial^2_{\dot F^2}R(\Id,0)$ is the tensor of viscosity coefficients,  and $e(u):=(\nabla u+(\nabla u)^\top)/2$ denotes the linear strain tensor. 

The convergence of this system to its linearized variant has been first studied in \cite{friedrichPassageNonlinearLinearized2018} and then for
thermoviscoelasticity in \cite{BaFrKr23NLMTVE,BaFrKrMa24PTTVE}. In this work we rely at several points on the results obtained in these papers, in particular with respect to the scaling of the nonsimple material contribution.

\subsection{Assumptions on the underlying potentials}\label{Sec:potentials}

Before embarking into further discussion, we state the assumptions on the energy and dissipation potentials that we are going to use within the paper. For the contributions to the elastic energy we assume that $W$ satisfies the following:
\begin{align}\label{assumptions-W}
\begin{split}
\text{(i)}& \ \ \text{Frame indifference: } W(QF) = W(F) \text{ for all } F \in \R^{d \times d}, Q \in SO(d),\\
\text{(ii)}& \ \ \text{Smoothness: }W \text{ continuous and $C^3$ in a neighborhood of $\mathrm{SO}(d)$},\\
\text{(iii)}& \ \ \text{Coercivity: }W(F) \ge c\dist^2(F,\mathrm{SO}(d)), \  W(F) = 0 \text{ iff } F \in \mathrm{SO}(d), \\
\text{(iv)}& \ \ \text{Stress free reference configuration: } \partial_F W(\mathbf{Id}) =0   \\
\text{(v)}& \ \ \text{Orientation preservation: } W(F) = + \infty \text{ iff } \det(F) \leq 0 \\
\text{(vi)} & \ \ \text{Penalized compression: } W(F)\ge (\det F)^{-pd/(p-d)} \text{ for all $F \in \mathbb{R}^{d \times d}$ with $\det F \leq 1/2$,}
\end{split}
\end{align}
where $\mathrm{SO}(d) = \lbrace Q\in \R^{d \times d}: Q^\top Q = \Id, \, \det Q=1 \rbrace$.
These assumptions will be needed for not only our analytical study, but are also sound from a physical point of view. In fact, the independence of the observer is a basic physical requirement, so it is imposed in (i). Requirements (iii) and (iv) yield that the material is of minimal energy only in rotations. Requirements (v), (vi) ensure that any deformation preserves orientation and that there is sufficient resistance of the material to extremely large compression (point (vi)).

Here, notice that in (vi), the choice of $p$ corresponds to that in \eqref{assumptions-P}(iii) below. Moreover, we note that the choice that (vi) should be satisfied in matrices with $\det F \leq 1/2$ is essentially arbitrary; any other number smaller than 1 (in order not to conflict with assumption (iii)) would also be possible. We only need a sufficiently fast blow-up of $W(F)$ once $\det F \to 0$.

Finally, the second derivative of $W$ at the identity should represent the tensor of elastic constants
$$
\mathbb{C} := \partial^2_{F} W (\mathbf{Id})
$$
which, as a consequence of this and the frame invariance, satisfies the following symmetry relations for all $i,j,k,l\in\{1,2,\ldots, d\}$
$$
\mathbb{C}_{i j k l}=\mathbb{C}_{j i k l} \quad \text { and } \quad \mathbb{C}_{i j k l}=\mathbb{C}_{i j l k}.
$$

Moreover, we assume that the higher order contribution $P \in C^2(\R^{d\times d \times d})$ satisfies the following:
\begin{align}\label{assumptions-P}
\begin{split}
\text{(i)}& \ \ \text{Frame indifference: } P(QG) = P(G) \text{ for all } G \in \R^{d \times d \times d}, Q \in \mathrm{SO}(d),\\
\text{(ii)}& \ \ \text{$P$ is convex and continuously differentiable on  $\R^{d\times d \times d}$},\\
\text{(iii)}& \ \ \text{there exist constants $0<c_1<c_2$ such that for all $G \in \R^{d \times d \times d}$ we have } \\&   \ \ \ \ \ \    c_1 |G|^p \le P(G) \le c_2 |G|^p, \ \ \ \ \ \    |\partial_G P(G)|  \le c_2 |G|^{p-1} \
\end{split}
\end{align}

 We wish to emphasize that, thanks to the fact that $P$ is  convex and coercive,  no convexity properties of $W$ need to be imposed, which is physically desirable as these would conflict with the frame invariance.

It remains to specify the dissipation potential. Naturally, we need to require that $R(F,\dot F)\ge R(F,0)=0$.  The viscosity stress tensor must also comply with the time-continuous frame-indifference principle  meaning that for all $F$
\begin{align*}
\vstress(F,\dot F)=\tilde\vstress(C,\dot C)
\end{align*}
where $C=F^\top F$ is the right Cauchy-Green tensor and $\tilde\vstress$ is a symmetric matrix-valued function.  This condition constrains
$R$  because $\partial_{\dot F}R(F,\dot F )=S(F,\dot F)$, so that (see \cite{Antm98PUVS,MOS} and also \cite{Demo00WSCN})
\begin{align}\label{eq:frame indifference-R}
R(F,\dot F)=\tilde R(C,\dot C) =\hat R(E,\dot E), 
\end{align}
for some nonnegative functions $\tilde R$ and $\hat R$. Here again $E$ is the Green-Lagrange strain tensor,  for which $2E=C-\Id$ and consequently $\dot C=2\dot E$. In other words, $R$ must depend on the right Cauchy-Green strain tensor $C$ and its time derivative $\dot C$  or, equivalently,  on the Green-Lagrange strain  tensor  $E$ and its time derivative $\dot E$. Altogether,
\begin{align*}
    \mathcal{T}(F,G,\dot F)= \stress(F,G)+\vstress(F,\dot F)\ .
\end{align*}

To comply with these requirements, we set for simplicity
\begin{align} \label{eq:dissipationIntegrand}
    R(F,\dot F)= \frac12 \dot E: \mathbb{D}(C)\dot E.
\end{align}

We assume $\mathbb{D}$ to be continuous, uniformly positive definite, i.e.,
there are positive constants $\alpha_1$ and $\alpha_2$ such that  for all $C,A\in\Rsym^{d\times d}$ 
it holds 
\begin{align}
   \alpha_1|A|^2\le A:\mathbb{D}(C)A \le \alpha_2|A|^2
\end{align}

We assume the same symmetry as for the tensor of elastic constants
$$
\mathbb{D}_{i j k l}=\mathbb{D}_{j i k l} \quad \text { and } \quad \mathbb{D}_{i j k l}=\mathbb{D}_{i j l k},
$$
too. In this form, the dissipation potential is frame-indifferent and no further requirements for physical consistency are necessary.

\section{Auxiliary results}

\subsection{Properties of the elastic energy}

In this section, we summarize some results that will be used throughout the later sections.

\begin{lemma}[Rigidity and uniform bounds] \label{lem:rigidity}
 Fix a number $E_{\max}>0$. Then, there exists a constant $C>0$, such that for all $\delta>0$ small enough and all $y_{\delta} \in  W^{2,p}(\Omega;\R^d)$ with $y_{\delta}|_{\partial\Omega} = \id$ and
 $$
 \mathcal{E}_{\delta}(y_{\delta}) < E_{\max}
 $$
 we have
 \begin{align*}
  \norm[W^{1,2}(\Omega)]{u_{\delta}} &\leq C \quad \text{ and } \quad
  \norm[L^{\infty}(\Omega)]{\delta \nabla u_{\delta}} \leq C \delta^\alpha
 \end{align*} 
  where $u_{\delta}:= \delta^{-1}(y_{\delta}- \id)$.
\end{lemma}

The proof of this statement is based on the, by now classic, Friesecke-James-M\"uller rigidity estimate \cite{FJM}; the first part has been shown, for the first time, in \cite{DalMasoNegriPercivale:02} while the second can be found in \cite{friedrichPassageNonlinearLinearized2018}. We reprove it here for the reader's convenience.

\begin{proof}
 From the growth of the elastic energy near the set of rotations, it follows that
\begin{align*}
 \int_\Omega \frac{c}{{\delta}^2} \mathrm{dist}^2(&\Id+{\delta} \nabla u_{\delta}, \mathrm{SO}(d)) \mathrm{d} x + \frac{1}{{\delta}^{p\alpha}} \| \nabla^2 y_{\delta}\|^p_{L^p(\Omega; \mathbb{R}^{d \times d \times d})}\\ &\leq \left(\int_\Omega \frac{1} {{\delta}^2} W(\id+{\delta} \nabla u_{\delta}) \mathrm{d} x + \frac{1}{{\delta}^{p\alpha}} P({\delta} \nabla^2 y_{\delta}) \mathrm{d} x \right) \\ &\leq \mathcal{E}_{\delta}(y_{\delta}) \leq E_{\max}.
\end{align*}
Appealing to  \cite[Thm. 3.1.]{FJM}, the boundedness of the first term implies that there exists a fixed rotation $Q_{\delta} \in \mathrm{SO}(d)$ such that
\begin{align} \label{eq:rigidity}
\frac{1}{\delta^2} \|\Id+{\delta} \nabla u_{\delta}(t,x)-Q_{\delta}\|^2_{L^2(\Omega)} \leq C.
\end{align}
By the prescribed boundary condition, using a trace theorem we then have
\begin{align*}
 \norm[L^2(\partial \Omega)]{\id - Q_{\delta} \cdot + b} \leq \delta C
\end{align*}
which implies $\abs{Q_{\delta}-\Id} \leq C \delta$. Boundary conditions imply $Q_{\delta} = \Id$ in \eqref{eq:rigidity} so that by means of the Poincar\'e-inequality, we conclude that
$$
\|u_{\delta}\|_{W^{1,2}(\Omega)} \leq C.
$$
From the second-gradient term in the energy, we see that
\begin{equation*}
\|\nabla^2 y_{\delta}\|_{L^p(\Omega)} \leq C \delta^{\alpha},
\end{equation*}
and hence by the Poincaré inequality, we can find a constant matrix $F$ such that
\begin{equation}
 \|\nabla y_{\delta}-F\|_{L^\infty(\Omega)} \leq C \delta^{\alpha} .
 \label{Aux-Poincare}
\end{equation}
We want to estimate the distance of $F$ to $\Id$ and  therefore calculate
$$
\|F-\Id\|_{L^2(\Omega)} \leq \|F-\nabla y_{\delta}\|_{L^2(\Omega)} + \|\nabla y_{\delta}-\Id\|_{L^2(\Omega)} \leq C \delta^\alpha + C \delta,
$$
where the second inequality follows from the already proved estimate $\|u_{\delta}\|_{W^{1,2}(\Omega)} \leq C$. As $F$ is a constant matrix, we conclude that
\begin{align*}
 &\|\delta \nabla u_{\delta}-F\|_{L^\infty(\Omega)} \leq \|\nabla y_{\delta}-\Id\|_{L^\infty(\Omega) }\leq C \delta^{\alpha}.\qedhere
\end{align*}
\end{proof}

\begin{lemma}[Quantified linearization of the energy] \label{lem:linEnergy}
Fix the number $E_\mathrm{max}$. Then, there exists a constant $C>0$ such that for any $\delta >0$ small enough and any $y \in W^{2,p}(\Omega;\R^d)$ satisfying $y|_{\partial\Omega} = \id$ and
 $$
\mathcal{E}_{\delta}(y) < E_\mathrm{max}
 $$ 
it holds that
\begin{align*}
 \abs{ \int_\Omega \left( \delta^{-1}\partial_F W(\nabla y) - \mathbb{C} e(u)\right) : \nabla \varphi \mathrm{d}x } \leq C \delta \norm[L^2(\Omega)]{\nabla u}^2 \norm[L^\infty(\Omega)]{\nabla \varphi},
\end{align*}
for any $\varphi \in W^{2,p}(\Omega;\R^d)$. Here, $u = \delta^{-1}(y - \id)$ and $\mathbb{C} = \partial^2_{F} W(\Id)$.
\end{lemma}

\begin{proof}
 Since $W$ is a $C^3$-function, we can write it in terms of its Taylor expansion around the identity:
\begin{align}
\partial_F W(F) &= \partial_F W(\Id)+  \partial_F^2 W(\Id) (F-\Id)+  \int_0^1\big(\partial^2_F W((1-t)\Id + tF)-\partial_F^2 W(\Id)\big) (F-\Id) \mathrm{d} t \nonumber \\ 
&= \partial_F W(\Id)+  \partial_F^2 W(\Id) (F-\Id)+  \int_0^1\int_0^1 t \partial^3_F W((1-ts) \Id + tsF) (F-\Id) (F-\Id) \mathrm{d}t  \mathrm{d}s
\label{Taylor-E}
\end{align}
Now, plugging in $F=\nabla y(x)$ and additionally multiplying the identity by $\nabla \varphi$, integrating over $\Omega$ and realizing that $\partial_F W(\Id) = 0$ as well as $\partial_F^2W(\Id) (\nabla y - \Id) = \delta \partial_F^2 W(\Id) \nabla u = \mathbb{C} e(u)$ from our assumptions, we have that
\begin{align*}
 \abs{ \int_\Omega \left(\partial_F W(\nabla y) - \delta \mathbb{C} e(u)\right) \cdot \nabla \varphi \mathrm{d}x }  \leq \delta^2 \int_\Omega \int_0^1\int_0^1 t \left | \partial^3_F W(\Id+ts \delta \nabla u) \right| |\nabla u|^2 |\nabla \varphi| \mathrm{d}t  \mathrm{d}s,
\end{align*}
Now as we know from Lemma \ref{lem:rigidity}, as $\delta \in (0,1)$ and $\mathcal{E}_\delta(y) < E_\text{max}$ that also $\|\delta \nabla u\|_{L^\infty(\Omega)} $ is bounded uniformly in $\delta$ and hence there is a constant depending on $W$ only that allows us to estimate
\begin{align*}
 \abs{ \int_\Omega \left(\partial_F W(\nabla y) - \delta \mathbb{C} e(u)\right) \cdot \nabla \varphi \mathrm{d}x }  \leq C \delta^2   \int_\Omega \abs{\nabla u}^2 \abs{\nabla \varphi} \mathrm{d}x \leq C \delta^2 \norm[L^2(\Omega)]{\nabla u}^2 \norm[L^\infty(\Omega)]{\nabla \varphi}.
\end{align*}
Now, dividing by $\delta$ gives the desired estimate.
\end{proof}

We will also need the following result which can be found in \cite{HealeyKroemer:09}.
\begin{lemma}[Healey-Krömer estimate] \label{lem:HK}
Let $y\in W^{2,p}(\Omega;\R^d)$ for $p>d$ and $\Omega\subset\R^d$ be a bounded Lipschitz domain.  If $\det\nabla y>0$ a.e.~in $\Omega$ and $ 1/(\mathrm{det} \nabla y)^q\in L^1(\Omega)$ for some $q\ge pd/(p-d)$ then there is $\varepsilon>0$ depending only on $p,q,d $, and  $\|(\det\nabla y)^{-q}\|_{L^1(\Omega)}$ such that $\det\nabla y\ge\varepsilon$ in $\bar\Omega$.
\end{lemma}

\subsection{Properties of the dissipation functional}

For the dissipation functional the calculation corresponding to the previous lemma is a bit more technical.

\begin{lemma}[Quantified linearization of the dissipation] \label{lem:linDissip}
Fix $E_\mathrm{max} >0$. Then there exists a constant $C>0$ such that for any $\delta > 0$ small enough, all $y \in W^{2,p}(\Omega;\R^d)$ satisfying $y|_{\partial\Omega} = \id$, and all $\varphi\in W^{2,p}(\Omega;\R^d)$
 $$
\mathcal{E}_{\delta}(y) < E_\mathrm{max}
 $$ 
 and any $v \in W^{1,2}(\Omega;\R^d)$ it holds that
\begin{align*}
 \abs{\int_\Omega \delta^{-1}\partial_{\dot{F}} R(\nabla y,\nabla v) \cdot \nabla \varphi -  (\mathbb{D} e(\delta^{-1}v))\cdot e( \varphi) \mathrm{d} x} \leq C \delta  \norm[L^2(\Omega)]{ \delta^{-1}\nabla v} \norm[L^2(\Omega)]{\nabla u} \norm[L^\infty(\Omega)]{\nabla \varphi}.
\end{align*}
Here again $u = \delta^{-1}(y - \id)$.
\end{lemma}

Note that later we will use Lemma \ref{lem:linDissip} for $v = \dot{y}$ or a discretized variant thereof. So, if $y = \id+\delta u$, then $\dot{u} = \delta^{-1}v$. Thus, the scaling here is deliberate, even if the factors $\delta^{-1}$ and $\delta$ could be canceled.

\begin{proof}
We again start with a pointwise estimate. For any fixed matrices $F, \dot{F}$ and $H$ we find that
\begin{align*}
\phantom{{}={}}\partial_{\dot{F}} R\left(F, \dot{F}\right):  H  &= \frac{1}{4}  \left( \dot{F}^\top F + F^\top \dot{F} \right) \mathbb{D}(F^\top F)\left(H^\top F + F^\top H\right)  \\ &=
\frac{1}{4}  \left( \dot{F}^\top F + F^\top \dot{F} \right) \mathbb{D}(\Id)\left(H^\top F + F^\top H\right)\\
&+ \int_0^1  \left( \dot{F}^\top F + F^\top \dot{F} \right) \partial_C \mathbb{D}(t \Id + (1-t) F^\top F)(\Id - F^\top F)\left(H^\top F + F^\top H\right) \mathrm{d} t
\end{align*}

Now plugging in $F=\nabla y$, so that $F-\mathrm{
Id} = \delta \nabla u$ as well as $\dot{F}= \nabla v$ and $H=\nabla \varphi$
 and integrating over $\Omega$, we have that
\begin{align*}
&\left| \int_\Omega \partial_{\dot{F}} R\left(\nabla y, \nabla v \right)\cdot  \nabla \varphi - \frac{1}{4}  \left(\nabla v^\top \nabla y + \nabla y^\top \ \nabla v \right) \mathbb{D}(\Id)\left(\nabla \varphi^\top \nabla y + \nabla y^\top \nabla \varphi \right) \mathrm{d} x \right| \\ & \quad \leq C \int_\Omega \int_0^1 |\nabla y|^2 |\nabla v| |\nabla \varphi|  |\partial_C \mathbb{D}(t \Id + (1-t) \nabla y^\top \nabla y)| (|\delta \nabla u|^2+2|\delta \nabla u|) \mathrm{d} t \mathrm{d} x.
\end{align*}
Now, because we have $\delta \in (0,1)$ and $\mathcal{E}_\delta(y) < E_\text{max}$ we know that $\|\delta \nabla u\|_{L^\infty(\Omega)}$ is uniformly bounded in $\delta$  and thus also $\|\nabla y\|_{L^\infty(\Omega)} = \|\Id + \delta \nabla u\|_{L^\infty(\Omega)}$ is uniformly bounded in $\delta$. This allows us to find a constant $C$ independent of $\delta$ such that we can further estimate this as
\begin{align} \nonumber
 &\left| \int_\Omega \partial_{\dot{F}} R\left(\nabla y, \nabla v \right)\cdot  \nabla \varphi - \frac{1}{4}  \left(\nabla v^\top \nabla y + \nabla y^\top \ \nabla v \right) \mathbb{D}(\Id)\left(\nabla \varphi^\top \nabla y + \nabla y^\top \nabla \varphi \right) \mathrm{d} x \right| \\ &\leq C  \norm[L^2(\Omega)]{ \delta^{-1}\nabla v} \norm[L^2(\Omega)]{\nabla u} \norm[L^\infty(\Omega)]{\nabla \varphi}.\label{eq:estDissipDistorted}
\end{align}

Next we note that since $\nabla y=\Id + \delta\nabla v$ we have
\begin{align*}
 \abs{\nabla v^\top \nabla y + \nabla y^\top \nabla v - 2e(v)} = \delta \abs{\nabla v^\top \nabla u+\nabla u^\top \nabla v} \leq C\delta \abs{\nabla v}\abs{\nabla u}
\end{align*}
Applying the same with $\nabla \varphi$ in place of $\nabla v$, we obtain after some trivial expansion and using again the uniform boundedness of $\norm[L^\infty(\Omega)]{\delta \nabla u}$
\begin{align*}
 &\left| \int_\Omega e(v) \mathbb{D}(\Id) e(\varphi) - \frac{1}{4}  \left(\nabla v^\top \nabla y + \nabla y^\top \ \nabla v \right) \mathbb{D}(\Id)\left(\nabla \varphi^\top \nabla y + \nabla y^\top \nabla \varphi \right) \mathrm{d} x \right| \\
 &\leq C \left(\delta \abs{\nabla v}\abs{\nabla u}\abs{\nabla \varphi}\abs{\nabla y} + \abs{\nabla v}\abs{\nabla u}\abs{\nabla \varphi} \right) \leq  C  \norm[L^2(\Omega)]{ \delta^{-1}\nabla v} \norm[L^2(\Omega)]{\nabla u} \norm[L^\infty(\Omega)]{\nabla \varphi}
\end{align*}
Combining this with \eqref{eq:estDissipDistorted} and multipliying by $\delta$ finishes the proof.
\end{proof}

At last we mention the types of Korn inequalities that we will need.

\begin{lemma}[Generalized Korn's inequality \cite{Neff02KFIN,Pomp03KFIV}] \label{lem:KornGeneral}
 Let $y \in W^{2,p}(\Omega;\R^d)$, $p>d$,  with $y|_{\partial \Omega} = \id$, $\det \nabla y > 0$ everywhere in $\bar\Omega$,  and $v \in W^{1,2}(\Omega;\R^d)$ with $v|_{\partial \Omega} = 0$. Then there exists a constant $C$ depending only on $\norm[W^{2,p}]{y}$ and $\inf_{x\in \bar\Omega} \det \nabla y$ such that
 \begin{align*}
  \norm[W^{1,2}(\Omega)]{v}^2 \leq C R(\nabla y, \nabla v).
 \end{align*}
\end{lemma}

For the rescaled deformation this then directly implies:

\begin{corollary}[Discrete Korn-type rigidity estimate] \label{lem:KornDiscrete}
 Fix $E_{\max} > 0$. There exists a constant $C>0$ such that for any $\delta > 0$ small enough and any $(y_{k,\delta})_k \subset W^{2,p}(\Omega)$ with $\sup_{k} \mathcal{E}_{\delta}(y_{k,\delta}) < E_{\max}$ we have
 \begin{align*}
  \sum_{k=1}^N \norm[W^{1,2}(\Omega)]{\tfrac{ \nabla u_{k,\delta} - \nabla u_{k-1,\delta}}{\tau}} \leq   C \sum_{k=1}^N \mathcal{R}_\delta\left(y_{k-1,\delta},\tfrac{ y_{k,\delta} - y_{k-1,\delta}}{\tau}\right) = C \sum_{k=1}^N \mathcal{R}_\delta\left(\id+\delta u_{k-1,\delta},\delta \tfrac{ u_{k,\delta} - u_{k-1,\delta}}{\tau}\right).
 \end{align*}
 where we let $u_{k,\delta} = \delta^{-1}(y_{k,\delta}-\id)$.
\end{corollary}

\begin{corollary}[Continuous Korn-type rigidity estimate] \label{lem:KornCont}
 Fix $E_{\max} > 0$. There exists a constant $C>0$ such that for any $\delta > 0$ small enough and $y_\delta \in W^{1,2}(I;W^{1,2}(\Omega;\R^d)) \cap L^\infty(I;W^{2,p}(\Omega;\R^d))$ with $\sup_{t\in I} \mathcal{E}_{\delta}(y_\delta(t)) < E_{\max}$ we have
 \begin{align*}
  \int_I \norm[W^{1,2}(\Omega)]{\partial_t u_\delta}^2 \mathrm{d}t \leq C \int_I \mathcal{R}_\delta(\nabla y_\delta, \partial_t \nabla y_\delta) \mathrm{d}t = C \int_I \mathcal{R}_\delta(\id + \delta \nabla u_\delta, \delta \partial_t \nabla u_\delta) \mathrm{d}t.
 \end{align*}
 where we let $u_\delta = \delta^{-1}(y-\id)$.
\end{corollary}

\subsection{Dual dissipation term}

The final ingredient we will need is the Legendre dual of the dissipation term, as this will play a role both in the energy estimate, as well as in the dynamic convergence on the level of the functional. For this we recall the following definition:
\begin{definition}
 Let $X$ be a Banach space and let $f:X\to \R$ convex. Then the Legendre transform of $f$ is defined as
 \begin{align*}
  f^*:X^* \to \R \cup \{\infty\}: \xi \mapsto \sup_{x\in X} \inner[X^*\times X]{\xi}{x} - f(x).
 \end{align*}
\end{definition}
In particular if $f$ has some sort of coercivity (e.g.\ quadratic growth), then $f^*(\xi) <\infty$ and the supremum is uniquely defined.

We also need the following classical inequality, including its equality case:
\begin{theorem}[Fenchel-Young inequality] \label{thm:Fenchel}
 Let $X$ be a Banach space and let $f:X\to \R$ convex and Frech\'et-differentiable. Let $x \in X$, $\xi \in X^*$. Then
 \begin{align} \label{eq:FenchelEquality}
  \inner{\xi}{x} \leq f(x) + f^*(\xi)
 \end{align}
 where equality holds precisely if $Df(x) = \xi$.
\end{theorem}
\begin{proof}
 The inequality itself is a direct consequence of the definition of $f^*$. For a discussion of the equality case see e.g.\ \cite[Thm 23.5]{rockafellar}.
\end{proof}

As a direct consequence $f^*(\xi)$ is finite whenever $\xi \in \operatorname{Range}(Df)$.

We will use the dual representative for the dissipation functional; most importantly in its linearized form $\int_\Omega e(u_t) \cdot \left(\mathbb{D} e(u_t) \right) \mathrm{d}x$. Therefore, a particularly important setting for us is $X=W^{1,2}_0(\Omega;\R^d)$ and thus $X^*=W^{-1,2}(\Omega;\R^d) = (W^{1,2}_0(\Omega;\R^d))^*$. 

We also recall that since $\Omega$ is bounded, any element $\xi \in W^{-1,2}(\Omega;\R^d)$ can be represented by a function $A_\xi \in L^2(\Omega;\R^{d\times d})$ such that
\begin{align*}
 \inner{\xi}{u} = \int_\Omega A_\xi:\nabla u \, \mathrm{d}x.
\end{align*}
From the definition it is also immediately clear that such an $A_\xi$ is only well defined up to a weakly divergence-free field. This turns out to be crucial, in our case, as \emph{density} of the dissipation functional does not have full coercivity - it is only coercive in symmetric matrices. This means that its Legendre transform (understood as a function of $\mathbb{R}^{3 \times 3}$) takes infinite values. However, the actual integral functional does not suffer from this problem. We then get the following:

\begin{lemma}
 Let $f:\Omega \times \R^{d\times d} \to [0,\infty)$ be measurable with $f(x,\cdot)$ continuously differentiable and convex for almost all $x\in \Omega$ and assume that there exists a $c > 0$ such that $f(x,A) \leq c \abs{A}^2$ for almost all $x\in \Omega$ and $A\in \R^{d\times d}$. Define
 \begin{align*}
  \mathcal{F}:W^{1,2}_0(\Omega;\R^d) \to \R; \quad 
  u \mapsto \int_\Omega f(x, \nabla u(x)) \mathrm{d}x
 \end{align*}
 Assume that additionally there is a Korn-type inequality, i.e.\ there exists $C>0$ such that
 \begin{align} \label{eq:dualKorn}
  \norm[L^2(\Omega)]{\nabla u}^2 \leq C\left[1+ \mathcal{F}(u)\right] \text{ for all } u\in W^{1,2}_0(\Omega;\R^d).
 \end{align}
 Then the Legendre transform $\mathcal{F}^*:W^{-1,2}(\Omega;\R^d) \to \R$ is given by
 \begin{align*}
  \mathcal{F}^*(\xi) = \int_\Omega f^*(x,A_\xi) \mathrm{d}x
 \end{align*}
 where $f^*$ is the Legendre transform of $f$ with respect to the second argument and $A_\xi$ is a representative of $\xi$ for which $A_\xi \in \operatorname{Range}(D_2f(x,\cdot))$.
\end{lemma}

\begin{proof}
 Fix $\xi \in W^{-1,2}(\Omega;\R^d)$ and an arbitrary associated $\tilde{A}_\xi\in L^2(\Omega;\R^{d\times d})$. Consider the maximization problem
 \begin{align*}
  \max_{u\in W^{1,2}_0(\Omega;\R^d)} \int_\Omega \tilde{A}_\xi : \nabla u \,\mathrm{d} x - \mathcal{F}(u).
 \end{align*}
 Using \eqref{eq:dualKorn} for coercivity and the convexity of $f$ for semicontinuity, we can infer that the maximum is attained for some $u_\xi \in W^{1,2}_0$. The corresponding Euler-Lagrange equation is then given by
 \begin{align*}
  \int_\Omega \tilde{A}_\xi : \nabla \varphi - D_2f(x,\nabla u_\xi): \nabla \varphi\, \mathrm{d} x = 0.
 \end{align*}

 In other words, we can set $A_\xi := D_2f(x,\nabla u_\xi)$ as a new representative of $\xi$, as it only differs from $\tilde{A}_\xi$ by a weakly divergence free field. But then by \eqref{eq:FenchelEquality}
 \begin{align*}
  \mathcal{F}(u)+ \mathcal{F}^*(\xi)  &= \int_\Omega D_2f(x,\nabla u) : \nabla u \,\mathrm{d} x = \int_\Omega f(\nabla u) + f^*(x,D_2f(x,\nabla u)) \mathrm{d} x = \mathcal{F}(u) + \int_\Omega f^*(x,A_\xi) \mathrm{d}x. \qedhere
 \end{align*}
\end{proof}

\begin{remark}
 Note that independent of the choice of representative $A_\xi$, a pointwise application of the Fenchel-Young inequality implies
 \begin{align*}
  \mathcal{F}(u)+ \mathcal{F}^*(\xi)  &\leq \int_\Omega f(\nabla u) + f^*(x,A_\xi) \mathrm{d} x.
 \end{align*}
 Then  taking the minimum and using the previous Lemma yields the alternative characterization
 \begin{align*}
  \mathcal{F}^*(\xi) &= \min\left\{ \int_\Omega f^*(x,A_\xi) \mathrm{d} x : A_\xi \text{ represents } \xi\right\}. \qedhere
 \end{align*}
\end{remark}

Next we recall the following result on the Legendre transform of (semi-)definite quadratic forms.
\begin{lemma}
 Let $S \in \R^{n\times n}$ be symmetric positive semi-definite\footnote{Note that for us later $n= d^2$, as $S$ will be replaced by the fourth-order tensor $\mathcal{D}$.} and let $f: \R^n \to \R; v \mapsto v^T S v$ be the associated quadratic form. Then
 \begin{align*}
  f^* : \R^n \to \R \cup \{\infty\}; w \mapsto \begin{cases} \frac{1}{4} w S^{-1} w &\text{ if } w\in \operatorname{Range}(S) = \operatorname{Ker}(S)^\bot \\ \infty &\text{ otherwise} \end{cases}
 \end{align*}
 where $S^{-1}$ is understood as a map $\operatorname{Range}(S) \to \R^n$, or alternatively as a pseudoinverse.
\end{lemma}

\begin{proof}
 Using the spectral theorem, we can diagonalize $S$ and treat all eigenspaces separately. Now if $w$ is in the zero-eigenspace $(\lambda w) \cdot w - f(\lambda w) = \lambda \norm{w}^2$ is unbounded and thus $f^*(w) = \infty$. For all non-zero eigenspaces, the statement is equivalent to a weighted Young's inequality.
\end{proof}

For our situation, this allows us to conclude the following:
\begin{proposition} \label{prop:dualDissipation}
 Let $\mathcal{R}$ be defined as in \eqref{eq:dissipationIntegrand}. Then for any $y\in W^{1,2}(\Omega;\R^d)$ its Legendre-transform (with respect to the second argument) is given by 
 \begin{align*}
  \mathcal{R}^*(y,\cdot) : W^{-1,2}(\Omega;\R^d) &\to \R \cup \{\infty\};\\
  \xi &\mapsto \int_\Omega \frac{1}{4} (\nabla y^{-\top} A_\xi + A_\xi^{\top} \nabla y^{-1} )\mathbb{D}(\nabla y)^{-1} (\nabla y^{-\top} A_\xi + A_\xi^{\top} \nabla y^{-1} )
 \end{align*}
 where $A_\xi$ is chosen such that $\nabla y A_\xi$ is symmetric almost everywhere.
\end{proposition}

\begin{proof}
 As $\nabla y$ is invertible, we first note that we can choose a representation of $\xi$ such that $\nabla y^{-1} A_\xi$ is a symmetric matrix almost everywhere. We then rewrite using the same invertibility, as well as the previous lemma and the symmetry properties of $\mathbb{D}$
 \begin{align*}
  \int_\Omega A_\xi : \nabla \varphi \mathrm{d}x &= \int_\Omega (\nabla y^{-\top} A_\xi) : \nabla y^\top \nabla \varphi \mathrm{d}x\\
  &\leq \int_\Omega (\nabla y^\top \nabla \varphi) \mathbb{D}(\nabla y) (\nabla y^\top \nabla \varphi) + \int_\Omega (\nabla y^{-\top} A_\xi )\mathbb{D}(\nabla y)^{-1} (\nabla y^{-\top} A_\xi )\\
  &= \mathcal{R}(y,\varphi) + \int_\Omega \frac{1}{4} (\nabla y^{-\top} A_\xi + A_\xi^{\top} \nabla y^{-1} )\mathbb{D}(\nabla y)^{-1} (\nabla y^{-\top} A_\xi + A_\xi^{\top} \nabla y^{-1} )
 \end{align*}
 where equality holds precisely if
 \begin{align*}
  \nabla y^{-\top} A_\xi = \mathbb{D}(\nabla y) \nabla y^\top \nabla \varphi.
 \end{align*}
 Multiplying by $\nabla y^\top$, this is equivalent to $\xi = D_2\mathcal{R}(y,\varphi)$, which finishes the proof by the characterization.
\end{proof}

\begin{remark}
 The role of the symmetric gradient in linear elasticity and its properties even in Sobolev-spaces of negative order has been studied in \cite{amrouche}. Similarly the representation of operators in $W^{-1,2}$ by $L^2$-functions acting on the derivative is textbook knowledge. However to the authors' knowledge, the curious fact that this allows a symmetric representation in the case of vector-valued spaces, has so far not been noticed.
\end{remark}

 \section{\texorpdfstring{$\Gamma$}{Gamma}-convergence of the underlying functionals }\label{Sec:Gamma_conv_energy}

As we highlighted in the introduction, the functional convergence of the energy and the dissipation functional as $\delta \to 0$ are fundamental to all other results. Therefore, we recall them in this section while mostly relying on the works \cite{DalMasoNegriPercivale:02,friedrichPassageNonlinearLinearized2018}.

The proper notion of functional convergence that we shall use is the so-called $\Gamma$-convergence of functionals $\mathcal{F}_\delta$ to $\mathcal{F}_0$.
The general idea behind $\Gamma$-convergence  can be reduced to three parts:
 \begin{enumerate}
  \item Equi-coercivity: For any sequence $(y_\delta)_\delta$ such that $\sup_\delta \mathcal{F}_\delta(y_\delta) < \infty$, there exists some converging subsequence (possibly after rescaling), with a limit $y$.
  \item Lower semicontinuity: For any sequence $(y_\delta)_\delta$ converging in the above (possibly rescaled) sense to $y$, we have $\liminf_{\delta \to 0}\mathcal{F}_\delta(y_\delta) \geq \mathcal{F}(y)$
  \item Recovery sequence: For any $y$, there exists an approximating sequence $(y_\delta)_\delta$ converging to $y$ such that $\limsup_{\delta \to 0}\mathcal{F}_\delta(y_\delta) \leq \mathcal{F}(y)$
 \end{enumerate}
 Using these properties, one can then relatively quickly relate (almost)minimizers of $\mathcal{F}_\delta$ to those of $\mathcal{F}$, as we will see in the following.

 As the definition of $\Gamma$-convergence is highly dependent on the underlying spaces, we will always explicitly give the respective notion of convergence. However, we will always be in the same rescaled case used for linearization, where we actually want the displacements $u_\delta := \delta^{-1}(y_\delta -\id)$ to converge to $u$ in some appropriate space.

  \begin{remark}
     In the following, we will not relabel subsequences to reduce the complexity of the notation. In fact, since all solutions to the linearized systems we study are unique and we will show that any bounded sequence has a converging subsequence, this directly implies convergence for any sequence and thus for $\delta \to 0$ in general.
 \end{remark}

 \subsection{\texorpdfstring{$\Gamma$}{Gamma}-convergence of the energy}

 We will now study the energy functional. Although this result can essentially be deduced from the seminal work \cite{DalMasoNegriPercivale:02}, we still choose to present it here in full. We begin with equi-coercivity, which is a direct consequence of Lemma \ref{lem:rigidity} and the weak compactness.

 \begin{corollary}
  Let $(y_\delta)_\delta$ such that $\sup_{\delta} \mathcal{E}_\delta(y_\delta) < \infty$. Then there exists a  subsequence $(y_{\delta})_{\delta}$ and a function $u\in W^{1,2}_0(\Omega;\R^d)$ such that for $u_{\delta} := \delta^{-1} (y_{\delta}-\id)$ we have
  \begin{align*}
    u_{\delta} \rightharpoonup u \text{ in } W^{1,2}(\Omega;\R^d)
  \end{align*}
  and $\norm[L^\infty(\Omega)]{\nabla y_{\delta}-\Id} < C\delta^\alpha$ for some $C>0$.
 \end{corollary}

 Next we deal with lower-semicontinuity and the recovery condition.

 \begin{lemma}\label{lem:energyLSC}
  Let $(y_\delta)_\delta$ be such that the sequence $(u_{\delta})_\delta$ with $u_{\delta} := \delta^{-1} (y_{\delta}-\id)$ converges weakly to $u \in W^{1,2}(\Omega;\R^d)$. Then
  \begin{align*}
   \liminf_{\delta \to 0} \mathcal{E}_\delta(y_\delta) \geq \int_\Omega  \frac{1}{2} e( u) \cdot \mathbb{C} e ( u).
  \end{align*}
 \end{lemma}

 \begin{proof}
  We can assume that $\mathcal{E}_\delta(y_\delta)$ is bounded, as otherwise there is nothing to prove. Moreover, due to Lemmas \ref{lem:rigidity} and \ref{lem:HK}, we can find a  $\delta_0$ such that for all ${\delta} \leq \delta_0$ we have that uniform bounds 
\begin{equation}
\|\delta \nabla u_{\delta}\|_{L^\infty(\Omega)} \leq 1/2 \quad \text{ as well as } \quad
 \det (\Id + {\delta} \nabla u_{\delta}) \geq 1/2.
 \label{Bound}
\end{equation}
Therefore appealing to the $C^3$-smoothness of the function $W$ in a neighborhood of the identity, we may employ the Taylor expansion \eqref{Taylor-E} relying on \eqref{assumptions-W}(iii-iv), as well as the symmetry assumptions on $\mathbb{C}$ to observe that
\begin{equation}
 \frac{1}{\delta^2} \int_\Omega W(\nabla y_{\delta}) \mathrm{d} x \geq \int_\Omega \frac12 e(u_{\delta}) \mathbb{C} e( u_{\delta}) - C \delta,
\label{Taylor-W}
\end{equation}
where $C$ is independent of $\delta$. Since this is a quadratic form, it is then lower semi-continuous under weak convergence. As $\int_\Omega P(\nabla^2 y_{\delta})\mathrm{d} x \geq 0$, this is all we need to deal with the elastic energy.
\end{proof}

\begin{lemma}\label{lem:energyRecovery}
 Let $u \in W^{1,2}(\Omega;\R^d)$. For any sequence $(y_\delta)_\delta \subset W^{2,p}(\Omega;\R^d)$ such that the sequence $(u_{\delta})_\delta$ with $u_{\delta} := \delta^{-1} (y_{\delta}-\id)$ converges strongly to $u$ in $W^{1,2}(\Omega;\R^d)$ and $\int_\Omega \frac{1}{\delta^{\alpha p}} P(\nabla^2y_\delta) \mathrm{d}x \to 0$ we have
 \begin{align*}
  \limsup_{\delta \to 0}\mathcal{E}_\delta(y_\delta) \leq \int_\Omega  \frac{1}{2} e( u) \cdot \mathbb{C} e ( u).
 \end{align*}
 Such a sequence exists for any $u\in W^{1,2}(\Omega;\R^d)$.
\end{lemma}

\begin{proof}
 Estimating the error terms in the previous calculation with the opposite sign, we obtain
 \begin{align*}
 \frac{1}{\delta^2} \int_\Omega W(\nabla y_{\delta}) \mathrm{d} x \leq \int_\Omega \frac12 e(u_{\delta}) \mathbb{C} e( u_{\delta}) + C \delta \to \int_\Omega \frac12 e(u) \mathbb{C} e( u).
 \end{align*}
 As the term involving $P$ is assumed to vanish, this proves the inequality.

 Next we construct an appropriate sequence. By a general approximation argument in the theory of $\Gamma$-convergence it suffices to do so for smooth functions $u$, cf.\ \cite[Proposition 4.1]{DalMasoNegriPercivale:02}, i.e.\ we are free in assuming $u \in W^{2,p}(\Omega; \mathbb{R}^d)$. For such a function, we may set
$$
y_{\delta}(x) = x+ \delta u(x)
$$
to obtain the sought recovery sequence. Then convergence of $u_\delta$ in $W^{1,2}(\Omega;\R^d)$ is trivial. Moreover, as $\alpha < 1$
\begin{align*}
\frac{1}{\delta^{\alpha p}} \int_\Omega P(\nabla^2 y_{\delta}) \mathrm{d} x &\leq C_2 \delta^{p(1-\alpha)} \|u\|_{W^{2,p}(\Omega)} \to 0.\qedhere
\end{align*}
\end{proof}

Combining these three properties we recover a variant of a classic result:
\begin{corollary} \label{cor:energyGamma}
 The family of functionals $\mathcal{E}_\delta$ $\Gamma$-converges to $\mathcal{E}_0$ in the above sense.
\end{corollary}

\subsection{\texorpdfstring{$\Gamma$}{Gamma}-convergence of the dissipation functional}

We now proceed similarly for the dissipation functional. As dissipation always appears in conjunction with other terms, we will restrict ourselves to providing adjunctiary results.

\begin{lemma} \label{lem:dissipLSC}
 Let $(y_\delta)_\delta$ be such that $\sup_{\delta} \mathcal{E}_\delta$ is bounded. Furthermore, assume that $(v_\delta)_\delta \in W^{1,2}(\Omega;\R^d)$ is such that $\delta^{-1} v_\delta \rightharpoonup v$ in $W^{1,2}(\Omega;\R^d)$. Then
 \begin{align*}
  \liminf_{\delta \to 0} \mathcal{R}_\delta(y_\delta, v_\delta) \geq \frac{1}{2} \int_\Omega e(v) \mathbb{D}(\Id) e(v) \mathrm{d}x.
 \end{align*}
\end{lemma}

\begin{proof}
 We rewrite the dissipation term using $y_{\delta} = \id +\delta u_{\delta}$ and the symmetry properties of $\mathbb{D}$ as
\begin{align}
 \frac{1}{\delta^2} &\int_\Omega {R}\left(\nabla y_{\delta},\nabla v_\delta\right) \mathrm{d}x \nonumber \\&=
 \frac{1}{8\delta^2}\int_\Omega \left(\nabla v_\delta^\top \nabla y_{\delta}+ \nabla y_{\delta}^\top \nabla v_\delta\right) \mathbb{D}( \nabla y_{\delta}^\top \nabla y_{\delta}) \left(\nabla v_\delta^\top \nabla y_{\delta}+ \nabla y_{\delta}^\top \nabla v_\delta\right), \nonumber   \\&=
  \frac{1}{2\delta^2} \int_\Omega e\left(v_\delta\right) \mathbb{D}(\nabla y_{\delta}^\top \nabla y_{\delta}) e \left(v_\delta\right) + \frac{1}{2\delta} \int_\Omega \left(\nabla v_\delta^\top \nabla u_{\delta}+ \nabla u_{\delta}^\top \nabla v_\delta\right) \mathbb{D}( \nabla y_{\delta}^\top \nabla y_{\delta}) e\left(v_\delta\right) \nonumber \\ &+
     \frac{1}{8}\int_\Omega \left(\nabla v_\delta^\top \nabla u_{\delta}+ \nabla u_{\delta}^\top \nabla v_\delta\right)  \mathbb{D}( \nabla y_{\delta}^\top \nabla y_{\delta}) \left(\nabla v_\delta^\top  \nabla u_{\delta}+  \nabla u_{\delta}^\top \nabla v_\delta\right) ,
\label{rewriting-D}
\end{align}
Now, we know that due to continuity, $\mathbb{D}(\nabla y_{\delta}^\top \nabla y_{\delta})$ is bounded  and since $\|\delta \nabla u_{\delta}\|_{L^\infty(\Omega)} \leq C \delta^\alpha$, we get that
\begin{align}
    \frac{1}{\delta^2} &\int_\Omega {R}\left(\nabla y_{\delta},v_\delta\right) \mathrm{d}x \nonumber \geq  \frac{1}{2} \int_\Omega e\left(\delta^{-1} v_\delta\right)\mathbb{D}(\nabla y_{\delta}^\top \nabla y_{\delta}) e\left(\delta^{-1} v_\delta\right) - C \delta^{\alpha},
\end{align}
and further by Taylor-expansion (see also the proof of Lemma \ref{lem:linDissip})
\begin{align}
    \frac{1}{\delta^2} &\int_\Omega {R}\left(\nabla y_{\delta},\nabla v_\delta\right) \mathrm{d}x  \geq  \frac{1}{2}\int_\Omega e\left(\delta^{-1}v_\delta\right) \mathbb{D}(\Id) e\left(\delta^{-1}v_\delta\right) - C \delta^{\alpha}. \label{Taylor-D}
\end{align}
As the first term is convex, lower-semicontinuity then follows.
\end{proof}

\begin{lemma}\label{lem:dissipRecovery}
 Let $(y_\delta)_\delta$ such that $\sup_{\delta} \mathcal{E}_\delta$ is bounded and let $(v_\delta)_\delta$ such that $\delta^{-1}v_\delta \to v$ in $W^{1,2}(\Omega;\R^d)$. Then
 \begin{align*}
  \limsup_{\delta \to 0} \mathcal{R}_\delta(y_\delta,v_\delta) \leq \frac{1}{2} \int_\Omega e(v) \mathbb{D}(\Id) e(v) \mathrm{d}x.
 \end{align*}
\end{lemma}

\begin{proof}
 Taking \eqref{rewriting-D} from the previous proof and estimating it similarly in the other direction, we conclude
\begin{align*}
    \frac{1}{\delta^2} &\int_\Omega {R}\left(\nabla y_{\delta},\nabla v_\delta\right) \mathrm{d}x \leq  \frac{1}{2}\int_\Omega e\left(\delta^{-1}v_\delta\right) \mathbb{D}(\Id) e\left(\delta^{-1}v_\delta\right) + C \delta^{\alpha}.
\end{align*}
 Now the quadratic term converges due to the strong convergence of $\delta^{-1}v_\delta \to v$ and we obtain the desired result.
\end{proof}

\subsection{\texorpdfstring{$\Gamma$}{Gamma}-limit of the nonlinear discrete cost functional}

Lastly, in this subsection, we show $\Gamma$-convergence of the functionals $\mathcal{I}_{\delta, \tau}$ which correspond to \eqref{Eq:Altered-functional} with now the rescaled functionals so that we obtain 
\begin{definition}[Nonlinear discrete cost functional]
  \label{def:solNLdiscCost}
  Let $\delta>0$, $\tau> 0$, $h>0$ and consider $y_{k-1} \in W^{2,p}(\Omega;\R^d)$, $w_k  \in L^{2}(\Omega;\R^d)$ and $f_{k} \in L^2(\Omega;\R^d)$ be given. We will define
 \begin{align}
\mathcal{I}_{\delta,\tau}(y) := \mathcal{E}_\delta(y) + \tau \mathcal{R}_\delta\left( y_{k-1},\tfrac{ y- y_{k-1}}{\tau}\right) - \tau \frac{1}{\delta^2}\int_\Omega f_{k}\cdot {\tfrac{y-y_{k-1}}{\tau}}+ \frac{\varrho}{2h\delta^{2}} \left| \tfrac{y-y_{k-1}}{\tau} - w_k\right|^2 \mathrm{d} x.
\label{minim-prob-delta}
\end{align}
 \end{definition}

As we saw in the introduction, this functional needs to be minimized in order to be able to construct solutions to the problems (NL-disc) as well as (NL-TD) and (NL). Thus we expect $\Gamma$-convergence in these functionals to their linearized counterpart.  

 \begin{theorem}[$\Gamma$-convergence of $\mathcal{I}_{\delta,\tau}$]\label{maintheorem3}
Fix $k \in \{0, \ldots, N\}$. Assume that $(y_{k-1,\delta})_{\delta} \subset W^{2,p}(\Omega; \mathbb{R}^d)$, $(w_{k,\delta})_\delta \subset W^{2,p}(\Omega;\R^d)$ and $(f_{k,\delta})_\delta \subset L^{2}(\Omega;\R^d)$  are well-prepared in the sense that there exist functions $u_{k-1}, f_k, w_k$ such that
\begin{align*}
\delta^{-1}(y_{k-1,\delta} - \id) &\to u_{k-1} \in W^{1,2}(\Omega; \mathbb{R}^d),\\
\delta^{-1} f_{k,\delta} &\to f_k \in L^2(\Omega;\mathbb{R}^d)\\
\delta^{-1}w_{k,\delta} &\to w_k \in L^2(\Omega; \mathbb{R}^d)
\end{align*}
and
$$
\|\delta^{-1} \nabla u_{k-1,\delta}\|_{L^\infty(\Omega)} \leq C \delta^\alpha.
$$
Then the nonlinear discrete cost functionals are equi-coercive and $\Gamma$-converge to 
\begin{align}
\mathcal{I}_{0,\tau}(u) := &\int_\Omega  \frac{1}{2} e( u) \cdot \mathbb{C} e ( u) +\frac{1}{2}e\big(\tfrac{ u - u_{k-1}}{\tau}\big)\cdot \mathbb{D} e\big(\tfrac{u - u_{k-1}}{\tau}\big) - \tau \int_\Omega {f}_k\cdot {\tfrac{u-u_{k-1}}{\tau}}+ \frac{\rho}{2h} \left| \tfrac{u-u_{k-1}}{\tau} - w_k\right|^2 \mathrm{d} x,
\label{linear-disc-cost}
\end{align}    
that we call the linearized discrete cost functional, in the weak $W^{1,2}(\Omega; \mathbb{R}^d)$ topology as $\delta \to 0$.
\end{theorem}

\begin{proof}
\emph{Step 1 (equi-coercivity):} Let us take a sequence of deformations $(y_{\delta})_{\delta} \subset W^{1,2}(\Omega; \mathbb{R}^d)$ and a constant $C >0$ such that
\begin{align*}
\sup_\delta \mathcal{I}_{\delta,\tau}(y_{\delta}) = \sup_\delta \Big[&\mathcal{E}_{\delta}(y_{\delta}) + \tau \mathcal{R}_{\delta}\left(y_{k-1,\delta},\tfrac{ y_{\delta}- y_{k-1,{\delta}}}{\tau}\right) \\ &- \tau \frac{1}{{\delta}^2}\int_\Omega f_{k,\delta}\cdot {\tfrac{y_{\delta}-y_{k-1,\delta}}{\tau}}+ \frac{\rho}{2h\delta^2} \left| \tfrac{y_{\delta}-y_{k-1,\delta}}{\tau} - w_{k,\delta}\right|^2 \mathrm{d} x \Big] \leq C.
\end{align*}

Then, omitting the dissipation functional and rewriting the inertial and force term in displacements, we find that
\begin{align*}
\sup_\delta \Big[\mathcal{E}_{\delta}(y_{\delta})   &- \tau \int_\Omega  \delta^{-1}f_{k,\delta}\cdot  \tfrac{u_{\delta}-\delta_k^{-1} (y_{k-1,\delta}-\id)}{\tau} \\ &+ \frac{1}{2h}  \left| \tfrac{u_{\delta}-\delta_k^{-1} (y_{k-1,\delta}-\id)}{\tau} - \delta^{-1}w_{k,\delta}\right|^2 \mathrm{d} x \Big] \leq C,
\end{align*}
estimating the force term leads to
\begin{align*}
\mathcal{E}_{\delta}(y_{\delta}) + &\int_\Omega \frac{1}{4h} \left| \tfrac{u_{\delta}-\delta_k^{-1} (y_{k-1,\delta}-\id)}{\tau}\right|^2  \\ & \leq C \big(1+\|\delta^{-1}f_{k,\delta}\|_{L^2(\Omega)}^2 + \|\delta^{-1}w_{k,\delta}\|_{L^2(\Omega)}^2\big),
\end{align*}
where the right-hand side is bounded independently of $\delta$. Let $u_{k,\delta} := \delta^{-1} (y_{k,\delta}-\id)$. Using Lemma \ref{lem:rigidity} we then get that
$$
\|u_{k,\delta}\|_{W^{1,2}(\Omega)} \leq C
$$ as well a uniform proximity of the deformation gradient to the identity since \begin{equation*}
\|\delta \nabla u_{k,\delta}\|_{L^\infty(\Omega)} \leq C \text{ and }
\|\delta \nabla^2 u_{k,\delta}\|_{L^p(\Omega)} \leq \delta^{\alpha}.
\end{equation*}
Then equi-coercivity follows by weak compactness.

\emph{Step 2 ($\Gamma$-convergence):}
We now prove the $\Gamma$-convergence results and show the ``$\liminf$''-inequality first. We may take a sequence of deformations $(y_{\delta})_{j\in \N} \subset W^{2,p}(\Omega; \mathbb{R}^d)$ such that
$$
\mathcal{I}_{\delta,\tau}(y_{\delta}) \leq C,
$$
for otherwise there is nothing to prove. We also assume that there is a $u \in W^{1,2}(\Omega, \mathbb{R}^d)$ such that
$$
u_{\delta}:= \delta^{-1}(y_{\delta} -\id) \rightharpoonup u \quad \text{ in } W^{1,2}(\Omega; \mathbb{R}^d)
$$

The lower semicontinuity of the elastic energy is found in Lemma \ref{lem:energyLSC}.
By taking $v_\delta := \tfrac{y_\delta-y_{k-1,\delta}}{\tau}$ in Lemma \ref{lem:dissipLSC} and noting that $\delta^{-1} v_\delta = \tfrac{u_\delta-u_{k-1,\delta}}{\tau}$, we obtain the lower-semicontinuity of the dissipation term.

Combining all of the above with the fact that the force term is affine in $u_{\delta}$ and the last term is lower-semicontinuous as a norm, we conclude that
\begin{align*}
\liminf_{j \to \infty} \mathcal{I}_{\delta,\tau}(y_{k,\delta}) \geq &\liminf_{j \to \infty }  \int_\Omega \delta^{-1}W(\nabla y_{k,\delta}) + \delta^{-p\alpha}P(\nabla^2 y_{k,\delta}) +\tau \delta^{-2} R\left(\nabla y_{k-1,\delta},\tfrac{\nabla y_{\delta}-\nabla y_{k-1,{\delta}}}{\tau}\right) \\ &\qquad -  \tau \frac{1}{{\delta}^2}\int_\Omega  f_{k,\delta}\cdot {\tfrac{y_{\delta}-y_{k-1,\delta}}{\tau}}+ \frac{1}{2h\delta^2} \left| \tfrac{y_{\delta}-y_{k-1,\delta}}{\tau} - w_{k,\delta}\right|^2 \mathrm{d} x \\ &
\geq \liminf_{j \to \infty } \frac{1}{2}\int_\Omega e(u_{k,\delta}) \mathbb{C} \nabla  e(u_{k,\delta}) + e\left(\tfrac{ u_{k,\delta}- u_{k-1,\delta}}{\tau}\right) \mathbb{C_D}(id) e\left(\tfrac{ u_{k,\delta}- u_{k-1,\delta}}{\tau}\right) \mathrm{d}x
\\ &\qquad -\frac{1}{2h}  \left| \tfrac{u_{\delta}-\delta^{-1} (y_{k-1,\delta}-\id)}{\tau} - \delta^{-1}w_{k,\delta}\right|^2 \mathrm{d} x - C \delta \\ &\geq \mathcal{I}_{0,\tau}(u).
\end{align*}

The existence of a recovery sequence on the other hand is a direct consequence of Lemmas \ref{lem:energyRecovery}, \ref{lem:dissipRecovery}, the compact embedding from $W^{1,2}$ to $L^2$ for the quadratic term and the linearity of the force term that is left.
\end{proof}

\section{Discrete iterations and a-priori estimates}
\label{sec:disc-level}

 Within this section, we concentrate on the \emph{fully discrete setting}; that is, when the inertial term has been discretized via the two aforementioned scales and when the solution is constructed by solving a variational problem. Although all the results here are contained in earlier works, we still present them in full, not only for the convenience of the reader, but also because the obtained a-priori estimates turn out to be essential in the next sections.

 We start the discussion by defining the needed constructs. In order to deal with data of potentially very low time-regularity, we will define the following averaged discretization:
  \begin{definition}[Averaged discretization] \label{def:average}
   Let $I = [0,h]$ and $v \in L^2(I\times \Omega)$. Fix $\tau >0$, $N \in \N$ such that $\tau N = h$. Then we define $(v_k^{(\tau)})_{k\in\{1,N\}} \subset L^2(\Omega)$ as
   \begin{align*}
    v_k^{(\tau)} := \fint_{\tau(k-1)}^{\tau k} v \mathrm{d} t.
   \end{align*}
  \end{definition}

  Note that using Jensen's inequality
  \begin{align*}
   \tau \sum_{k=1}^N \norm[L^2(\Omega)]{v_k^{(\tau)}}^2 =\tau \sum_{k=1}^N \norm[L^2(\Omega)]{\fint_{\tau(k-1)}^{\tau k} v \mathrm{d} t}^2 \leq \tau \sum_{k=1}^N \fint_{\tau(k-1)}^{\tau k} \norm[L^2(\Omega)]{ v }^2\mathrm{d} t = \norm[L^2(I\times \Omega)]{v}^2.
  \end{align*}

We also recall the definition of nonlinear discrete cost functional from \eqref{minim-prob-delta}: For $y_{k-1} \in W^{2,p}(\Omega;\R^d)$, $w_k  \in L^{2}(\Omega;\R^d)$ and $f_{k} \in L^2(\Omega;\R^d)$, we will define
 \begin{align*}
\mathcal{I}_{\delta,\tau}(y) := \mathcal{E}_\delta(y) + \tau \mathcal{R}_\delta\left( y_{k-1},\tfrac{ y- y_{k-1}}{\tau}\right) - \tau \frac{1}{\delta^2}\int_\Omega f_{k}\cdot {\tfrac{y-y_{k-1}}{\tau}}+ \frac{\varrho}{2h\delta^{2}} \left| \tfrac{y-y_{k-1}}{\tau} - w_k\right|^2 \mathrm{d} x.
\end{align*}

 \begin{lemma}
 \label{lem-exist-Idelta}
 Under assumptions stated in Section \ref{Sec:potentials}, the nonlinear discrete cost functional $\mathcal{I}_{\delta,\tau}$ has a minimizer.
 \end{lemma}

 \begin{proof}
 The proof of the existence of minimizers in \eqref{minim-prob-delta} is standard and is performed by the direct method. In fact, the assumed coercivity of the energy together with the inertial term provide us (even in the case of a pure traction problem that we do not consider here) with a weakly converging infimizing sequence $\{y_{k,l}\}_{l\in \mathbb{N}}$ such that 
 $$
 y_{k,l} \rightharpoonup y_k \text{ in }  W^{2,p}(\Omega;\R^d).
 $$

 Now, the second-grade contribution to the energy is convex, which yields its lower semicontinuity with respect to the convergence above. Moreover, as $p > d$, we have uniform convergence in the functions themselves as well as in the first gradient. Thus, owing to continuity, we get convergence in all other terms.
 \end{proof}

 We can define a similar functional for the linear problem and note that existence of minimizers follows by the same arguments.

\begin{definition}
  \label{def:solLdiscCost}
    Let $\delta>0$, $\tau> 0$, $h>0$ and consider $u_{k-1} \in W^{1,2}(\Omega;\R^d)$, $w_{k}  \in L^2(\Omega;\R^d)$ and $f_{k}\in L^2(\Omega;\R^d)$. We say that
  \begin{align}
  \nonumber
\mathcal{I}_{0,\tau}(u) := &\int_\Omega  \frac{1}{2} e( u) \cdot \mathbb{C} e ( u) +\frac{1}{2}e\big(\tfrac{ u - u_{k-1}}{\tau}\big)\cdot \mathbb{D} e\big(\tfrac{u - u_{k-1}}{\tau}\big) - \tau \int_\Omega {f}_k\cdot {\tfrac{u-u_{k-1}}{\tau}}+ \frac{\rho}{2h} \left| \tfrac{u-u_{k-1}}{\tau} - w_k\right|^2 \mathrm{d} x
\end{align}    
is the linearized discrete cost functional. 
\end{definition}

 \begin{corollary}
 Under assumptions stated in Section \ref{Sec:potentials}, the linear discrete cost functional has a minimizer.
 \end{corollary}

 Iterating over $k$ for both of these functionals is how we will obtain the discrete approximations to our problem. To begin with this, we first define precisely what we mean by a solution. Note that while the main approximation used in the equation is discrete in time, the optimal handling of the energy inequality already requires us to work with a time-continuous quantity.

 \begin{definition}[Weak solution to the nonlinear discrete problem (NL-disc)]
  \label{def:solNLdisc}
  Let $\tau> 0$ and $h=N\tau$ for some $N \in \mathbb{N}$, so that we have the partition $0 < \tau < 2\tau < \ldots <N\tau = h$ of the interval $[0,h]$. Consider the initial data
  $$
  y^{0,\delta} \in W^{1,2}(\Omega;\mathbb{R}) \quad \text{ with } \quad \mathcal{E}_\delta(y^{0,\delta} ) < +\infty,
  $$
  and  $v^{0,\delta}  \in L^2(\Omega; \R^d)$ as well as a discretized force $(f_{k,\delta})_{k\in \{0, \ldots, N\}} \subset L^2(\Omega; \R^d)$. Moreover, assume that $(w_{k,\delta})_{k\in \{0, \ldots, N\}} \subset L^2(\Omega;\R^d)$ is given.\footnote{The $w_{k,\delta}$ will later be an averaged discretization of $\partial_t y_{\delta}(\cdot-h)$. However at this stage, this uses values out of the interval $[-h,0]$, which is outside of our considerations. So effectively it performs the role of arbitrary given data.}

  We say that the collection $(y_{k,\delta})_{k\in \{0, \ldots, N\}} \subset W^{1,2}(\Omega;\R^d)$ is a weak solution to the nonlinear discrete problem (NL-disc) if $y_{0,\delta} = y^{0,\delta}$,
  \begin{align}
   &\int_\Omega  \frac{\delta^{-1} \varrho}{h}\big(\tfrac{y_{k,\delta}-y_{k-1,\delta}}{\tau}- w_{k,\delta}\big) \cdot \varphi + \delta^{-1} \partial_F W(\nabla y_{k,\delta}) \cdot \nabla \varphi  + \delta^{-\alpha p+1}\partial_G P(\nabla^2 y_{k,\delta}) \cdot \nabla^2 \varphi \, \mathrm{d} x  \nonumber\\
   &+ \int_\Omega \delta^{-1} \partial_{\dot{F}}  R\left(\nabla y_{k-1,\delta}, \nabla \tfrac{y_{k,\delta}-y_{k-1,\delta}}{\tau}\right)\cdot \nabla \varphi  \, \mathrm{d} x = \int_\Omega \delta^{-1}f_{k,\delta} \cdot \varphi \mathrm{d} x
   \label{weakFormNl}
  \end{align}
  holds for all $\varphi \in C^\infty_0(\Omega;\R^d)$ and the energy inequality
  \begin{align}
   &\phantom{{}={}}\mathcal{E}_\delta(y_{k,\delta}) +  \sum_{l=1}^k\int_{(l-1)\tau}^{l\tau} \int_\Omega \frac{\varrho}{2h\delta^2}  \abs{\tfrac{\tilde{y}_{\delta}(t)- y_{l-1,\delta}}{t-(l-1)\tau}}^2  \mathrm{d}x \mathrm{d}t + \sum_{l=1}^k \tau \int_\Omega \frac{\varrho}{2h\delta^2} \abs{\tfrac{y_{l,\delta}-y_{l-1,\delta}}{\tau}-w_{l,\delta}}^2 \mathrm{d} x \nonumber\\
   &+\sum_{l=1}^k \tau \mathcal{R}_\delta\left(y_{l-1,\delta}, \tfrac{y_{l,\delta}-y_{l-1,\delta}}{\tau}\right) + \sum_{l=1}^k\int_{(l-1)\tau}^{l\tau} \mathcal{R}_\delta^*\left( y_{l-1,\delta},D_2\mathcal{R}_\delta\left(y_{l-1,\delta},\tfrac{\tilde{y}_{\delta}(t)- y_{l-1,\delta}}{t-(l-1)\tau}\right)\right)  \,\mathrm{d}t\nonumber \\
   &\leq \mathcal{E}(y_{0,\delta}) + \sum_{l=1}^k \tau \int_\Omega \frac{\varrho}{2h\delta^2} \abs{w_{l,\delta}}^2 \mathrm{d}x +  \sum_{l=1}^k \tau \delta^{-2} \int_\Omega {f_{l,\delta}}\cdot \tfrac{y_{l,\delta}-y_{l-1,\delta}}{\tau} \, \mathrm{d} x
   \label{EnergyIneq-Disc}
  \end{align}
  is satisfied for all $k \in\{ 1, \ldots ,N\}$. Here, $\tilde{y}_\delta(t) \in W^{1,2}(\Omega;\R^d)$ is the so called De Giorgi-interpolant, which is the minimizer of the interpolated cost functional
 \begin{align*}
\mathcal{I}_{\delta,\sigma}(y) := \mathcal{E}_\delta(y) + &\sigma \mathcal{R}_\delta\left( y_{k-1,\delta},\tfrac{ y- y_{k-1,\delta}}{\sigma}\right) \\ &- \sigma \frac{1}{\delta^2}\int_\Omega f_{k,\delta}\cdot {\tfrac{y-y_{k-1,\delta}}{\sigma}}+ \frac{\varrho}{2h\delta^2} \left| \tfrac{y-y_{k-1,\delta}}{\sigma} - w_{k,\delta}\right|^2 \mathrm{d} x
\end{align*}
for $\sigma = t-(k-1)\tau \in (0,\tau]$.
 \end{definition}

 \begin{remark}
  Note that all terms in \eqref{EnergyIneq-Disc} involving the De Giorgi-interpolant are non-negative and on the left hand side of the inequality. In fact dropping them, one has
  \begin{align}
   &\phantom{{}={}}\mathcal{E}_\delta(y_{k,\delta}) + \sum_{l=1}^k \tau \int_\Omega \frac{\varrho}{2h\delta^2} \abs{\tfrac{y_{l,\delta}-y_{l-1,\delta}}{\tau}-w_{k,\delta}}^2 \mathrm{d}x + \mathcal{R}_\delta\left( y_{l-1,\delta}, \tfrac{y_{l,\delta}-y_{l-1,\delta}}{\tau}\right) \nonumber \\
   &\leq \mathcal{E}_\delta(y_{0,\delta}) + \sum_{l=1}^k \tau \int_\Omega \frac{\varrho}{2h\delta^2} \abs{w_{l,\delta}}^2 +  \delta^{-2} {f_{l,\delta}}\cdot \tfrac{y_{l,\delta}-y_{l-1,\delta}}{\tau} \, \mathrm{d} x
   \label{energy-est-simplified}
  \end{align}
  which is enough of an estimate to proceed with the limit $\tau\to0$. Estimate \eqref{energy-est-simplified} can be obtained by comparing the minimizer $y_{k,\delta}$ with the competitor $y_{k-1,\delta}$ in the minimization of $\mathcal{I}_{\delta,\tau}$. However, in terms of the $h$-dependency, this simplified estimate no longer telescopes when appending intervals $[0,h]$, $[h,2h]$ and so on. Therefore, the refined estimate \eqref{EnergyIneq-Disc} will turn out to be necessary. 
 \end{remark}

  \begin{proposition}\label{prop:NL-discExistence}
 For any $h>0$ the problem (NL-disc) has a weak solution in the sense of Definition \ref{def:solNLdisc}.
  \end{proposition}

 \begin{proof}
  As both the equation and the energy inequality are stated for each $k$ separately, we can assume that the result is true for $k-1$ and prove the result by induction. In order to prove \eqref{weakFormNl}, we minimize the nonlinear discrete cost functional $\mathcal{I}_{\delta,\tau}$ from \eqref{minim-prob-delta}. We know from Lemma \ref{lem-exist-Idelta} that minimizers of $\mathcal{I}_{\delta,\tau}$ exist. It is also easy to see that the resulting Euler-Lagrange equations are precisely the correct weak equations.


  What is left is the energy-inequality. To prove this, we rely on a rather general argument due to De Giorgi and introduce an additional interpolant. For a classic treatment of this approach see also \cite{AGS} and for a similar application in the continuum mechanics of solids compare also \cite{CGKLipschitz}.

  Fix $\tau$ and $k$. We consider an interpolated cost functional
 \begin{align*}
\mathcal{I}_{\delta,\sigma}(y) := \mathcal{E}_\delta(y) + &\sigma \mathcal{R}_\delta\left( y_{k-1,\delta},\tfrac{ y- y_{k-1,\delta}}{\sigma}\right) - \sigma \frac{1}{\delta^2}\int_\Omega f_{k}\cdot {\tfrac{y-y_{k-1,\delta}}{\sigma}}+ \frac{\varrho}{2h\delta^2} \left| \tfrac{y-y_{k-1,\delta}}{\sigma} - w_{k,\delta}\right|^2 \mathrm{d} x
\end{align*}
where $\sigma \in (0,\tau]$. Note that we are varying in a single time step, $y_{k-1,\delta}$ and $f_{k,\delta}$ stay the same discrete approximation for a fixed timestep $\tau$.  By Lemma \ref{lem-exist-Idelta}, $\mathcal{I}_{\delta,\sigma}$ has a minimizer $y_{\sigma}$ for any such $\sigma$.

Next we want to prove that $\sigma \mapsto \mathcal{I}_{\delta,\sigma}(y_\sigma)$ is absolutely continuous and we want to determine its derivative (Note that we require no continuity of $y_\sigma$ itself). For this consider $\sigma_1, \sigma_2 \in (\sigma_0,\tau]$ for some $\sigma_0>0$ and calculate
\begin{align*}
 &\phantom{{}={}} \mathcal{I}_{\delta,\sigma_1}(y_{\sigma_1}) - \mathcal{I}_{\delta,\sigma_2}(y_{\sigma_2}) \leq \mathcal{I}_{\delta,\sigma_1}(y_{\sigma_2}) - \mathcal{I}_{\delta,\sigma_2}(y_{\sigma_2}) \\
 &= \sigma_1 \mathcal{R}_\delta\left( y_{k-1,\delta},\tfrac{ y_{\sigma_2}- y_{k-1,\delta}}{\sigma_1}\right) - \sigma_2 \mathcal{R}_\delta\left( y_{k-1,\delta},\tfrac{ y_{\sigma_2}- y_{k-1,\delta}}{\sigma_2}\right) \\
 &\quad + \int_\Omega \tfrac{\sigma_1\varrho}{2h\delta^2} \left| \tfrac{y_{\sigma_2}-y_{k-1,\delta}}{\sigma_1} - w_{k,\delta}\right|^2 \mathrm{d} x - \int_\Omega \tfrac{\sigma_2\varrho}{2h\delta^2} \left| \tfrac{y_{\sigma_2}-y_{k-1,\delta}}{\sigma_2} - w_{k,\delta}\right|^2 \mathrm{d} x
\end{align*}
Now for the first two terms we use the convexity of $\mathcal{R}_\delta$ and the equality case of Fenchel-Young to estimate
 \begin{align*}
  &\phantom{{}={}} \sigma_1 \mathcal{R}_\delta\left(y_k,\tfrac{y_{\sigma_2}-y_k}{\sigma_1}\right) - \sigma_2 \mathcal{R}_\delta\left(y_k,\tfrac{y_{\sigma_2}-y_k}{\sigma_2}\right) \\
  &= (\sigma_1-\sigma_2) \mathcal{R}_\delta\left(y_k,\tfrac{y_{\sigma_2}-y_k}{\sigma_1}\right) - \sigma_2 \left[\mathcal{R}_\delta\left(y_k,\tfrac{y_{\sigma_2}-y_k}{\sigma_2}\right)- \mathcal{R}_\delta\left(y_k,\tfrac{y_{\sigma_2}-y_k}{\sigma_1}\right)\right] \\
  &\leq (\sigma_1-\sigma_2) \mathcal{R}_\delta\left(y_k,\tfrac{y_{\sigma_2}-y_k}{\sigma_1}\right) - \sigma_2\inner{D_2\mathcal{R}_\delta\left(y_k,\tfrac{y_{\sigma_2}-y_k}{\sigma_1}\right)}{\tfrac{y_{\sigma_2}-y_k}{\sigma_2}- \tfrac{y_{\sigma_2}-y_k}{\sigma_1}} \\
  &\leq (\sigma_1-\sigma_2) \mathcal{R}_\delta\left(y_k,\tfrac{y_{\sigma_2}-y_k}{\sigma_1}\right) - (\sigma_1-\sigma_2)\inner{D_2\mathcal{R}_\delta\left(y_k,\tfrac{y_{\sigma_2}-y_k}{\sigma_1}\right)}{\tfrac{y_{\sigma_2}-y_k}{\sigma_1}} \\
  &\leq -(\sigma_1-\sigma_2) \mathcal{R}_\delta^*\left(y_k,D_2\mathcal{R}_\delta\left(y_k,\tfrac{y_{\sigma_2}-y_k}{\sigma_1}\right)\right)
 \end{align*}
 while for the last two we have by direct calculation that
 \begin{align*}
   \tfrac{\sigma_1\varrho}{2h\delta^2} \left| \tfrac{y_{\sigma_2}-y_{k-1,\delta}}{\sigma_1} - w_{k,\delta}\right|^2 - \tfrac{\sigma_2\varrho}{2h\delta^2} \left| \tfrac{y_{\sigma_2}-y_{k-1,\delta}}{\sigma_2} - w_{k,\delta}\right|^2 =  (\sigma_2-\sigma_1) \left[ \tfrac{\varrho}{2h\delta^2}\tfrac{\left| y_{\sigma_2}-y_{k-1,\delta} \right|^ 2}{\sigma_1\sigma_2} +  \tfrac{\varrho}{2h\delta^2}\abs{w_{k,\delta}}^2 \right].
 \end{align*}
 We thus have
 \begin{align*}
  &\phantom{{}={}} \mathcal{I}_{\delta,\sigma_1}(y_{\sigma_1}) - \mathcal{I}_{\delta,\sigma_2}(y_{\sigma_2}) \leq (\sigma_2-\sigma_1) \left[ \mathcal{R}_\delta^*\left(y_k,D_2\mathcal{R}_\delta\left(y_k,\tfrac{y_{\sigma_2}-y_k}{\sigma_1}\right)\right) + \tfrac{\varrho}{2h\delta^2}\tfrac{\left| y_{\sigma_2}-y_{k-1,\delta} \right|^ 2}{\sigma_1\sigma_2} +  \int_\Omega \tfrac{\varrho}{2h\delta^2}\abs{w_{k,\delta}}^2\mathrm{d}x \right]
 \end{align*}
and similarly
\begin{align*}
 &\phantom{{}={}} \mathcal{I}_{\delta,\sigma_1}(y_{\sigma_1}) - \mathcal{I}_{\delta,\sigma_2}(y_{\sigma_2}) \geq \mathcal{I}_{\delta,\sigma_1}(y_{\sigma_1}) - \mathcal{I}_{\delta,\sigma_2}(y_{\sigma_1}) \\
 &= (\sigma_2-\sigma_1) \left[ \mathcal{R}_\delta^*\left(y_k,D_2\mathcal{R}_\delta\left(y_k,\tfrac{y_{\sigma_1}-y_k}{\sigma_2}\right)\right) + \tfrac{\varrho}{2h\delta^2}\tfrac{\left| y_{\sigma_1}-y_{k-1,\delta} \right|^ 2}{\sigma_1\sigma_2} +  \int_\Omega \tfrac{\varrho}{2h\delta^2}\abs{w_{k,\delta}}^2\mathrm{d}x \right].
\end{align*}

Dividing by $\sigma_2-\sigma_1$ gives an upper and a lower bound on the Lipschitz-constant of $\sigma \mapsto \mathcal{I}_{\delta,\sigma}(y_\sigma)$ in the interval $(\sigma_0,\tau]$, which proves the absolute continuity. Now fixing  $\sigma = \sigma_1$ in the last equation and sending $\sigma_2 \searrow \sigma_1$, gives
\begin{align*}
 - \tfrac{d}{d\sigma} \mathcal{I}_{\delta,\sigma}(y_{\sigma}) &\geq \mathcal{R}_\delta^*\left(y_k,D_2\mathcal{R}_\delta\left(y_k,\tfrac{y_{\sigma}-y_k}{\sigma}\right)\right) + \int_\Omega \tfrac{\varrho}{2\sigma^2h\delta^2}\left| y_{\sigma}-y_{k-1,\delta} \right|^ 2 +  \tfrac{\varrho}{2h\delta^2}\abs{w_{k,\delta}}^2 \mathrm{d} x
\end{align*}
for a.a.\ $\sigma \in (0,\tau]$.

Additionally we have the following uniform estimate
\begin{align*}
\mathcal{I}_{\delta,\sigma}(y_\sigma) \leq \mathcal{I}_{\delta,\sigma}(y_{k-1,\delta}) = \mathcal{E}_\delta(y_{k-1,\delta}) + \int_\Omega \tfrac{\sigma\varrho}{2h\delta^2} \left| w_{k,\delta}\right|^2 \mathrm{d} x
\end{align*}
If we multiply this with $\sigma$, we obtain that
\begin{align*}
 &\phantom{{}={}}\int_\Omega \frac{1}{\delta^2} R\left(\nabla y_{k-1,\delta},\nabla y_\sigma-\nabla y_{k-1,\delta}\right) \mathrm{d}x -\sigma \frac{1}{\delta^2}\int_\Omega f_{k}\cdot (y_\sigma-y_{k-1,\delta})  \\&= \frac{\sigma^2}{\delta^2} \int_\Omega R\left(\nabla y_{k-1,\delta},\tfrac{\nabla y-\nabla y_{k-1,\delta}}{\sigma}\right) \mathrm{d}x -\sigma^2 \frac{1}{\delta^2}\int_\Omega f_{k}\cdot \frac{y_\sigma-y_{k-1,\delta}}{\sigma}
 \leq \sigma \mathcal{I}_{\delta,\sigma}(y_\sigma) \to 0
\end{align*}
for $\sigma \to 0$. Thus, using Lemma \ref{lem:KornGeneral} and Young's inequality, we have $\lim_{\sigma \to 0} y_\sigma = y_{k-1,\delta}$ in $W^{1,2}(\Omega;\R^d)$.

With this in hand, we have
\begin{align*}
 &\phantom{{}={}}-\int_0^\tau \mathcal{R}_\delta^*\left(y_k,D_2\mathcal{R}_\delta\left(y_k,\tfrac{y_{\sigma}-y_k}{\sigma}\right)\right) + \int_\Omega \tfrac{\varrho}{2\sigma^2h\delta^2}\left| y_{\sigma}-y_{k-1,\delta} \right|^ 2 - \tfrac{\rho}{2h\delta^2}\abs{  w_{k,\delta}}^2 \mathrm{d}x \mathrm{d} \sigma \\
 &\geq \int_0^\tau  \tfrac{d}{d\sigma} \mathcal{I}_{\delta,\sigma}(y_{\sigma})\, \mathrm{d}t = \mathcal{I}_{\delta,\tau}(y_\tau) - \lim_{\sigma \to 0} \mathcal{I}_{\delta,\sigma}(y_\sigma) \\
 &\geq \mathcal{E}_\delta(y_{\tau}) + \tau \mathcal{R}_\delta\left(\nabla y_{k-1,\delta},\tfrac{\nabla y_\tau-\nabla y_{k-1,\delta}}{\tau}\right) \\ &- \tau \tfrac{1}{\delta^2}\int_\Omega f_{k}\cdot {\tfrac{y_\tau-y_{k-1,\delta}}{\tau}}+ \tfrac{\tau\varrho}{2h\delta^2} \left| \tfrac{y_\tau-y_{k-1,\delta}}{\tau} - w_{k,\delta}\right|^2 \mathrm{d} x - \mathcal{E}_\delta(y_{k-1,\delta})
\end{align*}
  which can be rewritten into the discrete energy-inequality:
\begin{align} \label{eq:discrEnergyIneq}
 &\phantom{{}={}}  \mathcal{E}_\delta(y_{\tau}) + \tfrac{\varrho}{2h\delta^2}\int_0^\tau \left| \tfrac{y_{\sigma}-y_{k-1,\delta}}{\sigma} \right|^ 2 \mathrm{d} \sigma \nonumber \\
 &+\int_0^\tau \mathcal{R}_\delta^*\left(y_k,D_2\mathcal{R}_\delta\left(y_k,\tfrac{y_{\sigma}-y_k}{\sigma}\right)\right) + \mathcal{R}_\delta\left( y_{k-1,\delta},\tfrac{ y_\tau- y_{k-1,\delta}}{\tau}\right)\mathrm{d} \sigma \nonumber \\ &+ \int_\Omega \tfrac{\tau\varrho}{2h\delta^2} \left| \tfrac{y_\tau-y_{k-1,\delta}}{\tau} - w_{k,\delta}\right|^2 \mathrm{d} x \nonumber \\
 &\leq \mathcal{E}_\delta(y_{k-1,\delta}) + \int_\Omega\tfrac{\tau\varrho}{2h\delta^2} \abs{w_{k,\delta}}^2 + \tau \tfrac{1}{\delta^2} f_{k,\delta}\cdot {\tfrac{y_\tau-y_{k-1,\delta}}{\tau}} \mathrm{d} x
\end{align}

  Identifying $y_\tau$ with $y_{k,\delta}$ and summing over $k$ this then yields \eqref{EnergyIneq-Disc}.
 \end{proof}

 In the same way, we define the solution to the linear discrete problem.

  \begin{definition}[Weak solution to the linear discrete problem (L-disc)]
  \label{def:solLdisc}
  Let $\tau> 0$ and $h=N\tau$ for some $N \in \mathbb{N}$, so that we have the partition $0 < \tau < 2\tau < \ldots < N\tau = h$ of the interval $[0,h]$. Consider the initial data
  $$
  u^{0} \in W^{1,2}(\Omega;\mathbb{R}) \text{ with } \mathcal{E}_0(u_0) < +\infty,
  $$
  and  $v_0 \in L^2(\Omega; \R^d)$ as well as a discretized force $(f_k)_{k\in \{0 ,\dots, N\}} \subset L^2(\Omega)$. Moreover, assume that $(w_k)_{k\in\{1,\ldots, N\}} \subset L^2(\Omega; \R^d)$ is given.

  We say that the collection $(u_{k})_{k\in \{0 ,\dots, N\}} \subset W^{1,2}(\Omega;\R^d)$ is a weak solution to the linear discrete problem (NL-disc) if $u_{0} = u^{0}$ and
  \begin{align*}
   &\int_\Omega \frac{\varrho}{h} \big(\tfrac{u_{k}-u_{k-1}}{\tau}- w_k \big) \cdot \varphi + \frac{1}{2}\mathbb{C} e( u_k) \cdot e(\varphi)  + \frac{1}{2}\mathbb{D} e\big(\tfrac{u_{k}-u_{k-1}}{\tau}\big)\cdot e(\varphi)   \, \mathrm{d} x = \int_\Omega f_k  \cdot \varphi \, \mathrm{d} x
  \end{align*}
  for all $\varphi \in C^\infty_0(\Omega;\R^d)$, as well as the energy inequality
    \begin{align}
   &\phantom{{}={}}\mathcal{E}_0(u_{k}) +  \sum_{l=1}^k\int_{(l-1)\tau}^{l\tau} \int_\Omega \frac{\varrho}{2h} \abs{\tfrac{\tilde{u}(t)- u_{l-1}}{t-(l-1)\tau}}^2 \mathrm{d}x \mathrm{d}t + \sum_{l=1}^k \tau \int_\Omega \frac{\varrho}{2h} \abs{\tfrac{u_{l}-u_{l-1}}{\tau}-w_l}^2 \mathrm{d} x \nonumber\\
   &+\sum_{l=1}^k \tau \int_\Omega \mathbb{D} e\big(\tfrac{u_{l}-u_{l-1}}{\tau}\big)\cdot  e\left(\tfrac{u_{l}-u_{l-1}}{\tau} \right)\, \mathrm{d} x + \sum_{l=1}^k\int_{(l-1)\tau}^{l\tau} \int_\Omega \mathbb{D} e\big(\tfrac{\tilde{u}(t)-u_{l-1}}{t-(l-1)\tau}\big)\cdot  e\left(\tfrac{\tilde{u}(t)-u_{l-1}}{t-(l-1)\tau}\right) \mathrm{d}x \mathrm{d}t\nonumber \\
   &\leq \mathcal{E}_0(u_0) + \frac{\tau\varrho}{2h}\sum_{l=1}^k \norm[L^2(\Omega)]{w_l}^2 + \tau \sum_{l=1}^k \int_\Omega {f_k}\cdot \tfrac{u_{l}-u_{l-1}}{\tau} \, \mathrm{d} x
   \label{EnergyIneq-LDisc}
  \end{align}
  for all $k \in \{0, \ldots, N\}$, where $\tilde{u}(t) \in W^{1,2}(\Omega;\R^d)$ is again a De Giorgi-interpolant, in this case the minimizer of the interpolated cost functional
 \begin{align*}
\mathcal{I}_{0,\sigma}(u) := \mathcal{E}_0(u) + &\sigma \int_\Omega \mathbb{D} e\big(\tfrac{u-u_{k-1}}{\sigma\tau}\big)\cdot  e\left(\tfrac{u-u_{k-1}}{\sigma}\right) \mathrm{d}x \\ &- \sigma \int_\Omega  f_{k}\cdot {\tfrac{u-u_{k-1}}{\sigma}}+ \frac{\varrho}{2h} \left| \tfrac{u-u_{k-1}}{\sigma} - w_k\right|^2 \mathrm{d} x
\end{align*}
for $\sigma = t-(k-1)\tau \in (0,\tau]$.
 \end{definition}

 \begin{corollary} \label{cor:L-discExistence}
 For any $h>0$ the problem (L-disc) has a unique weak solution in the sense of Definition \ref{def:solNLdisc}.
 \end{corollary}

 \begin{proof}
  With the same arguments as in Proposition \ref{prop:NL-discExistence} we get existence of solutions. For the uniqueness we proceed by a standard argument using the linearity of the equation. Specifically, the difference of two solutions $(u_k)_k$ and $(v_k)_k$ for the same force and initial data will satisfy the weak equation for zero force and initial data. Testing this equation with the difference then gives that $e(u_k-v_k) = 0$ almost everywhere for all $k$, which using the boundary data implies $u_k = v_k$.
 \end{proof}

\section{Passing to the limit from (NL-disc) or (L-disc)}
\label{sec:disc-limits}

Next we will use the energy inequality inherent to the discrete problem in order to obtain the limits for $\tau \to 0$ and $\delta \to 0$ directly on the level of equations.

\subsection{The limit \texorpdfstring{$\tau \to 0$}{tau to zero} from (NL-disc) and (L-disc)}
 We pass to the limit $\tau \to 0$ to arrive at the time-delayed problem (NL-TD), which is time-continuous but where the inertial term is discretized via the $h$-scale through a difference quotient. We start with a definition of the weak solution of the nonlinear time-delayed problem:

 \begin{definition}[Weak solution to the nonlinear time-delayed problem (NL-TD)]
  \label{def:solNLTD}
  Let $h>0$, $I = [0,T]$ and consider the initial data
$$
y_{0,\delta} \in W^{1,2}(\Omega; \R^d) \text{ with }\mathcal{E}_\delta(y_{0,\delta}) < +\infty \text{ and } v_{0,\delta} \in L^2(\Omega)
$$
and a force $f_\delta \in L^2(I \times \Omega)$. We say that $y_{h,\delta} \in W^{1,2}(I;W^{1,2}(\Omega;\R^d)) \cap L^\infty(I;W^{2,p}(\Omega;\R^d))$ is a weak solution to the nonlinear time-delayed problem (NL-TD) if $y_{h,\delta}(0) = y_{0,\delta}$ and $y_{h,\delta}$ satisfies
  \begin{align*}
   0 &= \int_0^T \int_\Omega \frac{\delta^{-1}\varrho}{h}\big(\partial_t y_{h,\delta}(t)-\partial_t y_{h,\delta}(t-h)\big)\cdot \varphi + \delta^{-1} \partial_F W(\nabla y_{h,\delta}) \cdot \nabla \varphi + \delta^{-\alpha p-1} \partial_G P(\nabla^2 y_{h,\delta}) \cdot \nabla^2 \varphi \\
   & \qquad + \delta^{-1} \partial_{\dot{F}} R(\nabla y_{h,\delta}, \nabla \partial_t y_{h,\delta}) \cdot \nabla \varphi - \delta^{-1} f_\delta \cdot \varphi \mathrm{d}x \mathrm{d}t
  \end{align*}
  for all $\varphi \in C^\infty_c((0,T) \times \Omega)$ where we define by abuse of notation $\partial_t y_{h,\delta}(t) := v_{0,\delta}$ for $t<0$, as well as the energy inequality
  \begin{align*}
   &\phantom{{}={}}\mathcal{E}_\delta(y_{h,\delta}(s)) + \frac{\varrho}{\delta^2 2} \fint_{s-h}^s \norm[L^2(\Omega)]{\partial_t y_{h,\delta}(t)}^2 \mathrm{d}t + \int_0^s \frac{1}{\delta^2} \int_\Omega \partial_{\dot{F}} R(\nabla y_{h,\delta}, \partial_t \nabla y_{h,\delta}) \cdot\partial_t \nabla y_{h,\delta} \mathrm{d} t \\
   &+ \int_0^s \frac{\rho}{2h\delta^2} \norm[L^2(\Omega)]{\partial_t y_{h,\delta}(t)-\partial_t y_{h,\delta}(t-h)}^2 \mathrm{d}t \leq \mathcal{E}_\delta(y_0) + \frac{\varrho}{2 \delta^2}\norm[L^2(\Omega)]{v_0}^2 + \int_0^s \frac{1}{\delta^2}\inner[L^2]{f_\delta}{\partial_t y_{h,\delta}} \mathrm{d}t
  \end{align*}
  for almost all $s \in [0,T]$.
 \end{definition}

 The first main result for this section is the following existence result.
 \begin{proposition} \label{prop:NLTDexistence}
  The problem (NL-TD) has a weak solution for any $T \in (0,\infty)$ and initial data/forces as specified in Definition \ref{def:solNLTD}.
 \end{proposition}

 We will prove the result iteratively on intervals of length $h$. Proving this also involves some definitions and estimates that will be of independent use later on. We thus state them separately. The following definition will be our standard way to construct continuous approximations out of discrete counterparts.
 \begin{definition}[Piecewise constant and affine interpolants] \label{def:interpolants}
  Fix $\tau >0$ and let $(v_k)_{k\in\{0,\dots,N\}} \subset X$, where $X$ is a given Banach space. Then we define the upper and lower piecewise constant as well as the piecewise affine interpolations as functions from $I = (0,\tau N]$ to $X$ by
  \begin{align*}
   \overline{v}^{(\tau)}(t) &:= v_k  &\text{ for } & (k-1)\tau < t \leq k\tau, \\
   \underline{v}^{(\tau)}(t) &:= v_{k-1}  &\text{ for } & (k-1)\tau < t \leq k\tau, \\
   \hat{v}^{(\tau)}(t) &:= (1-\tfrac{t-(k-1)\tau}{\tau} ) v_{k-1} +  \tfrac{t-(k-1)\tau}{\tau}  v_{k} &\text{ for } & (k-1)\tau < t \leq k\tau. \\
  \end{align*}
 \end{definition}

 Note that 
  \begin{align*}
   \partial_t \hat{v}^{(\tau)}(t) &= \tfrac{ v_{k} -v_{k-1}}{\tau}  &\text{ for } & (k-1)\tau < t < k\tau
  \end{align*}
  and we can conclude directly from the definition that
  \begin{align*}
   \norm[L^\infty(I;X)]{\overline{v}^{(\tau)}} = \sup_{k\in\{1,\dots,N\}} \norm[X]{v_k}, \quad \norm[L^\infty(I;X)]{\underline{v}^{(\tau)}} = \sup_{k\in\{0,\dots,N-1\}} \norm[X]{v_k}\\
   \norm[L^\infty(I;X)]{\hat{v}^{(\tau)}} \leq \sup_{t\in I} \max\left( \norm[X]{\overline{v}^{(\tau)}(t)},\norm[X]{\underline{v}^{(\tau)}(t)} \right) \leq \sup_{k\in\{0,\dots,N\}} \norm[X]{v_k}
  \end{align*}
  as well as
  \begin{align*}
   \norm[L^2(I;X)]{\partial_t \hat{v}^{(\tau)}}^2 = \tau \sum_{k=1}^N \norm[X]{\tfrac{ v_{k} -v_{k-1}}{\tau}}^2.
  \end{align*}
  Note also that this forms in some sense the inverse to the averaged discretization from Definition \ref{def:average}.

  \begin{lemma}[Equiboundedness of the interpolants] \label{lem:interpolantBounds}
  Fix $I=[0,h]$. Let $\delta > 0$ and $\tau >0$, as well as $N$ such that $\tau N = h$. Then there exists a constant $C$ independent of $\delta$ and $\tau$ such that for every weak solution $(y_{k,\delta})_{k \in \{0, \ldots, N\}} \subset W^{1,2}(\Omega; \mathbb{R}^d)$ to (NL-disc) with initial data $y_0$, force $(f_k^{(\tau)})_k$ and $(w^{(\tau)}_k)_k$ we have
  \begin{align*}
   &\norm[L^\infty(I;W^{1,2}(\Omega))]{\overline{y_{\delta}}^{(\tau)}-\id}^2 +\norm[L^\infty(I;W^{1,2}(\Omega))]{\underline{y_{\delta}}^{(\tau)}-\id}^2 +\norm[L^\infty(I;W^{1,2}(\Omega))]{\hat{y}_{\delta}^{(\tau)}-\id}^2 + \norm[L^2(I;W^{1,2}(\Omega))]{\partial_t \hat{y_{\delta}}^{(\tau)}}^2 \\
   &\leq C \left(\norm[L^2(I\times \Omega)]{f}^2+ \norm[L^2(I\times \Omega)]{w}^2 + \mathcal{E}_{\delta}(y_0)\right)
  \end{align*}
  for the respective interpolants as well as
  \begin{align*}
  \norm[L^\infty(I\times \Omega)]{\delta^{-1} \nabla \underline{u}_\delta(t)} + \sup_{t\in I}\norm[W^{2,p}(\Omega)]{\delta^{-1}\overline{u}_\delta(t)} \leq C \sqrt{\norm[L^2(I\times \Omega)]{f}^2+ \norm[L^2(I\times \Omega)]{w}^2 + \mathcal{E}_{\delta}(y_0)}
  \end{align*}
  \end{lemma}

  \begin{proof} 
Recall that we have the energy inequality \eqref{EnergyIneq-Disc}, which we now rewrite in term of the interpolation-functions above
  \begin{align*}
   &\phantom{{}={}}\mathcal{E}_\delta(y_{k,\delta}) + \frac{\varrho}{2h\delta^2}  \sum_{l=1}^k\int_{(l-1)\tau}^{l\tau} \norm[L^2(\Omega)]{\tfrac{\tilde{y}_{\delta}(t)- y_{l-1,\delta}}{t-(l-1)\tau}}^2 \,\mathrm{d}t + \frac{\tau\varrho}{2h\delta^2}\sum_{l=1}^k \norm[L^2(\Omega)]{\tfrac{y_{l,\delta}-y_{l-1,\delta}}{\tau}-w^{(\tau)}_k}^2 \nonumber\\
   &+\tau\delta^{-2}\sum_{l=1}^k \int_\Omega R\left(\nabla y_{l-1,\delta}, \nabla \tfrac{y_{l,\delta}-y_{l-1,\delta}}{\tau}\right)\, \mathrm{d} x + \delta^{-2}\sum_{l=1}^k\int_{(l-1)\tau}^{l\tau} \int_\Omega R\left(\nabla y_{l-1,\delta},\tfrac{\nabla\tilde{y}_{\delta}(t)-\nabla y_{l-1,\delta}}{t-(l-1)\tau}\right) \,\mathrm{d}x \,\mathrm{d}t\nonumber \\
   &\leq \mathcal{E}_\delta(y_{0,\delta}) + \frac{\tau\varrho}{2h\delta^2}\sum_{l=1}^k \norm[L^2(\Omega)]{w^{(\tau)}_k}^2 + \delta^{-2}\tau \sum_{l=1}^k \int_\Omega {f_k}\cdot \tfrac{y_{l,\delta}-y_{l-1,\delta}}{\tau} \, \mathrm{d} x.
  \end{align*}
  We can then use the energy inequalities to get bounds on
  \begin{align*}
   \sup_{t \in I} \mathcal{E}_\delta(\bar{y}^{(\tau)}), \quad \sup_{t\in I} \frac{1}{h}\int_{t-h}^t \norm{\partial_t \hat{y}^{(\tau)}}^2 \mathrm{d}t \quad \text{ and } \quad \int_0^T R(\nabla \underline{y}^{(\tau)}, \nabla \partial_t \hat{y}^{(\tau)}) \mathrm{d}t
  \end{align*}
  which only depend on the initial data. From this we can derive the uniform bounds using Lemma \ref{lem:rigidity} and Lemma \ref{lem:KornCont} as well as the previous remarks.
  \end{proof}

For the rest of this section as well as in any other section where $\delta$ is fixed, we will omit it as an index and write $\bar{y}$ instead of $\bar{y}_\delta$ and so on.

\begin{lemma}[Minty trick] \label{lem:minty}
 Let $(y_k)_{k\in\N} \subset W^{2,p}(\Omega;\R^d)$ with $y_k \rightharpoonup y$ in $W^{2,p}(\Omega;\R^d)$ and assume that
 \begin{align*}
  \lim_{k \to \infty} \inner{\partial_G P(\nabla^2 y_k)}{\nabla^2(y-y_k)} = 0.
 \end{align*}
 Then $\partial_G P(\nabla^2 y_k) \rightharpoonup \partial_G P(\nabla^2 y)$ in $L^{p'}(\Omega;\R^{d\times d\times d})$.
\end{lemma} 

\begin{proof}
 Note that we required $P$ to be convex in our assumptions. This means that $y \mapsto \int_\Omega P(\nabla^2 y) \mathrm{d}x$ is convex and thus that its Frechet-derivative $w \mapsto \inner{\partial_G P(\nabla^2 y)}{\nabla^2 w}$ is a maximal monotone operator. Since $y_k$ converges, it is bounded in $W^{2,p}(\Omega;\R^d)$. But then $\partial_G P(\nabla^2 y_k)$ is bounded in $L^{p'}(\Omega;\R^{d\times d\times d})$ and (after  choosing a subsequence) has a weak limit $\alpha$. But then we have for any $y' \in W^{2,p}(\Omega;\R^d)$
 \begin{align*}
  &\phantom{{}={}}\inner{\partial_G P(\nabla^2 y')-\alpha}{y'-y} = \lim_{k \to \infty} \inner{\partial_G P(\nabla^2 y')-\partial_G P(\nabla^2 y_k)}{y'-y} \\
  &= \lim_{k \to \infty} \inner{\partial_G P(\nabla^2 y')-\partial_G P(\nabla^2 y_k)}{y'-y_k} - \lim_{k \to \infty} \inner{\partial_G P(\nabla^2 y_k)}{y-y_k} \geq 0
 \end{align*}
 where the first term is non-negative by monotonicity and the second as per our assumption. But then maximal monotonicity implies $\alpha = \partial_G P(\nabla^2 y)$. As the same holds for any subsequence we thus have $\partial_G P(\nabla^2 y_k) \rightharpoonup \partial_G P(\nabla^2 y)$.
\end{proof}

 \begin{proof}[Proof of Proposition \ref{prop:NLTDexistence}]
 We begin by first constructing a solution on the interval $[0,h]$.
We recall that in Section \ref{sec:disc-level}, in Proposition \ref{prop:NL-discExistence} for any $h>0$ and any $\tau>0$ such that $N \tau = h$, we constructed a collection of functions $(y_{k, \delta}^{(\tau)})_{k} \subset W^{1,2}(\Omega; \mathbb{R}^d)$ that is a weak solution to (NL-disc) according to Definition \ref{def:solNLdisc}. Notice, that in this proof we also include the additional superscript "$\tau$" in order to stress the dependence of the solution on this variable. Indeed, it becomes crucial in this section as we wish to pass to the limit $\tau \to 0$.

For this solution, we set $(f_{k,\delta})_k$ as an averaged discretization of $f_{\delta}$ and $(w_{k,\delta})_k$ as an averaged discretization of $\partial_t y_\delta(\cdot -h)$ (possibly extended by initial velocity). We then define the piecewise constant and piecewise affine approximations as in Definition \ref{def:interpolants}. Using the uniform bounds obtained using Lemma \ref{lem:interpolantBounds} we can then find a compact subsequence (not relabeled) and a limit $y \in W^{1,2}([0,h];W^{1,2}(\Omega;\R^d)) \cap L^\infty([0,h];W^{2,p}(\Omega;\R^d))$ for which
  \begin{align*}
   \bar{y}^{(\tau)},\underline{y}^{(\tau)},\hat{y}^{(\tau)} &\rightharpoonup y & \text{ in } & L^2([0,h];W^{2,p}(\Omega;\R^d)) \\
   \hat{y}^{(\tau)} &\rightharpoonup y & \text{ in } & W^{1,2}([0,h];W^{1,2}(\Omega;\R^d)),
   \intertext{as well as using Aubin-Lions theorem,}
   \bar{y}^{(\tau)},\underline{y}^{(\tau)},\hat{y}^{(\tau)} &\to y & \text{ in } & L^2([0,h];W^{1,2}(\Omega;\R^d)).
  \end{align*}

  Note that the respective limits need to coincide since
  \begin{align*}
   \norm[W^{1,2}(\Omega)]{y_k^{(\tau)}-y_{k-1}^{(\tau)}}^2 \leq \tau  \int_{(k-1)\tau}^{k\tau} \norm[W^{1,2}(\Omega)]{\partial_t\hat{y}^{(\tau)}}^2 \mathrm{d}t \leq \tau  \int_I \norm[W^{1,2}(\Omega)]{\partial_t\hat{y}^{(\tau)}}^2 \mathrm{d}t \leq C \tau \to 0
  \end{align*}
  as the latter integral is uniformly bounded.

  If we now take a time-dependent test-function $\varphi$ and integrate over the respective weak equations of (NL-disc) for $k$ such that $t \in [\tau k,\tau(k+1))$ we obtain
   \begin{align} \nonumber
   0 &= \int_0^h \frac{\varrho}{h\delta}\inner[L^2]{\partial_t \hat{y}^{(\tau)}(t)-\partial_t y^{(\tau)}(t-h)}{\varphi} + \inner[L^2]{\delta^{-1}\partial_F W(\nabla \bar{y}^{(\tau)})}{\nabla \varphi} + \inner[L^2]{\delta^{-p\alpha+1}\partial_G P(\nabla^2 \bar{y}^{(\tau)})}{\nabla^2 \varphi} \\ \label{eq:rewrittenDiscWeak}
   &+ \inner[L^2]{\delta^{-1}\partial_{\dot{F}} R(\nabla \underline{y}^{(\tau)}, \nabla \partial_t \hat{y}^{(\tau)})}{\nabla \varphi} - \inner[L^2]{\delta^{-1}f_\delta}{\varphi} \mathrm{d}t.
  \end{align}
  Using the above limits and the properties of $W$ and $R$, most of the terms above converge to their corresponding counterpart in the weak equation for (NL-TD).

  The only term for which this does not immediately follow is $\inner[L^2]{\partial_G P(\nabla^2 \bar{y}^{(\tau)})}{\nabla^2 \varphi}$. For this, we note that setting $\varphi = y-\bar{y}^{(\tau)}$ in \eqref{eq:rewrittenDiscWeak} and letting $\tau \to 0$
  \begin{align*}
   &\phantom{{}={}}\lim_{\tau \to 0} \int_0^h \inner{\delta^{-p\alpha+1}\partial_G P(\nabla^2 \bar{y}^{(\tau)})}{\nabla^2 (y-\bar{y}^{(\tau)})}\mathrm{d}t \\
   &= \lim_{\tau \to 0} -\int_0^h \frac{\varrho}{h\delta}\inner[L^2]{\partial_t \hat{y}^{(\tau)}(t)-\partial_t \hat{y}^{(\tau)}(t-h)}{y-\bar{y}^{(\tau)}} + \inner[L^2]{\delta^{-1}\partial_F W(\nabla \bar{y}^{(\tau)})}{\nabla (y-\bar{y}^{(\tau)})} \\
   &\quad + \inner[L^2]{\delta^{-1}\partial_{\dot{F}} R(\nabla \underline{y}^{(\tau)}, \nabla \partial_t \hat{y}^{(\tau)})}{\nabla (y-\bar{y}^{(\tau)})} - \inner[L^2]{\delta^{-1}f_\delta}{y-\bar{y}^{(\tau)}} \mathrm{d}t \to 0.
  \end{align*}
  Then the same limit needs to hold for almost every $t\in[0,h]$ and thus Lemma \ref{lem:minty} implies $\inner[L^2]{\partial_G P(\nabla^2 \bar{y}^{(\tau)})}{\nabla^2 \varphi} \to \inner[L^2]{\partial_G P(\nabla^2 y)}{\nabla^2 \varphi}$. Thus we have obtained the weak equation on $[0,h]$:
 \begin{align} \label{eq:weakh}
   0 &= \int_0^h \frac{\varrho}{h\delta}\inner[L^2]{\partial_t y(t)-\partial_t y(t-h)}{\varphi} + \inner[L^2]{\delta^{-1}\partial_F W(\nabla y)}{\nabla \varphi} + \inner[L^2]{\delta^{-p\alpha+1}\partial_G P(\nabla^2 y)}{\nabla^2 \varphi} \\
   &+ \inner[L^2]{\delta^{-1}\partial_{\dot{F}} R(\nabla y, \nabla \partial_t y)}{\nabla \varphi} - \inner[L^2]{\delta^{-1}f_\delta}{\varphi} \mathrm{d}t. \nonumber
 \end{align}
 for all $\phi \in L^2([0,h];W^{2,q}_0(\Omega;\R^d))$

  The last step is the energy inequality. For this we have a closer look at \eqref{eq:discrEnergyIneq}. Summing up over $k$ and rewriting the result in terms of approximations we get
\begin{align} \label{eq:rewrittenDiscreteIneq}
 &\phantom{{}={}}  \mathcal{E}_\delta(\bar{y}^{(\tau)}(t_0)) + \tfrac{\varrho}{2h\delta^2}\int_0^{t_0} \int_\Omega \left| \beta^{(\tau)}(t) \right|^ 2 \mathrm{d} x\, \mathrm{d} t \\ \nonumber
 &+\int_0^{t_0} \mathcal{R}_\delta^*\left(\underline{y}^{(\tau)},D_2\mathcal{R}_\delta\left(\underline{y}^{(\tau)}, \beta^{(\tau)}\right)\right)  + \mathcal{R}_\delta\left( \underline{y}^{(\tau)},\partial_t \hat{y}^{(\tau)}\right) \mathrm{d} t \\ \nonumber
 &+ \int_0^{t_0} \int_\Omega \tfrac{\varrho}{2h\delta^2} \left| \partial_t  \hat{y}^{(\tau)} - \partial_t y_\delta(\tau k -h)\right|^2 \mathrm{d} x \mathrm{d} t \\ \nonumber
 &\leq \mathcal{E}_\delta(y_0) + \tfrac{\varrho}{2h\delta^2} \int_0^{t_0} \int_\Omega \abs{\partial_t y_\delta(\tau k -h)}^2 \mathrm{d} x  + \tfrac{1}{\delta^2}\int_\Omega  f\cdot {\tfrac{y_\tau-y_{k-1,\delta}}{\tau}} \mathrm{d} x \,\mathrm{d} t
\end{align}
where
\begin{align*}
 \beta^{(\tau)}(t) := \tfrac{y_{\sigma}-y_{l-1,\delta}}{\sigma} \text{ for } t = l\tau +\sigma
\end{align*}
 is an additional approximation. By the above estimate and the properties of $\mathcal{R}_\delta$, we also know that $\beta^{(\tau)}$ is uniformly bounded in $L^2(I;W^{1,2}(\Omega;\R^d))$ and we can thus extract a weakly converging subsequence $\beta^{(\tau)} \rightharpoonup \beta$. Note that this also implies $D_2\mathcal{R}_\delta\left(\underline{y}^{(\tau)},\beta^{(\tau)}\right) \rightharpoonup D_2\mathcal{R}_\delta\left(y,\beta\right)$ in $L^2(I;W^{-1,2}(\Omega;\R^d))$.

 Since it was defined via minimization, $y_\sigma$ fulfills a similar Euler-Lagrange equation to that of $y_k$, specifically
 \begin{align*}
  D\mathcal{E}_\delta(y_\sigma) + D_2\mathcal{R}_\delta\left(y_{k-1,\delta},\tfrac{y_\sigma-y_{k-1,\delta}}{\sigma} \right)  + \frac{\rho}{h\delta^2} \left(\tfrac{y_\sigma-y_{k-1,\delta}}{\sigma} - w_k\right) = - \delta^{-2} f_k
 \end{align*}
 Similarly as we did with the actual equation, we can now integrate this in time and rewrite it using our new variables (where we set $\tilde{y}^{(\tau)}(t) := y_{\sigma}$ for  $t = k\tau +\sigma$ correspondingly) to obtain
 \begin{align*}
   0 &= \int_0^h \frac{\varrho}{h\delta}\inner[L^2]{\beta^{(\tau)}(t)-\partial_t y^{(\tau)}(t-h)}{\varphi} + \inner[L^2]{\delta^{-1}\partial_F W(\nabla \tilde{y}^{(\tau)})}{\nabla \varphi} + \inner[L^2]{\delta^{-p\alpha+1}\partial_G P(\nabla^2 \tilde{y}^{(\tau)})}{\nabla^2 \varphi} \\
   &+ \inner[L^2]{\delta^{-1}\partial_{\dot{F}} R(\nabla \underline{y}^{(\tau)}, \nabla \beta^{(\tau)})}{\nabla \varphi} - \inner[L^2]{\delta^{-1}f_\delta}{\varphi} \mathrm{d}t.
 \end{align*}
 Then sending $\tau \to 0$, we note that $\tilde{y}^{(\tau)} \to y$ as well, so we obtain
 \begin{align*}
   0 &= \int_0^h \frac{\varrho}{h\delta}\inner[L^2]{\beta(t)-\partial_t y(t-h)}{\varphi} + \inner[L^2]{\delta^{-1}\partial_F W(\nabla y)}{\nabla \varphi} + \inner[L^2]{\delta^{-p\alpha+1}\partial_G P(\nabla^2 y)}{\nabla^2 \varphi} \\
   &+ \inner[L^2]{\delta^{-1}\partial_{\dot{F}} R(\nabla y, \nabla \beta)}{\nabla \varphi} - \inner[L^2]{\delta^{-1}f_\delta}{\varphi} \mathrm{d}t.
 \end{align*}
 Comparing this with the weak equation \eqref{eq:weakh} we obtain
 \begin{align*}
  D_2\mathcal{R}_\delta\left(y,\partial_t y \right)  + \frac{\rho}{h\delta^2} \partial_t y  = D_2\mathcal{R}_\delta\left(y,\beta \right)  + \frac{\rho}{h\delta^2} \beta.
 \end{align*}
 Now we note that each side is the derivative of $\mathcal{R}_\delta\left(y,\cdot\right) + \frac{\rho}{2h\delta^2} \norm[L^2(\Omega)]{\cdot}^2$ at $\partial_t y$ and $\beta$ respectively. Since this expression is strictly convex, this implies $\beta = \partial_t y$.

 Then finally, we can apply the lower semicontinuity of $\mathcal{E}$ and $\mathcal{R}$ respectively, as well as that of the $L^2$ norm to the terms on the LHS of \eqref{eq:rewrittenDiscreteIneq}. With this, it converges to the limit inequality:
 \begin{align} \label{eq:TDRineq}
 &\phantom{{}={}}  \mathcal{E}_\delta(y_\delta(t_0)) + \tfrac{\varrho}{2h\delta^2}\int_0^{t_0} \int_\Omega\left| \partial_t y_\delta \right|^ 2 \mathrm{d} x\, \mathrm{d} t +\int_0^{t_0} \mathcal{R}_\delta^*\left(y_\delta,D_2\mathcal{R}_\delta(y_\delta,\partial_t y_\delta)\right)+ \mathcal{R}_\delta\left( y_\delta,\partial_t  y_\delta\right) \mathrm{d} t \\ \nonumber
 &+ \int_0^{t_0} \int_\Omega \tfrac{\varrho}{2h\delta^2} \left| \partial_t  y_\delta - \partial_t y_\delta(\tau k -h)\right|^2 \mathrm{d} x \mathrm{d} t
 \leq \mathcal{E}_\delta(y_0) + \tfrac{\varrho}{2h\delta^2} \int_0^{t_0} \int_\Omega \abs{\partial_t y_\delta(\tau k -h)}^2 \mathrm{d} x  + \tfrac{1}{\delta^2}\int_\Omega f\cdot \partial_t y_\delta \mathrm{d} x \mathrm{d} t
\end{align}
Noting that $\mathcal{R}_\delta^*\left(y_\delta,D_2\mathcal{R}_\delta(y_\delta,\partial_t y_\delta)\right)+ \mathcal{R}_\delta\left( y_\delta,\partial_t  y_\delta\right) = \inner{D\mathcal{R}_\delta(y_\delta,\partial_t y_\delta)}{\partial_t y_\delta}$ via the Fenchel-Young inequality, this is our desired energy inequality.

 Having thus constructed a solution on $[0,h]$, we now look at the next interval. Note that the energy estimate for $[0,h]$ gives us bounds on $\mathcal{E}_\delta(y(h))$ and $\sup_{t\in[0,h]}\norm[L^2(\Omega)]{\partial_t y(t)} = \sup_{t\in[h,2h]}\norm[L^2(\Omega)]{\partial_t y(t-h)}$, which we precisely take as data for a similar solution on $[h,2h]$.

 Specifically, we set $y(h)$ as initial data and $w(t) = \partial_t y(t-h)$ for the delayed velocity on $[h,2h]$. With this we then repeat the previous argument constructing $y$ on $[h,2h]$. We then simply iterate the procedure until the final time $T$. By the choice of $w$, we then obtain the correct time-delayed equation and we also immediately note that the energy estimate telescopes.
 \end{proof}

 We can also obtain a similar result for the corresponding linearized problem:

 \begin{definition}[Weak solution to the linear time-delayed problem (L-TD)]
  \label{def:solLTD}
  Let $h>0$, $I = [0,T]$ and consider initial data $u_0 \in \mathcal{E}$, $v_0 \in L^2(\Omega)$ and a force $f \in L^2(I \times \Omega)$. We say that $u \in W^{1,2}(I;W^{1,2}(\Omega;\R^d)) \cap L^\infty(I;W^{2,p}(\Omega;\R^d))$ is a weak solution to the linear time-delayed problem (L-TD) if $u(0) = u_0$ and $u$ satisfies
  \begin{align*}
   0 &= \int_0^T \frac{\varrho}{h}\inner[L^2]{\partial_t u(t)-\partial_t u(t-h)}{\varphi} + \inner[L^2]{\C_W e(u)}{\nabla \varphi}\\
   &+ \inner[L^2]{\C_D e(\partial_t u)}{\nabla \varphi}  - \inner[L^2]{f}{ \varphi} \mathrm{d}t
  \end{align*}
  for all $\varphi \in C^\infty_c((0,T) \times \Omega)$ where we define by abuse of notation $\partial_t u(t) := v_0$ for $t<0$, as well as the energy inequality
  \begin{align*}
   &\phantom{{}={}}\mathcal{E}_0(u(s)) + \frac{\varrho}{2} \int_{s-h}^s \norm[L^2(\Omega)]{\partial_t u(t)}^2 \mathrm{d}t +  \int_0^s \inner[L^2]{\C_D e(\partial_t u)}{\nabla u} \mathrm{d}t +\int_0^s \frac{\rho}{2h} \norm[L^2(\Omega)]{\partial_t u(t)-\partial_t u(t-h)}^2 \mathrm{d}t \\
   &\leq \mathcal{E}_0(u_0) + \frac{\varrho}{2}\norm[L^2(\Omega)]{v_0}^2 + \int_0^s \inner[L^2]{f}{\partial_t u} \mathrm{d}t
  \end{align*}
  for almost all $s \in [0,T]$.
 \end{definition}

\begin{corollary} \label{cor:LTDexistence}
 For all $T \in (0,\infty)$ and given initial data/forces, the linearized time-delayed problem (L-TD) has a unique weak solution.
\end{corollary}

\begin{proof}
 The existence of weak solutions follows by showing convergence of solutions to the corresponding linearized discrete problem (L-disc), using the same arguments as in the last proof. The only slight difference is that there is no higher order term $P$. Thus the relevant space will be $u\in W^{1,2}(I;W^{1,2}(\Omega;\R^d)) \cap L^\infty(I;W^{1,2}(\Omega;\R^d))$ instead. This gives only weak convergence of the first spatial derivatives. However as now all terms in the equation are linear, this is still enough to obtain the limit equation.

 In contrast to the nonlinear case, the energy inequality can simply be derived by testing the equation with the time derivative of the solution $\varphi = \partial_t u$ itself, as the lack of a second order term reduces the regularity required of test-functions.

 Uniqueness of solutions then follows from a standard argument. Assume that there are two solutions $u_1,u_2$ for the same initial data. Then their difference $w=u_1-u_2$ will be a solution for zero initial data and right hand side. But then by the energy-inequality
 \begin{align*}
  \mathcal{E}_0(w(t)) + \int_0^t \int_\Omega\mathbb{D}e(\partial_t w) \partial_t w \mathrm{d}x \mathrm{d}t = 0 \text{ for all times }t\in (0,T)
 \end{align*}
 which using Korn's inequality implies $w=0$ and thus $u_1=u_2$.
\end{proof}

\subsection{The limits \texorpdfstring{$\delta \to 0$}{delta to zero} and \texorpdfstring{$(\tau,\delta) \to 0$}{(tau,delta) to zero} from (NL-disc)}

 To relate nonlinear and linear problem, we need to consider the limit $\delta \to 0$. As stated in the introduction, this can be done on all levels. We begin with the limit on the fully discrete level. We have already seen how this is obtained on the level of functionals using $\Gamma$-convergence. We will now repeat this on the level of equations. While both approaches seem to result in a similar statement, there are some subtle differences between the results. The $\Gamma$-limit gives us additional information about the energetic structure of the problem. On the other hand, the limit on the level of equation is not restricted to minimizers and as we will see also allows for combined limits, which do not fit into the framework of $\Gamma$-convergence.

Before we start, we will define the concept of well-prepared data.
\begin{remark}[well-prepared data]
\label{rem-wellPrep-data}
 Here and in the following, we will always assume ``well-prepared data'' in the following sense: As $\delta \to 0$, the initial deformation, the initial velocity and the forces (as well as previous velocity for the time-delayed setting) scale in such a way that the rescaled energy of the problem stays finite. In practice, that means that as $\delta \to 0$ we in particular always have
 \begin{align*}
  \delta^{-1}(y_{\delta,0} - \id) &\to u_0 & \text{ in }& W^{1,2}_0(\Omega;\R^d) \\
  \delta^{-1}v_{\delta,0} &\to v_{0,0} & \text{ in }& L^2(\Omega;\R^d)\\
  \delta^{-1} f_\delta &\to f & \text{ in }& L^2(I\times \Omega;\R^d).
 \end{align*}
 The same notions also immediately transfer to the discretizations.
\end{remark}

 \begin{proposition} \label{prop:discDelta}
 Assume that for each $\delta >0$ small enough, the collection $(y_{k,\delta})_{k} \subset W^{1,2}(\Omega; \mathbb{R}^d)$ is a weak solution of the nonlinear discrete problem (NL-disc) for well-prepared data. Assume moreover that $\delta^{-1}w_{k,\delta} \rightharpoonup w_k$ in $L^2(\Omega)$ for all $k \in \{0 \ldots N\}$. Then, there exists a collection $(u_{k})_{k} \subset W^{1,2}(\Omega; \mathbb{R}^d)$ such that for $\delta \to 0$
 $$
 \delta^{-1}(y_{k,\delta}-\id) \rightharpoonup u_{k} \quad \text{in} \quad W^{1,2}(\Omega; \mathbb{R}^d)
 $$
 and $(u_{k})_{k}$ is a solution of the linear discrete problem (L-disc). Additionally if $(\tilde{y}_\delta)_\delta$ denote the corresponding De Giorgi-interpolants, then
 $$ \delta^{-1} (\tilde{y}_{\delta}(t) - \id) \rightharpoonup \tilde{u}(t) \quad \text{in} \quad W^{1,2}(\Omega; \mathbb{R}^d)$$
 for almost all $t$, where $\tilde{u}$ is the De Giorgi-interpolant of the linearized problem.
 \end{proposition}

 \begin{proof}
 Due to the energy inequality \eqref{EnergyIneq-Disc}, we find that for every fixed $k$, the equi-coercivity step from the proof of Proposition \ref{maintheorem3} stands in place in the very same manner as before. Thus, we may find a collection $(u_k)_{k} \subset W^{1,2}(\Omega; \mathbb{R}^d)$ such that
  $$
 \delta^{-1}(y_{k,\delta}(x)-\id) \rightharpoonup u_{k}(x) \quad \text{in} \quad W^{1,2}(\Omega; \mathbb{R}^d).
 $$
 Moreover, we may deduce from Lemma \ref{lem:rigidity} that
 $$
 \|\delta \nabla u_k\|_{L^\infty(\Omega)} \leq \delta^\alpha \quad \text{ and }  \quad  \|\delta \nabla u_k\|_{L^\infty(\Omega)} \leq 1/2
 $$

For passing to the limit in the weak formulation we need to identify the limit of
  \begin{align*}
   &\int_\Omega \frac{\varrho}{h} \delta^{-1 } \big(\tfrac{y_{k,\delta}-y_{k-1,\delta}}{\tau}- w_{k,\delta} \big) \cdot \varphi + \delta^{-1 } \partial_F W(\nabla y_{k,\delta}) \cdot \nabla \varphi  + \delta^{-\alpha p +1 } \partial_G P(\nabla^2 y_k) \cdot \nabla^2 \varphi \, \mathrm{d} x  \\
   &+ \delta^{-1 } \int_\Omega \partial_{\dot{F}} R\left(\nabla y_{k-1,\delta}, \nabla \tfrac{y_{k,\delta}-y_{k-1,\delta}}{\tau}\right)\cdot \nabla \varphi  \, \mathrm{d} x = \int_\Omega f \cdot \varphi \mathrm{d} x
  \end{align*}
For the first term, we see that it can be rewritten as
$$
\int_\Omega  \frac{\varrho}{h} \big(\tfrac{u_{k,\delta}-u_{k-1,\delta}}{\tau}- \delta^{-1}w_{k,\delta},\cdot) \big) \cdot \varphi\, \mathrm{d} x,
$$
which means that we pass to the limit just by the assumption on $w_{k,\delta}$ and the weak convergence.

For the second term, Lemma \ref{lem:linEnergy} implies
\begin{align*}
\abs{ \int_\Omega \left(\delta^{-1}\partial_F W(\nabla y_{k,\delta}) -  \mathbb{C} e(u_{k,\delta})\right) \cdot \nabla \varphi \mathrm{d}x } \leq \delta C \norm[L^2(\Omega)]{\nabla u_{k,\delta}}^2\norm[L^\infty(\Omega)]{\varphi} \to 0
\end{align*}
as $u_{k,\delta}$ is uniformly bounded for all $k$ and $\delta$.

For the next term, we use that due to \eqref{assumptions-P}, we have that
$$
\|\delta^{-\alpha p + 1}\delta_G P(\delta \nabla^2 u_{k,\delta})\|_{L^\frac{p}{p-1}(\Omega)} \leq C
\delta^{-\alpha p + 1}\|\delta \nabla^2 u_{k,\delta}\|_{L^{p}(\Omega)}^\frac{p}{p-1} \leq C \delta^{-\alpha p + 1} \delta^\frac{\alpha p}{p-1} \leq C \delta^{1/2}
$$
where the exponent in the last term is positive $\alpha < \frac{p-1}{p(p-2)}$; note that as $p>d$, we also have $p>2$. Thus, this term vanishes via the H\"older inequality. For the last term on the left-hand side we apply Lemma \ref{lem:linDissip} to obtain
\begin{align*}
&\int_\Omega \left| \delta^{-1} \partial_{\dot{F}} R\left(\nabla y_{k-1,\delta}, \nabla \tfrac{y_{k,\delta}-y_{k-1,\delta}}{\tau}\right)\cdot  \nabla \varphi- e\left(\tfrac{ u_{k,\delta}- u_{k-1,\delta}}{\tau}\right) \mathbb{D}(\id) e(\nabla \varphi)  \right| \mathrm{d}x \\
&\leq C \delta  \norm[L^2(\Omega)]{ \tfrac{ \nabla u_{k,\delta}- \nabla u_{k-1,\delta}}{\tau}} \norm[L^2(\Omega)]{\nabla u_{k,\delta}} \left(1  + \delta \norm[L^\infty(\Omega)]{\nabla u_{k,\delta}} \right) \norm[L^\infty(\Omega)]{\nabla \varphi} \to 0
\end{align*}
where we use the energy estimate and that $ \|\delta \nabla u_{k-1,\delta}\|_{L^\infty(\Omega)}$ is uniformly bounded.

To pass to the limit with the De Giorgi-interpolant, first we note that similar to Lemma \ref{lem:interpolantBounds} the energy-inequality implies uniform bounds on $\delta^{-1}(\tilde{y}_{k,\delta}-\id)$. Thus for every subsequence of $\delta^{-1}(\tilde{y}_{k,\delta}-\id)$ we can extract another weakly converging subsequence. By applying the $\Gamma$-convergence proved in Proposition \ref{maintheorem3}, for each fixed $\sigma$, the resulting limit then has to be a De Giorgi-interpolant $\tilde{u}_k$ for the linearized problem. Now the argument of Corollary \ref{cor:L-discExistence} also applies to the interpolant of the linearized problem and this proves that this limit is unique and thus the limit of the whole sequence.

Finally, with the convergence of all the rescaled deformations, the convergence of the energy inequality is a direct consequence of the $\Gamma$-convergence or, respectively, the lower-semicontinuity of the terms involved.
 \end{proof}

 With a combination of the previous arguments, we can also show simultaneous convergence to the linearized time-delayed problem independent of the individual rates of convergence.

 \begin{proposition} \label{prop:TDDiagonal}
 Let $(\tau_l)_l$ and $(\delta_l)_l$ be two sequences such that $\tau_l \searrow 0$ and $\delta_l \searrow 0$. For each $l$ let $(y_{k,l})_{k} \subset W^{1,2}(\Omega; \mathbb{R}^d)$ be a weak solution of the nonlinear discrete problem (NL-disc) for the step size $\tau_l$ and the parameter $\delta_l$, and assume that the initial data and forces are well-prepared in the sense of Remark \ref{rem-wellPrep-data}. In particular, assume that the corresponding piecewise-constant interpolations are well-prepared for the forces; let the well-preparedness for the corresponding piecewise-constant interpolations be as in Definition \ref{def:interpolants}. Moreover, assume that $(w_{k,l})_{k} \subset L^2(\Omega;\R^d)$ are such that $\delta_l^{-1}\overline{w_l} \rightharpoonup v_0$ in $L^2(I\times \Omega;\R^d)$ if $I\subset [0,h]$ and $\delta_l^{-1}\overline{w_l} \rightharpoonup \partial_t u(\cdot-h)$ in $L^2(I\times \Omega;\R^d)$ otherwise. Then the previously defined interpolations converge to the weak solution of (L-TD).
 \end{proposition}

 \begin{proof}
We define the piecewise constant and piecewise affine approximations as in Definition \ref{def:interpolants}. Using the uniform bounds obtained in Lemma \ref{lem:interpolantBounds}, we see that the respective rescaled displacements defined through
\begin{align*}
 \overline{u}_l := \delta_l^{-1}(\overline{y}_l-\id ), \quad \underline{u}_l := \delta_l^{-1}(\underline{y}_l-\id), \quad \text{ and } \quad \hat{u}_l := \delta_l^{-1}(\hat{y}_l-\id ), \quad
\end{align*}
are bounded in $L^\infty(I;W^{1,2}(\Omega;\R^d))$ and also $W^{1,2}(I;W^{1,2}(\Omega;\R^d))$ for the last one. Additionally we note that
\begin{align*}\sup_{t\in I}\norm[W^{2,p}(\Omega)]{\delta^{-1}\overline{u}_l(t)} \leq C \quad \text{ and } \quad \norm[L^\infty(I\times \Omega)]{\delta^{-1} \nabla \underline{u}_l(t)} \leq C  \quad \text{ for all } l\in \N
\end{align*}
as well.

We can then find a subsequence (not relabeled) and a limit $u \in W^{1,2}(I;W^{1,2}(\Omega;\R^d)) \cap L^\infty(I;W^{1,2}(\Omega;\R^d))$ for which
  \begin{align*}
   &&\bar{u}_l,\underline{u}_l,\hat{u}_l &\rightharpoonup u & \text{ in } & L^2(I;W^{1,2}(\Omega;\R^d)) \\
  &\text{and }&  \hat{u}_l &\rightharpoonup u & \text{ in } & W^{1,2}(I;W^{1,2}(\Omega;\R^d)).
  \end{align*}
  Note that the fact that the two obtained limits coincide by the same arguments as used in the proof of Proposition \ref{prop:discDelta}.

  Rescaling the rewritten weak equation \eqref{eq:rewrittenDiscWeak} for $(y_{k,l})_k$ we obtain
  \begin{align*}
   0 &= \int_0^h \frac{\varrho}{h}\inner[L^2]{\partial_t \hat{u}_l(t)-\delta_l^{-1}\overline{w}_l}{\varphi} + \delta_l^{-1}\inner[L^2]{\partial_F W(\id +\delta_l\nabla \bar{u}_l)}{\nabla \varphi} + \inner[L^2]{\partial_G P(\delta\nabla^2 \bar{u}_l)}{\nabla^2 \varphi} \\
   &+ \delta_l^{-1}\inner[L^2]{\partial_{\dot{F}} R(\id+\delta_l\nabla \underline{u}_l, \nabla \partial_t \hat{u}_l)}{\nabla \varphi} - \inner[L^2]{\delta_l^{-1}\overline{f}_l}{\varphi} \mathrm{d}t.
  \end{align*}

  The first and the final term converge to their limit by weak convergence. Regarding the elastic energy, Lemma \ref{lem:linEnergy} again implies
\begin{align*}
\abs{ \int_\Omega \left(\delta_l^{-1}\partial_F W(\nabla \overline{y}_{l}) -  \mathbb{C} e(\overline{u}_{l})\right) \cdot \nabla \varphi \mathrm{d}x } \leq \delta_l C \norm[L^2(\Omega)]{\nabla \overline{u}_{l}}^2\norm[L^\infty(\Omega)]{\varphi} \to 0
\end{align*}
for all times $t\in I$, where then linearity allows us to conclude $\int_\Omega \mathbb{C} e(u_{l}) \cdot \nabla \varphi \mathrm{d} x \to \int_\Omega \mathbb{C} e(u) \cdot \nabla \varphi \mathrm{d} x$.

Similarly we can again use \eqref{assumptions-P} to conclude that
$$
\|\delta_l^{-\alpha p + 1}\partial_P(\delta \nabla^2 \overline{u}_{l})\|_{L^\frac{p}{p-1}(\Omega)} \leq C
\delta_l^{-\alpha p + 1}\|\delta \nabla^2 \overline{u}_{l}\|_{L^{p}(\Omega)}^\frac{p}{p-1} \leq C \delta_l^{-\alpha p + 1} \delta_l^\frac{\alpha p}{p-1} \leq C \delta_l^{1/2} \to 0
$$
which removes the respective term.

Finally for the dissipation term we apply Lemma \ref{lem:linDissip} to obtain
\begin{align*}
&\int_\Omega \left| \delta_l^{-1} \partial_{\dot{F}} R\left(\nabla \underline{y}_{l}, \partial_t \nabla \hat{y}_l\right)\cdot  \nabla \varphi- e\left(\partial_t \hat{u}_l\right) \mathbb{D}(\id) e(\nabla \varphi)  \right| \mathrm{d}x \\
&\leq C \delta  \norm[L^2(\Omega)]{ \partial_t \nabla \hat{u}_l} \norm[L^2(\Omega)]{\nabla \underline{u}_{l}} \left(1  + \delta \norm[L^\infty(\Omega)]{\nabla \underline{u}_{l}} \right) \norm[L^\infty(\Omega)]{\nabla \varphi}.
\end{align*}
Integrating in time and using Hölder's inequality again, we can then estimate
\begin{align*}
&\int_0^h \int_\Omega \left| \delta_l^{-1} \partial_{\dot{F}} R\left(\nabla \underline{y}_{l}, \partial_t \nabla \hat{y}_l\right)\cdot  \nabla \varphi- e\left(\partial_t \hat{u}_l\right) \mathbb{D}(\id) e(\nabla \varphi)  \right| \mathrm{d}x \mathrm{d} t\\
&\leq C \delta  \norm[L^2(I\times\Omega)]{ \partial_t \nabla \hat{u}_l} \norm[L^\infty(I;L^2(\Omega))]{\nabla \underline{u}_{l}} \left(1  + \delta \norm[L^\infty(I\times\Omega)]{\nabla \underline{u}_{l}} \right) \norm[L^2(I;L^\infty(\Omega))]{\nabla \varphi} \to 0
\end{align*}
by the energy estimate an the fact that $ \|\delta \nabla u_{l}\|_{L^\infty(\Omega)}$ is uniformly bounded.
 \end{proof}

\section{Passing to the limit from (NL-TD), (NL) or (L-TD)}

 Similar to the last section we now proceed with the limits for $h\to 0$ and $\delta \to 0$ starting from the respective time-delayed equations.

 \subsection{The limit \texorpdfstring{$h \to 0$}{h to zero} from (NL-TD) and (L-TD)}

We begin by defining the notion of weak  solutions to the fully dynamic problem, both in the nonlinear and linear case.

 \begin{definition}[Weak solution to the nonlinear dynamic problem (NL)] \label{def:solNL}
  Let $I = [0,T]$, $\delta >0$ and consider initial data $y_0 \in \mathcal{E}$, $v_0 \in L^2(\Omega)$ and a force $f \in L^2(I \times \Omega)$. We say that $y \in W^{1,2}(I;W^{1,2}(\Omega;\R^d)) \cap L^\infty(I;W^{2,p}(\Omega;\R^d))$ is a weak solution to the nonlinear dynamic problem (NL) if $y(0) = y_0$ and $y$ satisfies
  \begin{align*}
   0 &= \int_0^T \varrho\delta^{-1}\inner[L^2]{\partial_t y}{\partial_t \varphi} + \delta^{-1}\inner[L^2]{\partial_F W(\nabla y)}{\nabla \varphi} + \delta^{-p\alpha+1}\inner[L^2]{\partial_G P(\nabla^2 y)}{\nabla^2 \varphi} \\
   &+ \delta^{-1}\inner[L^2]{\partial_{\dot{F}} R(\nabla y, \nabla \partial_t y)}{\nabla \varphi} - \inner[L^2]{\delta^{-1}f}{ \varphi} \mathrm{d}t - \varrho\delta^{-1} \inner[L^2]{v_0}{\varphi(0)}
  \end{align*}
  for all $\varphi \in C^\infty_c([0,T) \times \Omega)$, as well as the energy inequality
  \begin{align*}
   &\phantom{{}={}}\mathcal{E}_\delta(y(s)) + \frac{\varrho}{2\delta^2} \norm[L^2(\Omega)]{\partial_t y(s)}^2 + \int_0^s \delta^{-2}\inner[L^2]{\partial_{\dot{F}} R(\nabla y, \nabla \partial_t y)}{\partial_t \nabla y} \mathrm{d}t \\
   &\leq \mathcal{E}_\delta(y_0) + \frac{\varrho}{2\delta^2}\norm[L^2(\Omega)]{v_0}^2 + \delta^{-2}\int_0^s \inner[L^2]{f}{\partial_t y} \mathrm{d}t
  \end{align*}
  for almost all $s \in [0,T]$.
 \end{definition}

 \begin{definition}[Weak solution to the linear dynamic problem (L)] \label{def:solL}
  Let $I = [0,T]$ and consider initial data $u_0 \in \mathcal{E}$, $v_0 \in L^2(\Omega)$ and a force $f \in L^2(I \times \Omega)$. We say that $u \in W^{1,2}(I;W^{1,2}(\Omega;\R^d))$ is a weak solution to the linear dynamic problem (L) if $u(0) = u_0$ and $u$ satisfies 
  \begin{align*}
   0 &= \int_0^T \varrho\inner[L^2]{\partial_t u}{\partial_t \varphi} + \inner[L^2]{\C e(u)}{\nabla \varphi}
   + \inner[L^2]{\mathbb{D} e(\partial_t u)}{\nabla \varphi} - \inner[L^2]{f}{ \varphi} \mathrm{d}t - \varrho\inner[L^2]{v_0}{\varphi(0)}
  \end{align*}
  for all $\varphi \in C^\infty_c([0,T) \times \Omega)$, as well as the energy inequality
  \begin{align*}
   &\phantom{{}={}}\mathcal{E}_0(u(s)) + \frac{\varrho}{2} \norm[L^2(\Omega)]{\partial_t u(s)}^2 + \int_0^s \inner[L^2]{\mathbb{D} e(\partial_t u)}{\partial_t \nabla u} \mathrm{d}t 
   \leq \mathcal{E}_0(u_0) + \frac{\varrho}{2}\norm[L^2(\Omega)]{v_0}^2 + \int_0^s \inner[L^2]{f}{\partial_t u} \mathrm{d}t
  \end{align*}
  for almost all $s \in [0,T]$.
 \end{definition}

We first need the following preliminary result that can be basically derived from earlier results:
\begin{lemma}[Equiboundedness for the time delayed problem] \label{lem:TDBounds}
  Fix $I=[0,T]$. Let $\delta > 0$ and $h >0$. Then there exists a constant $C$ independent of $\delta$ and $h$ such that for every weak solution $y_{h,\delta} \in L^\infty(I;W^{2,p}(\Omega;\R^d)) \cap W^{1,2}(I;W^{1,2}(\Omega; \mathbb{R}^d))$ to (NL-TD) with initial data $y_0$, $v_0$ and force $f_\delta \in L^2(I\times \Omega)$  we have
  \begin{align*}
   &\sup_{t\in[h,T]} \fint_{t-h}^t \norm[L^2(\Omega)]{\partial_t y_{h,\delta}}^2 \mathrm{d} t +\norm[L^\infty(I;W^{1,2}(\Omega))]{y_{h,\delta}-\id}^2 +  \norm[L^2(I;W^{1,2}(\Omega))]{y_{h,\delta}-\id}^2 \\
   &\leq C \left( \norm[L^2(\Omega)]{v_0}^2 + \norm[L^2(I\times \Omega)]{f_\delta}^2+  \delta^2\mathcal{E}_{\delta}(y_0)\right)
  \end{align*}
  for the respective interpolants as well as
  \begin{align*}
  \norm[L^\infty(I\times \Omega)]{\delta^{-1} \nabla u_{h,\delta}} + \sup_{t\in I}\norm[W^{2,p}(\Omega)]{\delta^{-1}u_{h,\delta}} \leq C \sqrt{\norm[L^2(I\times \Omega)]{\delta^{-1}f}^2+  \delta^{-2} \norm[L^2(\Omega)]{v_0}^2  + \mathcal{E}_{\delta}(y_0)}
  \end{align*}
  \end{lemma}

  \begin{proof}
Any weak solution to (NL-TD) according to Definition \ref{def:solNLTD} satisfies an energy inequality from which we conclude uniform bounds on
  \begin{align*}
   \sup_{t \in I} \mathcal{E}_\delta(y_{h,\delta}), \quad \sup_{t\in I} \frac{1}{h}\int_{t-h}^t \norm{\partial_t y_{h,\delta}}^2 \mathrm{d}t \quad \text{ and } \quad \int_0^T \!\! R(\nabla y_{h,\delta}, \nabla \partial_t y_{h,\delta}) \mathrm{d}t
  \end{align*}
  that only depend on the initial data. The proof is then closed by using Lemma \ref{lem:rigidity} and Lemma \ref{lem:KornCont}.
  \end{proof}

 Next we prove the existence of solutions to (NL) and (L) by continuing the strategy outlined in the introduction and showing that both can be seen as the $h \to 0$ limits of (NL-TD) and (L-TD), respectively (cmp. also \cite[Sec. 3]{benesovaVariationalApproachHyperbolic2020}).

 \begin{theorem} \label{thm:existence}
  The problems (NL) and (L) have weak solutions for any time $T$ and any valid initial data/forces. The solution to (L) is unique.
 \end{theorem}

 \begin{proof}
 We restrict ourselves to the proof for (NL), as the proof for (L) has the same structure except that there is no need to deal with nonlinear terms. In the same way as before, we thus take solutions $y^{(h)}$ to (NL-TD) for given $h$ and then study the limit $h\to 0$.

  Lemma \ref{lem:TDBounds} provides us with uniform bounds so that we find a compact subsequence (not relabeled) and a limit $y \in W^{1,2}(I;W^{1,2}(\Omega;\R^d)) \cap L^\infty(I;W^{2,p}(\Omega;\R^d))$ for which
  \begin{align*}
   y^{(h)} &\rightharpoonup y & \text{ in } & L^2(I;W^{2,p}(\Omega;\R^d)) \\
   y^{(h)} &\rightharpoonup y & \text{ in } & W^{1,2}(I;W^{1,2}(\Omega;\R^d)),
   \intertext{as well as using the Aubin-Lions theorem,}
   y^{(h)} &\to y & \text{ in } & L^2(I;W^{1,2}(\Omega;\R^d)).
  \end{align*}
Additionally, it should be noted that for the rolling average $m^{(h)}(t) = \fint_{t-h}^t \partial_t y^{(h)} \mathrm{d}s$ we have
  \begin{align*}
   \partial_t m^{(h)}(t) = \tfrac{\partial_t y^{(h)}(t)- \partial_t y^{(h)}(t-h)}{h}
  \end{align*}
  which is a term that occurs in the weak equation for (NL-TD). Using the weak equation itself, we can get a uniform estimate on $\partial_t m \in L^2(I;W^{-2,p}(\Omega;\R^d))$. With this, noting that Lemma \ref{lem:TDBounds} implies $m^{(h)} \in L^\infty(I;L^2(\Omega;\R^d))$ and a trivial identification of the limit, the Aubin-Lions theorem also implies:
  \begin{align*}
   m^{(h)} \to \partial_t y &\text{ in } L^2(I\times \Omega;\R^d).
  \end{align*}

  With all this in hand, we can pass to the limit both in the equation and the energy-inequality. Most of the terms in (NL-TD) converge to their counterpart in (NL) as a direct consequence of the above convergences. The only exceptions are the second order and the inertial term.

  For the former, we test the weak form of (NL-TD) by $y^{(h)}-y$ to conclude that for $h \to 0$
  \begin{align*}
   &\phantom{{}={}} - \int_0^T \delta^{-\alpha p-1} \inner{\partial_G P(\nabla^2 y^{(h)})}{\nabla^2(y^{(h)}-y)} \mathrm{d}t \\
   &= \int_0^T \int_\Omega \delta^{-1}\varrho\partial_t m^{(h)}(t)\cdot (y^{(h)}-y) + \delta^{-1} \partial_F W(\nabla y^{(h)}) \cdot \nabla (y^{(h)}-y) \\
   &+ \delta^{-1} \partial_{\dot{F}} R(\nabla y^{(h)}, \nabla \partial_t ^{(h)}) \cdot \nabla (y^{(h)}-y) - \delta^{-1} f_\delta \cdot (y^{(h)}-y) \mathrm{d}x \mathrm{d}t \to 0,
  \end{align*}
  Here convergence for all but the first term follows directly from the strong convergence of $y^{(h)}$. For the first term, we perform one more partial integration in time and then note that $m^{(h)}$ converges strongly in $L^2$ to $\partial_t y$ and $\partial_t (y^{(h)}-y)$ converges weakly in $L^2$. Then convergence of $\inner{\partial_G P(y^{(h)})}{\nabla^2 \varphi} \to \inner{\partial_G P(y)}{\nabla^2 \varphi}$ follows again by Lemma \ref{lem:minty}.

  For the inertial term we restrict ourselves to test-functions $\varphi$ supported in $(0,T-h_0)\times \Omega$ and rewrite for $h<h_0$ by shifting half of the integral as a sort of discrete partial integration
  \begin{align*}
   \int_0^T \int_\Omega \frac{\varrho}{h}\big(\partial_t y^{(h)}(t)-\partial_t y^{(h)}(t-h)\big)\cdot \varphi \mathrm{d} x \mathrm{d} t = - \int_0^T \int_\Omega \frac{\varrho}{h}\partial_t y^{(h)}(t)\cdot \left(\varphi(t+h) -\varphi(t)\right) \mathrm{d} x \mathrm{d} t.
  \end{align*}
 
  Then the difference quotient $\frac{\varphi(t+h) -\varphi(t)}{h}$ converges strongly to $\partial_t \varphi$, so weak convergence of $\partial_t y^{(h)}$ is enough to arrive at the desired limit.
 \end{proof}

 \begin{remark}[Other solution methods for the linear problem]
  Note that since the problems (L), (L-TD) and (L-disc) are all linear, involving well behaved operators, there is a plethora of more conventional ways for proving the existence of solutions. There is also plenty of associated regularity-theory which suggests that stronger concepts of solutions might apply.

  As we are, however, mainly interested in the convergence between solution concepts, it is much more advantageous to consider the convergence of weak solutions to weak solutions, even if the weak solutions in the limit later turn out to be strong.
 \end{remark}

 \subsection{The limits \texorpdfstring{$\delta \to 0$}{delta to zero} and \texorpdfstring{$(h,\delta) \to 0$}{(h,delta) to zero} for (NL-TD)}

 We now additionally show that we can pass to the limit $\delta \to 0$ in the time-delayed equation alone.

\begin{proposition} \label{prop:TDDelta}
 The time-delayed problem (NL-TD) converges to (L-TD) as $\delta \to 0$   
\end{proposition}

\begin{proof}
Using the estimates from Lemma \ref{lem:TDBounds} and standard compactness theorems, we can select a suitable subsequence of $\delta$'s (not relabeled) and a function
$$
u \in W^{1,\infty}([0,T]; L^2(\Omega; \mathbb{R}^d)) \cap  W^{1,2}([0,T]; W^{1,2}(\Omega; \mathbb{R}^d)),
$$
so that 
\begin{align*}
 u_{\delta} & \rightharpoonup u && \text{ in } L^\infty([0,T];W^{1,2}(\Omega;\R^d))\\
 \partial_t u_\delta & \rightharpoonup \partial_t u    & &\text{ in } L^2([0,T];W^{1,2}( \Omega; \mathbb{R}^d))
\end{align*}
We now want to show convergence of the weak equation. For elastic energy and dissipation functionals we argue similarly as in Proposition \ref{prop:TDDiagonal}.

Namely, Lemma \ref{lem:linEnergy} implies
\begin{align*}
\abs{ \int_\Omega \left(\delta^{-1}\partial_F W(\nabla y_{\delta}) -  \mathbb{C} e(u_\delta)\right) \cdot \nabla \varphi \mathrm{d}x } \leq \delta C \norm[L^2(\Omega)]{\nabla u_\delta}^2\norm[L^\infty(\Omega)]{\varphi} \to 0
\end{align*}
for all times $t\in I$.

Similarly we can again use \eqref{assumptions-P} to conclude that
$$
\|\delta^{-\alpha p + 1}\partial_G P(\delta \nabla^2 u_{\delta})\|_{L^\frac{p}{p-1}(\Omega)} \leq C
\delta^{-\alpha p + 1}\|\delta \nabla^2 u_{\delta}\|_{L^{p}(\Omega)}^\frac{p}{p-1} \leq C \delta^{-\alpha p + 1} \delta^\frac{\alpha p}{p-1} \leq C \delta^{1/2} \to 0
$$
which removes the nonlinear second order term.

For the dissipation term we apply Lemma \ref{lem:linDissip} to obtain
\begin{align*}
&\int_\Omega \left| \delta^{-1} \partial_{\dot{F}} R\left(\nabla y_{\delta}, \partial_t \nabla y_\delta\right)\cdot  \nabla \varphi- e\left(\partial_t u_\delta\right) \mathbb{D}(\id) e(\nabla \varphi)  \right| \mathrm{d}x \\
&\leq C \delta  \norm[L^2(\Omega)]{ \partial_t \nabla u_\delta} \norm[L^2(\Omega)]{\nabla u_{\delta}} \left(1  + \delta \norm[L^\infty(\Omega)]{\nabla u_{\delta}} \right) \norm[L^\infty(\Omega)]{\nabla \varphi}.
\end{align*}
Integrating in time and using Hölder's inequality again, we can then estimate
\begin{align*}
&\int_0^h \int_\Omega \left| \delta^{-1} \partial_{\dot{F}} R\left(\nabla y_{\delta}, \partial_t \nabla y_\delta\right)\cdot  \nabla \varphi- e\left(\partial_t u_\delta\right) \mathbb{D}(\id) e(\nabla \varphi)  \right| \mathrm{d}x \mathrm{d} t\\
&\leq C \delta  \norm[L^2(I\times\Omega)]{ \partial_t \nabla u_\delta} \norm[L^\infty(I;L^2(\Omega))]{\nabla u_{\delta}} \left(1  + \delta \norm[L^\infty(I\times\Omega)]{\nabla u_{\delta}} \right) \norm[L^2(I;L^\infty(\Omega))]{\nabla \varphi} \to 0
\end{align*}
by the energy estimate and the fact that $ \|\delta \nabla u_{\delta}\|_{L^\infty(\Omega)}$ is uniformly bounded.

Now all the terms left over are linear in $u_\delta$, $\nabla u_\delta$ or $\partial_t \nabla u_\delta$, so by weak convergence we reach the desired limit.
\end{proof}

This limit procedure is again compatible with $h \to 0$.

\begin{theorem} \label{thm:Diagonal}
 Let $(h_l)_l$ and $(\delta_l)_l$ be two sequences such that $h_l \searrow 0$ and $\delta_l \searrow 0$. For each $l$ let $y_{l} \in L^\infty(I;W^{2,p}(\Omega; \mathbb{R}^d)) \cap W^{1,2}(I;W^{1,2}(\Omega;\R^d))$ be a weak solution of the nonlinear time-delayed problem (NL-TD) for delay $h_l$ and parameter $\delta_l$, with well-prepared initial data $y_0^l$ such that $u_0^l := \delta^{-1}(y_0^l-\id) \to u_0$ in $W^{1,2}(\Omega;\R^d)$, well-prepared right hand side data $f_l$ and initial velocity $v_l$ such that $\delta_l^{-1} v_l$  converges weakly in $L^2(I\times\Omega;\R^d)$. Then the previously defined interpolations converge to the weak solution of problem (L-TD).
\end{theorem}

\begin{proof}
 We again rely on Lemma \ref{lem:TDBounds} to derive uniform bounds on $u_l := \delta^{-1}(y_l -\id)$. From these we then find a limit and a subsequence such that
\begin{align*}
  u_l & \rightharpoonup u    & &\text{ in } W^{1,2}([0,T] \times \Omega; \mathbb{R}^d) \\
 \partial_t u_l & \rightharpoonup \partial_t u    & &\text{ in } L^2([0,T] \times \Omega; \mathbb{R}^d)
\end{align*}
By the definition of $P$ and H\"older's inequality with powers $\frac{p}{p-1}$ and $p$ we derive that
  \begin{align}
    \frac{1}{\delta} \Bigg|
      \int_0^{T} \int_\Omega \partial_G P(\nabla^2 y_\delta) \cdot \nabla^2  \mathrm{d} x \mathrm{d} t
    \Bigg|
    &\leq \frac{C}{\delta} \int_0^{T} \int_\Omega \abs{\delta \nabla^2 u_\delta}^{p-1}
      \abs{\nabla^2 \varphi}\mathrm{d} x \mathrm{d} t\\
    &\leq  \frac{C}{\delta} \int_{0}^T \| \delta \nabla^2 u_\delta \|^{p-1}_{L^p(\Omega)}
      \|\nabla^2 \varphi\|_{L^p(\Omega)}\mathrm{d} t
      \leq C \delta^{\alpha(p-1) - 1}  \to 0,
\end{align}
provided $\alpha > \frac{1}{p-1}$.

Using Lemma \ref{lem:linEnergy} we have
\begin{align*}
 \abs{ \int_\Omega \left( \delta^{-1}\partial_F W(\nabla y_\delta) - \mathbb{C} e(u_\delta)\right) \cdot \nabla \varphi \mathrm{d}x } \leq C \delta \norm[L^2(\Omega)]{\nabla u_\delta}^2 \norm[L^\infty(\Omega)]{\nabla \varphi} \to 0
\end{align*}
pointwise.

For the dissipation term we apply Lemma \ref{lem:linDissip} to obtain
\begin{align*}
&\int_\Omega \left| \delta_l^{-1} \partial_{\dot{F}} R\left(\nabla y_{\delta}, \partial_t \nabla y_\delta\right)\cdot  \nabla \varphi- e\left(\partial_t u_\delta\right) \mathbb{D}(\id) e(\nabla \varphi)  \right| \mathrm{d}x \\
&\leq C \delta  \norm[L^2(\Omega)]{ \partial_t \nabla u_\delta} \norm[L^2(\Omega)]{\nabla u_{\delta}} \left(1  + \delta \norm[L^\infty(\Omega)]{\nabla u_{\delta}} \right) \norm[L^\infty(\Omega)]{\nabla \varphi}.
\end{align*}
Integrating in time and using Hölder's inequality again, we can then estimate
\begin{align*}
&\int_0^h \int_\Omega \left| \delta_l^{-1} \partial_{\dot{F}} R\left(\nabla y_{\delta}, \partial_t \nabla y_\delta\right)\cdot  \nabla \varphi- e\left(\partial_t u_\delta\right) \mathbb{D}(\id) e(\nabla \varphi)  \right| \mathrm{d}x \mathrm{d} t\\
&\leq C \delta  \norm[L^2(I\times\Omega)]{ \partial_t \nabla u_\delta} \norm[L^\infty(I;L^2(\Omega))]{\nabla u_{\delta}} \left(1  + \delta \norm[L^\infty(I\times\Omega)]{\nabla u_{\delta}} \right) \norm[L^2(I;L^\infty(\Omega))]{\nabla \varphi} \to 0
\end{align*}
by the energy estimate an the fact that $ \|\delta \nabla u_{\delta}\|_{L^\infty(\Omega)}$ is uniformly bounded.

This leaves us with the inertial term. This is dealt with in the same way as in Theorem \ref{thm:existence}.
\end{proof}

\subsection{Convergence of the (NL) as \texorpdfstring{$\delta \to 0$}{delta to zero}}
 
 Finally we show that taking the limit $\delta \to 0$ is also possible directly for solutions to dynamic problems, showing the final convergence in our diagram on the level of equations.

\begin{theorem} \label{thm:Delta}
 Let $(y_\delta)_\delta \subset W^{1,2}(I;W^{1,2}(\Omega;\R^d)) \cap L^\infty(I;W^{2,p}(\Omega;\R^d))$ be a family of weak solutions to (NL) with given well-prepared initial data and denote $u_\delta := \delta^{-1} (y_\delta - \id)$. Then there exists a subsequence (not relabeled) and a limit $u \in W^{1,2}(I;W^{1,2}(\Omega;\R^d)) \cap L^\infty(I;W^{1,2}(\Omega;\R^d))$ such that $u_\delta \to u$ in $W^{1,2}(I;W^{1,2}(\Omega;\R^d)) \cap L^\infty(I;W^{1,2}(\Omega;\R^d))$ and $u$ is the (unique) weak solution of (L).
\end{theorem}

\begin{proof}
 We begin as in Lemma \ref{lem:TDBounds} and from the energy-inequalities for the individual $y_\delta$, we deduce $\delta$-independent bounds on
 \begin{align*}
  \sup_{t\in I}\mathcal{E}_\delta(u_\delta(t)) \quad \text{ and } \quad \int_I \mathcal{R}_\delta(\nabla u_\delta,\partial_t \nabla u_\delta) \mathrm{d}t.
 \end{align*}
 Thus using Lemma \ref{lem:rigidity} and Lemma \ref{lem:KornCont}, we obtain
$$
\norm[L^2(\Omega)]{\partial_t u_\delta(t)}^2  \leq C,
$$
as well as
$$
\|\partial_t \nabla u_\delta \|^2_{L^2([0,T]\times \Omega; \mathbb{R}^{d \times d})} \leq C.
$$
In addition the control of the kinetic energy gives us the estimate
$$
\|\partial_t y_\delta \|^2_{L^\infty([0,T]; L^2(\Omega; \mathbb{R}^{d \times d}))} \leq C \delta^{-2},
$$
so that we have that
$$
\|\partial_t u_\delta \|^2_{L^\infty([0,T]; L^2(\Omega; \mathbb{R}^{d \times d}))} \leq C ,
$$
In order to pass to the limit, we now select a suitable subsequence of $\delta$'s (not relabeled) and a function
$$
u \in W^{1,\infty}([0,T]; L^2(\Omega; \mathbb{R}^d)) \cap  W^{1,2}([0,T]; W^{1,2}(\Omega; \mathbb{R}^d)),
$$
so that
\begin{align*}
  u_\delta & \rightharpoonup u    & &\text{ in } W^{1,2}([0,T] \times \Omega; \mathbb{R}^d) \\
 \partial_t u_\delta & \rightharpoonup \partial_t u    & &\text{ in } L^2([0,T] \times \Omega; \mathbb{R}^d)
\end{align*}

From this we immediately see that for any $\varphi \in C_c^\infty([0,T) \times \Omega)$
$$
\frac{1}{\delta} \int_0^T \int_\Omega \varrho \partial_t y_\delta \cdot \partial_t \, \varphi \mathrm{d} x \mathrm{d} t = \int_0^T \int_\Omega \varrho \partial_t u_\delta \cdot \partial_t \, \varphi \mathrm{d} x \mathrm{d} t \to \int_0^T \int_\Omega \varrho \partial_t u \cdot \partial_t \, \varphi \mathrm{d} x \mathrm{d} t
$$
By the definition of $P$ and H\"older's inequality with powers $\frac{p}{p-1}$ and $p$ we derive that
  \begin{align*}
    \frac{1}{\delta} \Bigg|
      \int_0^{T} \int_\Omega \partial_G P(\nabla^2 y_\delta) \cdot \nabla^2  \mathrm{d} x \mathrm{d} t
    \Bigg|
    &\leq \frac{C}{\delta} \int_0^{T} \int_\Omega \abs{\delta \nabla^2 u_\delta}^{p-1}
      \abs{\nabla^2 \varphi}\mathrm{d} x \mathrm{d} t\\
    &\leq  \frac{C}{\delta} \int_{0}^T \| \delta \nabla^2 u_\delta \|^{p-1}_{L^p(\Omega)}
      \|\nabla^2 \varphi\|_{L^p(\Omega)}\mathrm{d} t 
      \leq C \delta^{\alpha(p-1) - 1}  \to 0, 
\end{align*}
provided $\alpha > \frac{1}{p-1}$.

Using Lemma \ref{lem:linEnergy} we have
\begin{align*}
 \abs{ \int_\Omega \left( \delta^{-1}\partial_F W(\nabla y_\delta) - \mathbb{C} e(u_\delta)\right) \cdot \nabla \varphi \mathrm{d}x } \leq C \delta \norm[L^2(\Omega)]{\nabla u_\delta}^2 \norm[L^\infty(\Omega)]{\nabla \varphi} \to 0
\end{align*}
pointwise.

For the dissipation term we apply Lemma \ref{lem:linDissip} to obtain
\begin{align*}
&\int_\Omega \left| \delta_l^{-1} \partial_{\dot{F}} R\left(\nabla y_{\delta}, \partial_t \nabla y_\delta\right)\cdot  \nabla \varphi- e\left(\partial_t u_\delta\right) \mathbb{D}(\id) e(\nabla \varphi)  \right| \mathrm{d}x \\
&\leq C \delta  \norm[L^2(\Omega)]{ \partial_t \nabla u_\delta} \norm[L^2(\Omega)]{\nabla u_{\delta}} \left(1  + \delta \norm[L^\infty(\Omega)]{\nabla u_{\delta}} \right) \norm[L^\infty(\Omega)]{\nabla \varphi}.
\end{align*}
Integrating in time and using Hölder's inequality again, we can then estimate
\begin{align*}
&\int_0^h \int_\Omega \left| \delta_l^{-1} \partial_{\dot{F}} R\left(\nabla y_{\delta}, \partial_t \nabla y_\delta\right)\cdot  \nabla \varphi- e\left(\partial_t u_\delta\right) \mathbb{D}(\id) e(\nabla \varphi)  \right| \mathrm{d}x \mathrm{d} t\\
&\leq C \delta  \norm[L^2(I\times\Omega)]{ \partial_t \nabla u_\delta} \norm[L^\infty(I;L^2(\Omega))]{\nabla u_{\delta}} \left(1  + \delta \norm[L^\infty(I\times\Omega)]{\nabla u_{\delta}} \right) \norm[L^2(I;L^\infty(\Omega))]{\nabla \varphi} \to 0
\end{align*}
by the energy estimate an the fact that $ \|\delta \nabla u_{\delta}\|_{L^\infty(\Omega)}$ is uniformly bounded.
\end{proof}
 
\section{Energetical notions of convergence along the limit process}

Finally we will try and re-contextualize the previous limit processes in terms of the underlying functionals. We have already discussed this convergence on the discrete level in the sense of $\Gamma$-convergence. As this is ultimately a notion of convergence of minimizers, it cannot be directly applied to any dynamic problem. However, as we have seen, the energetic structure is still at the heart of both the full inertial problem and its time-delayed approximation. In particular, the latter has the structure of a quasistatic problem, for which similar approaches have long been studied.

Since a lot of this is either already well known (for the quasistatic case), or forms just an alternative way to prove results of the previous section, we will refrain from carrying out overly detailed proofs. Rather this section is intended to serve as a coda to the rest of the paper, recontextualizing some of the ideas.

We will first explain the relevant notions in abstract and then apply them to our specific problem.

\subsection{Abstract (Tilt-)EDP-convergence and EDP-convergence with forces}

The notion of EDP-convergence has been developed in the theory of gradient flows.
The fundamental idea, first observed by Sandier \& Serfaty \cite{S1} and later used and formalized by many others, relies on the following (formal) observation:

Let $X$ be a given vector space and $\mathcal{E} \in C^1(X)$, $\mathcal{R} \in C^1(X\times X)$, where $\mathcal{R}(\cdot,0) = 0$ and $\mathcal{R}$ is convex in the second argument. For any differentiable curve $y:  [0,T] \to X$ we then have using the Fenchel-Young inequality \eqref{eq:FenchelEquality}
\begin{align} \label{eq:curveIneq}
 \mathcal{E}(y(0))-\mathcal{E}(y(T)) = \int_0^T \inner{-D\mathcal{E}(y)}{\partial_t y} dt \leq \int_0^T \mathcal{R}(y,\partial_t y) + \mathcal{R}^*(y,-D\mathcal{E}(y)) dt.
\end{align}
In this, according to Theorem \ref{thm:Fenchel}, equality holds if and only if $D_2\mathcal{R}(y,\partial_t y) = -D\mathcal{E}(y)$ for almost all times, as the latter is the argument of $\mathcal{R}^*$. In other words, equality holds if and only if the curve $y$ is a solution of the gradient flow.

Furthermore, since the inequality \eqref{eq:curveIneq} holds for arbitrary curves, showing equality just means showing the opposite inequality. This we can rewrite as
\begin{align*}
 \mathcal{E}(y(T)) + \int_0^T \mathcal{R}(\partial_t y) + \mathcal{R}^*(-D\mathcal{E}(y)) dt\leq \mathcal{E}(y(0))
\end{align*}
which is nothing more than a rather specific way of writing the energy-inequality of the gradient flow.

One can easily adapt this to a $\delta$-dependent family of problems, where $y_\delta$ is now a solution to the ($\mathcal{E}_\delta$, $\mathcal{R}_\delta$)-gradient flow if
\begin{align} \label{eq:gradientFlowEnergy}
 \mathcal{E}_\delta(y_\delta(T)) + \int_0^T \mathcal{R}_\delta(y_\delta,\partial_t y_\delta) + \mathcal{R}_\delta^*(y_\delta,-D\mathcal{E}_\delta(y_\delta)) \mathrm{d}t \leq \mathcal{E}_\delta(y_\delta(0)).
\end{align}

If we now want to show that (assuming some compactness), solutions of this  $(\mathcal{E}_\delta$, $\mathcal{R}_\delta$)-gradient flow converge to solutions of a corresponding ($\mathcal{E}_0$, $\mathcal{R}_0$)-gradient flow, then it turns out to show two things.

First of all, similar as in the static problem, we expect $\mathcal{E}_\delta$ to $\Gamma$-converge to $\mathcal{E}_0$. Secondly, if we define the total dissipation as a functional of curves
\begin{align*}
 \mathcal{D}_\delta(y) := \int_0^T \mathcal{R}_\delta(y,\partial_t y) + \mathcal{R}_\delta^*(y,-D\mathcal{E}_\delta(y)) \mathrm{d}t
\end{align*}
then this should  $\Gamma$-converge to its respective limit
\begin{align*}
 \mathcal{D}_0(y) := \int_0^T \mathcal{R}_0(y,\partial_t y) + \mathcal{R}_0^*(y,-D\mathcal{E}_0(y)) \mathrm{d}t.
\end{align*}

If both of these convergences are true and a sequence of solutions $y_\delta$ satisfying \eqref{eq:gradientFlowEnergy} converges to a limit $y_0$ and assuming also that the initial data is well posed (i.e.\ $E_\delta(y_\delta(0)) \to E_\delta(y_0(0))$), then using the liminf-property of both $\Gamma$-convergences we obtain
\begin{align*}
 \mathcal{E}_0(y_0(T)) + \int_0^T \mathcal{R}_0(y_0,\partial_t y_0) + \mathcal{R}_0^*(y_0,-D\mathcal{E}_0(y_0)) \mathrm{d}t \leq \mathcal{E}_0(y_0(0)).
\end{align*}
By the initial considerations this then implies that $y_0$ solves the ($\mathcal{E}_0$, $\mathcal{R}_0$)-gradient flow.

The remarkable thing about this approach is that it purely operates on the level of functionals. Once the notion of solution is formulated into an inequality, there is no further need to consider an actual equation for the gradient flow.

One should note however that this is the procedure for an ideal gradient flow, where the evolution always goes precisely in the ``gradient''-direction of least energy and is thus less concerned with the behaviour of energy and dissipation in other directions. Clearly this is no longer true if the same evolution is perturbed using forces as is the case for our problem.

In fact a different consequence of this has been noted by the authors of \cite{mielkeExploringFamiliesEnergydissipation2020}; Since only behaviour in the ``gradient''-direction is important, one can show non-uniqueness of the limit energy-dissipation pair, by finding other solutions through artificially mixing effects of energy and dissipation in the other directions.

In order to thus obtain a more robust, canonical solution they argue roughly as follows: If the family of energies $\mathcal{E}_\delta$ is perturbed by adding a fixed, smooth enough energy $\mathcal{F}$, then this should perturb the limit problem only in the same way, i.e.\ while then $\mathcal{F}$ is added to the limit energy $\mathcal{E}_0$, the limit dissipation functional $\mathcal{R}_0$ should not change.

On the side of the energy, this does not change the required conditions as adding a regular enough, constant functional does not change the $\Gamma$-limit, but the same is not true for the total dissipation $\mathcal{D}_\delta$. When looking at \eqref{eq:gradientFlowEnergy}, if $\mathcal{E}_\delta$ is replaced by $\mathcal{E}_\delta + \mathcal{F}$, then the argument of $\mathcal{R}^*$ needs to be changed correspondingly as well. So instead of considering $\Gamma$-convergence of $\mathcal{D}_\delta$, we instead need to consider $\Gamma$-convergence of
\begin{align*}
 \mathcal{D}^{\mathcal{F}}_\delta(y) := \int_0^T \mathcal{R}_\delta(y,\partial_t y) + \mathcal{R}_\delta^*(y,-D\mathcal{E}_\delta(y)-D\mathcal{F}(y)) \mathrm{d}t
\end{align*}
for any $\mathcal{F} \in C^1(X)$, which is the definition of \emph{tilt-EDP convergence}.

However in our physical context, allowing for perturbation by a functional $\mathcal{F}$, corresponds only to pertubation by time-independent, conservative force. This is still not enough. Instead we need to allow for arbitrary forces $f_\delta$, which even can possibly even depend on $\delta$.

For this we first note that mirroring \eqref{eq:curveIneq} we have for some $f:[0,T] \to X^*$ smooth enough
\begin{align*}
 \mathcal{E}(y(0))-\mathcal{E}(y(T)) - \int_0^T \inner{f}{\partial y}dt = \int_0^T \inner{-D\mathcal{E}(y)-f}{\partial_t y} dt \leq \int_0^T \mathcal{R}(y,\partial_t y) + \mathcal{R}^*(y,-D\mathcal{E}(y)-f) dt.
\end{align*}
where equality holds again if and only if $D_2\mathcal{R}(y,\partial_t y) = - D\mathcal{E}(y)-f$. And again equality is precisely the case if
\begin{align*}
 \mathcal{E}(y(T)) + \int_0^T \mathcal{R}(y,\partial_t y) + \mathcal{R}^*(y,-D\mathcal{E}(y)-f) dt \leq \mathcal{E}(y(0)) - \int_0^T \inner{f}{\partial y}dt
\end{align*}
which also clearly identifies the extra term as the work done by the force $f$.

So a formal definition of EDP-convergence with arbitrary forces should look as follows:

\begin{definition}[Abstract EDP-convergence with forces] Let $X$ be a Banach space, families $\mathcal{E}_\delta : X \to \R$ and $\mathcal{R}_\delta: X \times X \to [0,\infty)$, convex in the last argument, given for any $\delta \leq 0$. We say the family $(\mathcal{E}_\delta,\mathcal{R}_\delta)_\delta$ \emph{EDP-converges with forces} to $(\mathcal{E}_0,\mathcal{R}_0)$ if
\begin{itemize}
 \item The family $(\mathcal{E}_\delta)_\delta$ $\Gamma$-converges to $\mathcal{E}_0$
 \item For any sufficiently regular sequence of time-dependent forces $(f_\delta)_\delta$ such that $f_\delta \to f_0$ we have that
\begin{align*}
 \mathcal{D}^{f_\delta}_\delta(y) := \int_0^T \mathcal{R}_\delta(y,\partial_t y) + \mathcal{R}_\delta^*(y,-D\mathcal{E}_\delta(y)-f_\delta) \mathrm{d}t
\end{align*}
$\Gamma$-converges to
\begin{align*}
 \mathcal{D}^{f_0}_0(y) := \int_0^T \mathcal{R}_0(y,\partial_t y) + \mathcal{R}_0^*(y,-D\mathcal{E}_0(y)-f_0) \mathrm{d}t
\end{align*}
\end{itemize}
\end{definition}

Of course in every actual application, the relevant convergences should be specified explicitly. Indeed we will do so for our problem in the next section. However by the previous considerations we can already not the following convergence ``principle'':

If $(\mathcal{E}_\delta,\mathcal{R}_\delta)_\delta$ EDP-converges with forces to $(\mathcal{E}_0,\mathcal{R}_0)$, the family of initial data $(y_\delta(0))_\delta$ is well-prepared (in the sense that $\mathcal{E}_\delta(y_\delta(0)) \to \mathcal{E}_0(y_0(0))$) and the actual forces $f_\delta$ converge to the limit $f_0$ in the matching, then any converging sequences of solutions $y_\delta$ to
\begin{align*}
 D\mathcal{E}_\delta(y_\delta) + D_2\mathcal{R}_\delta(y_\delta,\partial_t y_\delta) = f_\delta
\end{align*}
converges to a solution $y_0$ of
\begin{align*}
 D\mathcal{E}_0(y_0) + D_2\mathcal{R}_0(y_0,\partial_t y_0) = f_0.
\end{align*}

It should also be clear that such a principle will transfer to the type of ``time-delayed'' equation we studied.

\subsection{EDP-type convergence for linear elasticity}

The aim of this section is to now apply the previous abstract considerations to our specific problem. Same as in our previous consideration of $\Gamma$-convergence for the elastic energy, we need to now cope with the rescaling and the change of spaces. First of all we note the following:

\begin{remark}
 In our specific situation, we have $\mathcal{R}_\delta^*(y_\delta, \cdot): W^{-1,2}(\Omega;\R^d) \to \R$, yet generically $D\mathcal{E}_\delta(y_\delta) \in W^{-2,q}(\Omega;\R^d)$ only, which already makes $\mathcal{R}_\delta^*(y_\delta,-D\mathcal{E}_\delta(y_\delta)-f_\delta)$ seem ill-defined. However, first of all we are mainly interested in solutions, for which $-D\mathcal{E}_\delta(y_\delta) = D_2\mathcal{R}_\delta(y_\delta,\partial_t y_\delta) - f_\delta \in W^{-1,2}(\Omega;\R^d)$.

 Secondly and maybe more importantly, we can simply restrict $\mathcal{R}_\delta(y_\delta,\cdot)$ to $W^{2,q}(\Omega)\subset W^{1,2}(\Omega;\R^d)$ and take the dual of this. Then one quickly finds
 \begin{align*}
  \left(\mathcal{R}_\delta(y_\delta, \cdot)|_{W^{2,q}(\Omega;\R^d)}\right)^*(\xi) = \begin{cases}
  \mathcal{R}_\delta^*(y_\delta, \xi) &\text{ if } \xi \in W^{-1,2}(\Omega;\R^d)\\
  \infty & \text{ if } \xi \in W^{-2,q}(\Omega;\R^d) \setminus W^{-1,2}(\Omega;\R^d)
  \end{cases}
 \end{align*}
 Here the first line is a direct consequence of $W^{2,q}(\Omega)\subset W^{1,2}(\Omega;\R^d)$ being a dense embedding and the second line is due to the effect that $\mathcal{R}_\delta$ has no control over the second derivatives. On the side of the dual this then translates to the above indicator-function type behavior.

 We will abuse notation slightly and take the above as definition of $\mathcal{R}_\delta^*$.
\end{remark}

We then can deal with lower-semicontinuity of the dual dissipation-potential.

\begin{lemma}
 Let $(y_\delta)_\delta \subset W^{1,2}(\Omega;\R^d)$ such that $u_\delta := \delta^{-1}(y_\delta -\id)$ converges to $u$ in $ W^{1,2}(\Omega;\R^d)$ and $\norm[L^\infty(\Omega)]{\delta^{-1}\nabla u_\delta}$ is uniformly bounded.
 Let $\xi_\delta \rightharpoonup \xi$ in $W^{-1,2}$. Then
 $\liminf_{\delta \to 0} \mathcal{R}^*_\delta(y_\delta,\xi_\delta) \geq \mathcal{R}^*_0(u,\xi)$.
\end{lemma}

\begin{proof} Consider representatives $A_{\xi_\delta} \in L^2(\Omega;\R^{d\times d})$ of $\xi_\delta$ such that $\nabla y_\delta^\top A_{\xi_\delta}$ is symmetric. Using Proposition \ref{prop:dualDissipation} we have $\mathcal{R}^*_\delta(y_\delta,\xi_\delta) = \int_{\Omega} \nabla y_\delta A_{\xi_\delta}^\top \mathbb{D}(\nabla y_\delta)^{-1} \nabla y_\delta^\top A_{\xi_\delta} \mathrm{d} x$ which we can assume to be bounded. From the bounds on $\mathbb{D}(\nabla y_\delta)^{-1}$ we know that this also defines a family of norms for $\nabla y_\delta^\top A_\xi$ equivalent to the $L^2(\Omega;\R^{d\times d})$-norm. We can thus assume $\nabla y_\delta^\top A_{\xi_\delta}$ to converge weakly in $L^2$ to a symmetric function $A_0$.

 Furthermore, as $\nabla y_\delta$ and $\nabla y_\delta^{-1}$ converge strongly to $\Id$, we have
 \begin{align*}
  \inner{\xi_\delta}{\varphi} = \int_\Omega \nabla \varphi : A_{\xi_\delta}\mathrm{d} x = \int_\Omega ((\nabla y_\delta)^{-1}\nabla \varphi ) : \nabla y_\delta^\top A_{\xi_\delta}\mathrm{d} x \to \int_\Omega \nabla \varphi :A \mathrm{d} x
 \end{align*}
 from which we infer that $A$ is a representative of $\xi$.

 But then by the same calculation as in Lemma \ref{lem:dissipLSC}, now applied to $\mathbb{D}(\nabla y)^{-1}$, the result follows.
\end{proof}

 Note that this allows us to infer convergence of the De Giorgi-estimate \eqref{eq:discrEnergyIneq} when we send $\delta \to 0$.

\begin{proposition} Let $(f_\delta)_\delta \subset L^2(I\times\Omega)$ such that $\delta^{-1}f_\delta \to f$. Then
 the full dissipation $\mathcal{D}_\delta^{f_\delta}(y) := \int_0^T \mathcal{R}_\delta(y,\partial_t y) + \mathcal{R}_\delta^*(y,-D\mathcal{E}_\delta(y)-\delta^{-1}f_\delta) \mathrm{d}t$ $\Gamma$-converges to $\mathcal{D}_0(u)$ in the sense that
 \begin{enumerate}
  \item For any sequence $y_\delta \in L^\infty(I;W^{2,p}(\Omega;\R^d)) \cap W^{1,2}(I;W^{1,2}(\Omega;\R^d))$ such that $u_\delta := \delta^{-1}(y_\delta-\id) \rightharpoonup u$ in $W^{1,2}(I;W^{1,2}(\Omega;\R^d))$ and $\norm[L^\infty(\Omega)]{\nabla y_\delta}$ is uniformly bounded, we have
  \begin{align*}
   \liminf_{\delta \to 0} \mathcal{D}_\delta^{f_\delta}(y_\delta) \geq \mathcal{D}_0^{f}(u)
  \end{align*}
  \item For any $u \in W^{1,2}(I;W^{1,2}(\Omega;\R^d))$ there exists $y_\delta \in L^\infty(I;W^{2,p}(\Omega;\R^d)) \cap W^{1,2}(I;W^{1,2}(\Omega;\R^d))$ such that $u_\delta := \delta^{-1}(y_\delta-\id) \to u$ in $W^{1,2}(I;W^{1,2}(\Omega;\R^d))$ and $\norm[L^\infty(\Omega)]{\nabla y_\delta}$ is uniformly bounded, and we have
  \begin{align*}
   \limsup_{\delta \to 0} \mathcal{D}_\delta^{f_\delta}(y_\delta) \leq \mathcal{D}_0^f(u)
  \end{align*}
 \end{enumerate}
\end{proposition}

\begin{proof}
  The $\liminf$-inequality is a direct consequence of the preceding lemma and Lemma \ref{lem:dissipLSC}. For the recovery sequence, similar to before we can assume $u$ to be smooth by a density argument. Then we define $y_\delta := \id + \delta u$ and the convergence of the $\mathcal{R}_\delta$-term follows as in Lemma \ref{lem:dissipRecovery}.

  As $y_\delta$ is smooth, we have that
  \begin{align*}
   \tilde{A}_\delta := \frac{1}{\delta} \partial_{F} W(\nabla y_\delta) + \frac{1}{\delta^{\alpha p-1}} \mathcal{L}_P(\nabla^2 y_\delta) - F_\delta \in L^2(\Omega;\R^{d\times d})
  \end{align*}
  is a well defined representative of $-\delta D\mathcal{E}_\delta(y_\delta) - f_\delta$, where $F_\delta$ is similarly a symmetric representative of $\delta^{-1}f_\delta\in W^{-1,2}(\Omega;\R^d)$. Employing a similar reasoning to Lemma \ref{lem:linEnergy}, but testing with $\psi \in L^2(\Omega;\R^{d\times d})$ instead of $\nabla \phi$ and estimating the final product differently we have
  \begin{align*}
   \abs{\int_\Omega  (\delta^{-1}\partial_F W(\nabla y_{\delta}) - \mathbb{C}e(u)) \cdot \psi \mathrm{d}x } \leq \delta \norm[L^\infty(\Omega)]{\nabla u}^2 \norm[L^2(\Omega)]{\psi}
  \end{align*}
  which implies $\delta^{-1} \partial_F W(\nabla y_{\delta}) \to  \mathbb{C}e(u)$ in $L^2(\Omega;\R^{d\times d})$. As similarly $\frac{1}{\delta^{\alpha p}} \mathcal{L}_P(\nabla^2 y_\delta) \to 0$ we have that $\tilde{A}_\delta \to -\mathbb{C}e(u) - F$ in $L^2(\Omega;\R^{d\times d})$.

  From the construction of $\mathcal{R}_\delta^*$, we know that there exists a weakly divergence free $B_\delta \in L^2(\Omega;\R^{d\times d})$ such that $\nabla y_\delta^\top (\tilde{A}_\delta + B_\delta)$ is symmetric. Since $\nabla y_\delta$ is invertible in a uniform sense, the same argument as before also implies that after another subsequence $\delta^{-1} B_\delta \rightarrow B$ in $L^2(\Omega;\R^{d\times d})$. With this we have
  \begin{align*}
   0 &= \nabla y_\delta^\top (\tilde{A}_\delta+B_\delta) - (\tilde{A}_\delta + B_\delta)^\top \nabla y_\delta \\
   &= \tilde{A}_\delta - \tilde{A}_\delta^\top + B_\delta - B_\delta^\top + \delta \nabla u^\top (\tilde{A}_\delta + B_\delta) - \delta (\tilde{A}_\delta+B_\delta)^\top \nabla u
  \end{align*}
  where the first two terms cancel by symmetry. This implies
  \begin{align*}
   \norm[L^2(\Omega)]{\delta^{-1} (B_\delta-B_\delta^\top)} \leq 2 C \delta \norm[L^\infty(\Omega)]{\nabla u} \norm[L^2(\Omega)]{\delta^{-1}(\tilde{A}_\delta+B_\delta)}.
  \end{align*}
  So the limit $B$ has no antisymmetric part, which implies that $B=0$.  Proceeding precisely as in Lemma \ref{lem:dissipRecovery}, we then obtain the convergence of the $\mathcal{R}_\delta^*$-term in $\mathcal{D}_\delta^f$.
\end{proof}

Combining this with previous considerations we  then obtain the so called tilt-EDP convergence of the system, which can be summarized as a theorem.
\begin{theorem}[EDP convergence with forces]
 The sequence of functionals $(\mathcal{E}_\delta,\mathcal{R}_\delta)_\delta$ EDP converges with forces in the following sense:
 \begin{enumerate}
  \item The sequence $(\mathcal{E}_\delta)_\delta$ $\Gamma$-converges to $\mathcal{E}_0$ in the sense of Corollary \ref{cor:energyGamma}.
  \item For any family of forces $(f_\delta)_\delta \subset L^2([0,T]\times\Omega;\R^d)$ such that $\delta^{-1} f_\delta$ strongly converges in $L^2$, the sequence $(\mathcal{D}_\delta^{f_\delta})_\delta$ $\Gamma$-converges to $\mathcal{D}_0^f$.
 \end{enumerate}
\end{theorem}

Not only does this allow us to conclude convergence of the respective gradient flow problem, but it also allows us a different, parallel take at the convergence result for the time-delayed equation in Proposition \ref{prop:TDDelta}.

\begin{corollary}\label{cor:tilt-EDP} 
 Let $(y_\delta)_\delta$ be a sequence of solutions to (NL-TD) satisfying the energy inequality. Then there exists a subsequence of $u_\delta = \delta^{-1}(y_\delta -\id)$ which converges to a limit $u$ that solves (L-TD).
\end{corollary}

\begin{proof}
 For compactness we employ the uniform energy estimate from Lemma \ref{lem:TDBounds} to find a converging subsequence as before. Next we note that the Fenchel-Young inequality \eqref{eq:FenchelEquality} and (NL-TD) imply
 \begin{align*}
  \inner{D\mathcal{R}_\delta(y_\delta,\partial_t y_\delta)}{\partial_t y_\delta} &= \mathcal{R}(y_\delta,\partial_t y_\delta) + \mathcal{R}^*(y_\delta, D\mathcal{R}_\delta(y_\delta,\partial_t y_\delta)) \\&= \mathcal{R}(y_\delta,\partial_t y_\delta) + \mathcal{R}^*(y_\delta, -D\mathcal{E}_\delta(y_\delta) - f - \tfrac{\rho}{h} \partial_t y_\delta(\cdot -h) ).
 \end{align*}
 With this the energy inequality can be rewritten as
 \begin{align*}
  &\phantom{{}={}}\mathcal{E}_\delta(y_\delta(t_0)) +\fint_{t_0-h}^{t_0} \frac{\rho}{2} \norm[L^2(\Omega)]{\partial_t y}^2 \mathrm{d}t + \mathcal{D}_\delta^{\tilde{f}}(y_\delta) + \int_0^{t_0} \frac{\rho}{h} \norm[L^2(\Omega)]{\partial_t y - \partial_t y(t-h)}^2 dt \\
  &\leq \mathcal{E}_\delta(y_0) + \fint_{-h}^{0} \frac{\rho}{2} \norm[L^2(\Omega)]{\partial_t y}^2 \mathrm{d}t + \int_0^{t_0} f \cdot \partial_t y_\delta \mathrm{d}t
 \end{align*}
 for $\tilde{f} = f+ \tfrac{1}{h} \partial_t y_\delta(\cdot -h)$.

 By the lower semicontinuity of the left hand side, we can then infer that for $\delta \to 0$
 \begin{align} \label{eq:EDPlimitEnergy}
  &\mathcal{E}_0(u(t_0)) +\fint_{t_0-h}^{t_0} \frac{\rho}{2} \norm[L^2(\Omega)]{\partial_t u}^2 \mathrm{d}t+ \int_0^{t_0} \mathcal{R}_0(u,\partial_t u) + \mathcal{R}^*_0(u,-D\mathcal{E}_0(u) - f - \tfrac{1}{h} \partial_t u(\cdot -h) ) \mathrm{d}t \\
  &+ \int_0^{t_0} \frac{\rho}{h} \norm[L^2(\Omega)]{\partial_t u - \partial_t u(t-h)}^2 dt \leq \mathcal{E}_0(u_0)+ \fint_{-h}^{0} \frac{\rho}{2} \norm[L^2(\Omega)]{\partial_t u}^2 \mathrm{d}t+ \int_0^{t_0} \tilde f \cdot \partial_t u_\delta \mathrm{d}t. \nonumber
 \end{align}

 However, by the chain rule and the Fenchel-Young inequality \eqref{eq:FenchelEquality}, we obtain directly
 \begin{align*}
  & \int_0^{t_0} \mathcal{R}_0(u,\partial_t u) + \mathcal{R}^*_0(u,-D\mathcal{E}_0(u) - f - \tfrac{1}{h} \partial_t u(\cdot -h) ) \mathrm{d}t \geq \int_0^{t_0} \inner{-D\mathcal{E}_0(u) - f - \tfrac{\rho}{h} \partial_t u(\cdot -h) }{\partial_t u} \mathrm{d}t\\
  & \geq \mathcal{E}(u_0) - \mathcal{E}_0(u(t_0)) - \int_0^{t_0} \inner{f}{\partial_t u} - \inner{\partial_t u(\cdot -h)}{\partial_t u} \mathrm{d}t
 \end{align*}
 which after reordering the terms and completing the square for the last term is \eqref{eq:EDPlimitEnergy} with the opposite inequality. Thus the Fenchel-Young inequality has to be an equality which implies $-D\mathcal{E}_0(u) - f - \tfrac{1}{h} \partial_t u(\cdot -h) = D\mathcal{R}_0(u,\partial_t u)$, i.e.\ that $u$ is indeed a weak solution to (L-TD).
\end{proof}

\subsection{Remarks on a dynamic \texorpdfstring{$\Gamma$}{Gamma}-convergence}

We now want to apply a similar approach to the dynamic problem. Ultimately, we will rely on a simple inequality. For any $\delta > 0$ and any sufficiently regular function $y_\delta$ and $s \in [0,T]$, we have
\begin{align*}
 &\phantom{{}={}}\mathcal{E}_\delta(y_0)  + \int_\Omega \frac{\rho}{2} \abs{v_0}^2  - \left[ \mathcal{E}_\delta(y_\delta(s)) + \int_\Omega \frac{\rho}{2} \abs{\partial_t y_\delta (s)}^2 \right] = -\int_0^s \frac{\mathrm{d}}{\mathrm{d}t} \left[\mathcal{E}_\delta(y_\delta) + \int_\Omega \frac{\rho}{2} \abs{\partial_t y_\delta}^2 \right] \mathrm{d}t \\ &= \int_0^s \inner{-D\mathcal{E}_\delta(y_\delta) - \rho\partial_t^2 y_\delta-f_\delta}{\partial_t y_\delta} + \inner{f_\delta}{\partial_t y_\delta} \mathrm{d}t\\
 &\leq \int_0^s \mathcal{R}_\delta(y_\delta, \partial_t y_\delta) + \mathcal{R}_\delta^*(y_\delta , -D\mathcal{E}_\delta(y_\delta) - \rho\partial_t^2 y_\delta-f_\delta) + \inner{f_\delta}{\partial_t y_\delta}\mathrm{d}t
\end{align*}
where equality holds if and only if $D\mathcal{R}_\delta(y_\delta,\partial_t y_\delta) = -D\mathcal{E}_\delta(y_\delta) - \partial_t^2 y_\delta-f_\delta$, i.e.\ if $y$ is a solution to our problem.

But if in fact $y_\delta$ is a solution and we apply the same compactness arguments as before, then the left hand side is upper semicontinuous (as the energy is lower semicontinuous). So if the right hand side can be shown to be lower-semicontinuous, then we again have equality in the limit. By the same argument we then obtain $D\mathcal{R}_0(u_0,\partial_t u_0) = -D\mathcal{E}_0(u_0) - \partial_t^2 u_0$, i.e.\ that the limit is indeed a solution to the limit problem.

Following our previous considerations, we have seen that tilt-EDP convergence holds in our case. Only now, instead of having to study a given generalized force $\frac{1}{h}\partial_t y_\delta(\cdot - h) + f_\delta$, we will need to study $\rho\partial_t^2 y_\delta-f_\delta$, where the first term now depends on the solution itself. While from our previous considerations, we can immediately derive that the limit inequality in the above form is true, this involves using the limit equation, which, while still illustrative, somewhat removes the point of the construction. If instead we desire a direct proof, there are two obstacles:

The first is due to the specific regularity of our problem. While modifying the previous proof to allow for forces in $L^2(I;W^{-1,2}(\Omega;\R^d))$ is straightforward, using the approximative equation we can only show that $\partial_t^2 y_\delta \in L^2(I;W^{-2,p}(\Omega;\R^d))$. Only after adding $D\mathcal{E}_\delta(y_\delta)$ to it, do we end up in $L^2(I;W^{-1,2}(\Omega;\R^d))$ again, as the sum $-D\mathcal{E}_\delta(y_\delta) - \partial_t^2 y_\delta-f_\delta$ is equal to $D_2\mathcal{R}(y_\delta;\partial_t y_\delta) \in L^2(I;W^{-1,2}(\Omega;\R^d))$. However, for our purposes this is fine, as instead of showing full $\Gamma$-convergence of the dissipation, convergence along a sequence of solutions is sufficient.

The second problem is in identifying the limit. While $\delta^{-1}\partial_t^2 y_\delta$ is a linear term and thus each limit has to be equal to the distributional limit $\partial_t^2 u$, the same does not immediately follow for $D\mathcal{E}_\delta(y_\delta)$. However, due to the fact that the nonlinear terms vanish, we are able to proceed here as well.

We are thus able to show the following without using the limit equation.

\begin{proposition} \label{prop:dynamicEDP}
 Let $(y_\delta)_\delta \subset L^\infty(I;W^{2,p}(\Omega;\R^d)) \cap W^{1,2}(I;W^{1,2}(\Omega;\R^d))$ be a sequence of solutions to (NL-TD) for well-prepared forces and initial data. Then there exists a subsequence and a limit $u \in L^\infty(I;W^{1,2}(\Omega;\R^d)) \cap W^{1,2}(I;W^{1,2}(\Omega;\R^d)) \cap W^{2,2}(I;W^{-1,2}(\Omega;\R^d))$ with $u_\delta := \delta^{-1}(y_\delta-\id) \rightharpoonup u$ in that space, such that $u$ satisfies
 \begin{align} \label{eq:dynamicLimitEDP}
  \mathcal{E}_0(u(s)) + \frac{\rho}{2} \norm[L^2]{\partial_t u(s)}^2 + \int_0^s \mathcal{R}_0(u,\partial_t u) + \mathcal{R}_0^*(u,-D\mathcal{E}_0(u(s))-f-\partial_t^2 u)  +  \inner{f}{\partial_t u}\mathrm{d}t \leq \mathcal{E}_0(u_0) + \frac{\rho}{2} \norm[L^2]{v_0}^2
 \end{align}
 for all $s \in I$.
\end{proposition}

\begin{proof}
 As in Theorem \ref{thm:Delta}, we begin by using the bounds from the energy inequality for (NL) to extract a converging subsequence and a limit, such that 
 \begin{align*}
  u_\delta &\rightharpoonup u &\text{ in }& W^{1,2}(I;W^{1,2}(\Omega;\R^d)) \\
   y_\delta & \to \id &\text{ in }& L^2(I\times \Omega;\R^d).
 \end{align*}

 Using Lemma \ref{lem:linDissip} this implies in particular that $D\mathcal{R}_\delta(y_\delta,\partial_t y_\delta)$ is uniformly bounded in $L^2(I;W^{-1,2}(\Omega;\R^d))$. Since $y_\delta$ is a weak solution and $\delta^{-1} f_\delta$ is similarly bounded, this implies, after possibly another subsequence that
 \begin{align*}
  D\mathcal{E}_\delta(y_\delta) +\partial_t^2 y_\delta & \rightharpoonup \beta &\text{ in }& L^2(I;W^{-1,2}(\Omega;\R^d)).
 \end{align*}
 Additionally we know that in a distributional sense by linearity $\delta^{-1}\partial_t^2 y_\delta \rightharpoonup \partial_t^2 u$ and by Lemma \ref{lem:linEnergy} and the scaling of the second order term $D\mathcal{E}_\delta(y_\delta) \rightharpoonup D\mathcal{E}_0(u)$. This implies that $\beta = D\mathcal{E}_0(u) + \partial_t^2 u$.

 Next we rewrite the energy inequality, by noting that due to the Fenchel-Young inequality and the weak equation itself
 \begin{align*}
  \int_0^s \inner{D\mathcal{R}_\delta(y_\delta,\partial_t y_\delta)}{\partial_t y_\delta} \mathrm{d}s 
  &= \int_0^s \mathcal{R}_\delta(y_\delta,\partial_t y_\delta) + \mathcal{R}_\delta^*(y_\delta,D\mathcal{R}_\delta(y_\delta,\partial_t y_\delta)) \mathrm{d}s \\
  &= \int_0^s \mathcal{R}_\delta(y_\delta,\partial_t y_\delta) + \mathcal{R}_\delta^*(y_\delta,-D\mathcal{E}_\delta(y_\delta) - f_\delta- \partial_t^2 y_\delta)) \mathrm{d}s
 \end{align*}
 using the previously derived convergences, as well as the lower-semicontinuity of $\mathcal{E}$, $\mathcal{R}$ and $\mathcal{R}^*$ then implies \eqref{eq:dynamicLimitEDP}.
\end{proof}

Combining this with the considerations from the beginning of this section, we can then immediately conclude the existence of a limit solution in a purely variational way.
\begin{corollary} \label{cor:dynamicEDP}
 The limit $u$ of the previous proposition is a weak solution to (L).
\end{corollary}

\noindent
{\bf Acknowledgement.} This work was supported by the Czech Science Foundation  through the grant 23-04766S, the Ministry of Education, Youth and Sport of the Czech Republic through the grants LL2105 CONTACT (BB \& MKamp) and the projects Ferroic Multifunctionalities (FerrMion) [reg. no. CZ.02.01.01/00/22\_008/0004591] (BB \& MKamp) and Roboprox [reg. no. CZ.02.01.01/00/22 008/0004590] (MKr), within the Operational Programme Johannes Amos Comenius co-funded by
the European Union  as well as the Charles University through the  grant PRIMUS/24/SCI/020 and the Research Centre program No. UNCE/24/SCI/005 (MKamp).  

 \def\cprime{$'$} \def\ocirc#1{\ifmmode\setbox0=\hbox{$#1$}\dimen0=\ht0
  \advance\dimen0 by1pt\rlap{\hbox to\wd0{\hss\raise\dimen0
  \hbox{\hskip.2em$\scriptscriptstyle\circ$}\hss}}#1\else {\accent"17 #1}\fi}

\end{document}